\documentclass{article}
\usepackage[a4paper]{geometry}
\usepackage{graphicx} 
\usepackage{bbold}
\usepackage{lmodern}
\usepackage[T1]{fontenc}
\usepackage[utf8]{inputenc}  
\usepackage{amsmath,amssymb,amsthm} 
\usepackage{hyperref} 
\usepackage{color} 
\usepackage{inputenc}
\usepackage{url}
\usepackage{authblk}
\usepackage{soul}
\usepackage{xcolor}
\usepackage{float}
\usepackage{enumitem}
\usepackage{dsfont}

\newtheorem{theo}{Theorem}[section]
\newtheorem{lemma}{Lemma}[section]
\newtheorem{propo}{Proposition}[section]
\newtheorem{coro}{Corollary}[section]

\newtheorem{exx}{Example}
\newtheorem{defi}{Definition}[section]
\newtheorem{remark}{Remark}

\newtheorem{assumption}{Assumption}

\author[1]{Escobar-Bach, M.}
\author[2]{Popier, A.}
\author[1]{Sahin, M.}
\affil[1]{LAREMA, Université d'Angers}
\affil[2]{Laboratoire Manceau de Mathématiques, Le Mans Université}

\title{Ratio limit theorem for renewal processes}

\begin{document}

\maketitle

\begin{abstract}
  We consider a renewal process which models a cumulative shock model that fails when the accumulation of shocks up-crosses a certain threshold. The ratio limit properties of the probabilities of non-failure after $n$ cumulative shocks are studied. We establish that the ratio of survival probabilities converges to the probability that the renewal epoch equals zero. This limit holds for any renewal process, subject only to mild regularity conditions on the individual shock random variable. Precisions on the rates of convergence are provided depending on the support structure and the regularity of the distribution. Arguments are provided to highlight the coherence between this new results and the pre-existing results on the behavior of summands of i.i.d. real random variables.
\end{abstract}

\noindent
\textit{Some keywords:} Renewal process, Cumulative shock model, Ratio limit theorem, Large deviation theory.


\section{Introduction}

Renewal processes are stochastic models that generalize Poisson processes with exponential holding times to arbitrary distributions. This stochastic construction is useful for describing stochastic phenomena involving successive events. In this model, the inter-arrival times between events are assumed to be independent and identically distributed (\textit{i.i.d.}) non-negative random variables. This allows for a flexible description of complex systems, which is useful for risk analysis and follow-up studies. In reliability theory, this type of model is referred to as a cumulative shock model that describes systems that are subject to repeated random shocks over time. Each shock contributes incrementally to the system's degradation, so that failure occurs when the total accumulated damage shocks exceed a certain threshold.

In the literature, important developments in renewal theory date back to the proof of the strong renewal theorem in \cite{Gar62}, which was later generalized in \cite{Eri70,Fel71} with several extensions since. For instance, \cite{Gut83} treated the common case of independent sequences, later relaxed in \cite{Sha83,Sha85} with correlated random variables. Similar convergence results have been proposed in \cite{Gut90} in the context of two-dimensional stopped random walks. These understandings have also led to the development of substantial theories. Large deviation theory, in particular, is interested in estimating the probabilities of extremely unlikely deviations with empirical means. More recently, \cite{Bai15} has proposed a central limit result for local cumulative shock models with cluster structure. Additionally, questions concerning the evaluation of model reliability under random shocks have been addressed e.g. \cite{Mon15,Gon18}. Shock models exhibiting a specific dependence structure have also been the subject of detailed investigations. For instance, models with dependent inter-arrival times are studied in \cite{Ran19,Goy23}, and a cumulative shock model is proposed for time-dependent reliability analysis of deteriorating structures in \cite{Li19}. Few results also investigate the stochastic behaviors of crossing times for renewal processes at fixed critical threshold values. In recent works, \cite{Ome11,Mit14} formulate local limit theorems which establish the rates of convergence of crossing time probability mass functions. Within the same framework, \cite{Ome11} established a uniform convergence type theorem with an exponential rate of convergence for the probability mass functions, providing more precise bounds when $n$ approaches infinity.

In this paper, we focus on renewal processes with \textit{i.i.d.} non-negative random variables $(X_k)_{k \in \mathbb{N}^*}$ capturing the intensity of individual jumps, and a parameter $x>0$ that defines a fixed critical threshold beyond which the process fails. Formally, we define $S_n$ as the sum of the first $n$ variables $X_i$, and introduce the first instant at which $S_n$ jumps over $x$ as $\tau(x)$. Our probabilities of interest are given by $$c_{n,x}=\mathbb{P}[S_n \leq x]=\mathbb{P}[\tau(x) > n]$$ which describe the probability of non-failure after $n$ consecutive jumps. To properly determine the asymptotic behavior of these probabilities, we propose to consider the sequence $(\frac{c_{n+1,x}}{c_{n,x}})_{n\geq 1}$, which typically refers to the limit ratio properties of the process. The literature has widely addressed the study of ratio convergence properties, with many studies focusing on the analysis of Markov chains. The bulk of the analyses concerns probabilities of first return to a bounded domain, asymptotic frequencies of occupation as well as asymptotic behaviors of transition probabilities for Markov chains. In the second chapter of Freedman's seminal work \cite{Fre83}, numerous classical ratio limit theorems for Markov chains in the recurrent case are established. Random walks also provide a rich class of examples see e.g. \cite{KesI63,KesII63,Sto66}. Ratio limit theorems are also closely connected to Martin boundaries and to the long-time behaviors of Markov chains \cite{Sawyer1978,Sawyer1980,Woess2021}. Additionally, \cite{Den15} proposed similar conditions for ratio sequences that enable random walks to remain within a cone.
More recently, studies in \cite{Hof08} proposed a multi-state model on the integer lattice. It notably established a ratio limit theorem in order to obtain the asymptotic behavior for the size of the occupied cluster at the origin of the lattice. 

From a statistical perspective, ratio limit conditions were also employed in \cite{Sah25} to discuss model identifiability. In this study, a thorough discussion about model parametric classes of the random times is proposed according to whether the distributions are atomic or continuous. In particular, the assumptions on both distribution types are rather restrictive. They either require an isolated zero for the atomic support or impose conditions on the convolution of the density function in the continuous case.

\medskip 

Overall, the aforementioned studies were primarily interested in understanding the asymptotic properties of ratio sequences, as well as their behavior when approaching their limit. In line with these works, we propose to expand upon the findings from \cite{Sah25} and other sources to provide a broader analysis of the limits of ratio sequences for renewal processes. We obtain new results on the limiting behavior of ratio sequences for various classes of positive-valued random variables. One of the strengths of our study is that these different classes cover a wide range of discrete or absolutely continuous random variables. Moreover checking whether a given random variable belongs to one of these classes is relatively straightforward. As another novelty, we show that rates of convergence can be properly derived and depend on the regularities of the distributions of the individual jumps. In other words the behavior of the ratio sequences is explicitly derived from the ‘parameters’ of the class under consideration. This is particularly true for absolutely continuous random variables. Finally, to obtain our results, we also derive new properties on self-convolutions of functions with regular variations near zero.

\medskip

The remainder of the paper is given as follows. In Section \ref{sect:: main_results}, we introduce a cumulative shock model where the shock degradation are modeled with a renewal process and we establish the main limit ratio properties of the process. 
In Section \ref{sec:: LDT}, we discuss on the coherence between our ratio limit theorems and the framework of large deviations. All proofs and technical lemmas are postponed to Section \ref{sec:: main_proofs}.

\section{Asymptotic properties} \label{sect:: main_results}

Let $(\Omega,\mathcal{F},\mathbb{P})$ be a common probability space and $X = (X_n)_{n \in \mathbb{N}^*}$ be a sequence of \textit{i.i.d.} non-negative random variables, such that $X_1$ is not almost surely equal to 0. For any $n\geq1$ and $x>0$, we denote by $F$ the cumulative distribution function of $X_1$, $S_n=\sum_{k=1}^n X_k$ and recall that 
$$c_{n,x} = \mathbb{P} [ S_{n} \leq x ]$$
with $c_{0,x}= 1$. The nature and local behavior of the distribution $F$ around the origin play a central role in the asymptotic properties of the sequence $(c_{n,x})_{n \in \mathbb{N}}$. In the sequel, we choose to incorporate and discuss these characteristics among several classes of distributions. This allows us to consider a wide variety of renewal processes by distinguishing the different asymptotic behaviors of the ratio sequence $(\frac{c_{n+1,x}}{c_{n,x}})_{n\geq 1}$.
It is clear that distributions with support bounded away from zero have null probabilities when $n$ is large enough, yielding to trivial ratio sequence. We denote the latter set as the class $\mathcal{C}_1$ and focus primarily on distributions outside of it. We begin our analysis by dividing distributions into two categories: discrete and continuous, and then move on to more general configurations.\\ 

For any discrete distributions, we denote by $\{x_i, i \in J\}$ the set of strictly positive atoms where $J\subset\mathbb{N}$ and set $x_{\min}=\inf\{x_i, i \in J\}$. Given any threshold $s>0$ and for $ y = \inf \{x_i,\, x_i>s\}$, one can decompose the $n+1$-th jump probability by
 \begin{align} \nonumber
     \mathbb{P}& [S_{n+1} \leq x]\\ 
     \nonumber
     & = \mathbb{P} [S_{n} \leq x]  \mathbb{P} [ X_{1} = 0] + \sum_{i,\,x_i\leq s} \mathbb{P} [S_{n} \leq x - x_i] \mathbb{P} [ X_{1} = x_i]+ \sum_{\{i,\,x_i> s\}} \mathbb{P} [S_{n} \leq x - x_i] \mathbb{P} [ X_{1} = x_i] \\  \nonumber
     & \leq \mathbb{P} [S_{n} \leq x]  \mathbb{P} [ X_{1} = 0] + \sum_{i,\,x_i\leq s} \mathbb{P} [S_{n} \leq x - x_i] \mathbb{P} [ X_{1} = x_i]  +\mathbb{P} [S_{n} \leq x - y ]\\ \label{eq:discrete_case_spilting}
     & \leq \mathbb{P} [S_{n} \leq x]  \mathbb{P} [ X_{1} = 0] + \mathbb{P} [S_{n} \leq x ] (F(s) - F(0))  +\mathbb{P} [S_{n} \leq x - y ].
\end{align}  
so that
\begin{eqnarray*}
0\leq \dfrac{c_{n+1,x}}{c_{n,x}}-\mathbb{P}[X_1=0]\leq F(s) - F(0)  +\dfrac{c_{n,x-y}}{c_{n,x}}.
\end{eqnarray*}
This emphasizes that the ratio limit towards $\mathbb{P}(X_1=0)$ is closely related to the limiting behavior at 0 of $F$ and the arrangement of the atoms near the origin. As the number of iterations increases, the renewal processes can only make small jumps to stay below the threshold $x$. To account for this, we propose considering distributions with polynomial form at the origin. This approach can address a wide range of processes with discrete renewal epochs. We consider that a discrete random variable $X_1$ belongs to either one of the following classes: 
\begin{itemize}
    \item $\mathcal C_2$: $0 \in X_1(\Omega)$ and $ x_{\min} > 0$. 
    \item $\mathcal C_3$: $ x_{\min} = 0$ and $\exists \alpha > 0$ such that $F(t) - F(0) = \underset{t \xrightarrow{} 0^+}{ O } (t^{\alpha})$.
\end{itemize}
It is important to mention that the assumption regarding whether or not the origin is an isolated point in the set of atoms is significant. Note, however, that in $\mathcal{C}_3$, we do not specify whether the origin belongs to the set of atoms or not, though this assumption will be important in the proof section.

Although the arguments slightly differ in case of absolutely continuous distribution in the equation \eqref{eq:discrete_case_spilting}, we can impose similar behaviors for $F$ with regularly varying conditions on the density function $f$. We consider that an absolutely continuous variable $X_1$ belongs to the following class: 
\begin{itemize}
    \item $\mathcal C_4$: $F(x) >0$ for any $x>0$ and $f\in C^1((0,+\infty))$ with $f'$ monotone in the neighborhood of 0. There exists $\alpha > 0$ such that $f$ is regularly varying of index $\alpha -1$ at 0.
\end{itemize}
Regular variation is a standard notion. However for completeness, we propose to recall the definition. 
\begin{defi} \label{defi:regular_variation}
A positive measurable function $L:]0,a] \xrightarrow{} [0,+\infty[ $, with $a>0$ is said to be slowly varying at 0 if
\begin{eqnarray*}
    \lim_{x \xrightarrow{} 0^+} \frac{L(\lambda x)}{L(x)} = 1 \; ; \quad \quad \forall \lambda>0.
\end{eqnarray*}
A function $f$ defined on $]0,a]$ is said to be regularly varying at zero, with index of regular variation $\beta \in \mathbb{R}$, if there exists a slowly varying function $L$ such that $f(x) = x^{\beta} L(x), \forall x \in ]0,a]$.
\end{defi}

The class $\mathcal{C}_4$ contains a wide variety of distributions. Straightforward examples outside of this class usually have index of regular variation $\alpha-1$ equals to $-1$ or non-monotonic density in the neighborhood of zero (see e.g. the standard log-Cauchy law with $f(t) = (t \pi (1 + \log (t)^2))^{-1} $). Lastly, we consider the following class as a generalization of $\mathcal{C}_4$, taking into account the possibility of having strictly positive atoms.
\begin{itemize}
    \item $\mathcal C_5$: $X_1$ can be decomposed with an atomic part $X^{\text{disc}}_1$ and an absolutely continuous part $X^{\text{cont}}_1$ such that the density $f$ of $X^{\text{cont}}_1$ satisfies the same conditions as in the class $\mathcal C_4$ and the set of atoms $\{x_i, i \in J\}$ of $X_1$ is bounded away from zero.
\end{itemize}
Although our assumptions seem different between the continuous and discrete distributions, results from \cite[Proposition 1.5.8]{bing:gold:teug:89} show that functions $F$ in $\mathcal{C}_4$ and $\mathcal{C}_5$ fit the same regularity conditions than that of the class $\mathcal{C}_3$.

\subsection{Discrete random variables}
In this section, we propose to study the ratio limit properties for the discrete classes of distributions. We start by recalling a result from \cite[Theorem 4.6]{Sah25}.
\begin{propo}\label{prop:first_paper}
When $X_1 \in  \mathcal C_2$ and $x_{\min} \in X_1(\Omega)$, for any $x>0$, there exist $C_x > 0$ and $N_x \geq 1$ such that for any $n \geq N_x$
\begin{eqnarray*}
    \left| \frac{c_{n+1,x} }{ c_{n,x} } - \mathbb{P} [X_1 = 0] \right| \leq \frac{C_x}{n} .
\end{eqnarray*}
\end{propo}
\noindent
The latter proposition provides a rate of convergence for the ratio sequence, although there are still some differences with our settings. Additionally, it also appears that it requires a minimum element in the set of the strictly positive atoms. This condition is however not always available in $\mathcal{C}_2$ as illustrated by the following example. 
\begin{exx} \label{exx::exotic_var}
Let $X_1$ be a random variables such that $X_1(\Omega) = \{ 0 \} \cup  \{ 1 + \frac{1}{n} ; n \geq 1 \}\cup  \{ 2 - \frac{1}{n} ; n \geq 1 \}$ where $\mathbb{P} [X_1 = 0] = \frac{1}{2}$ and for any $n \geq 1$
\begin{align*}
    \mathbb{P} \left[ X_1 = 1 + \frac{1}{n} \right] = \mathbb{P} \left[ X_1 = 2 - \frac{1}{n} \right]=\frac{1}{2^{n+2}}.
\end{align*}
It turns out that $X_1$ is a random variables that belongs to $\mathcal{C}_2$ but zero is an isolated point of the support.
\end{exx}
\noindent
The following proposition extends the above result and considers random variables that belong to the class $\mathcal{C}_2$. It establishes the asymptotic equivalence of $\frac{c_{n+1,x}}{c_{n,x}} -  \mathbb{P} [ X_{1} = 0]$ and does not require a minimum of the set of strictly positive atoms. 
\begin{propo} \label{theo:: class_c2}
Let $X_1\in\mathcal{C}_2$. Define $M_x\in\mathbb{N}$ such that
\begin{align*}
    M_x = \left\{ \begin{array}{cl}
        x/x_{\min} -1 & \quad \text{ if $x/x_{\min} \in \mathbb{N}$ and $x_{\min} \notin X_1(\Omega)$},\\
        \lfloor x/x_{\min} \rfloor &  \quad \text{otherwise}.
    \end{array} \right.
\end{align*}
Then we have
\begin{eqnarray*}
    \dfrac{c_{n+1,x}}{c_{n,x}}  - \mathbb{P} [X_1 = 0] \underset{n \to + \infty}{\sim} \mathbb{P} [X_1 = 0]\dfrac{ M_x}{n}.
\end{eqnarray*}
\end{propo}
\noindent
Intuitively, the integer $M_x$ represents the maximum number of strictly positive jumps that the renewal process can make before crossing the threshold. Note in particular that in cases where $x< x_{\min}$, $M_x=0$ and the sequence $c_{n,x}$ trivially equals $\mathbb{P}(X_1=0)^n$ for all $n\in\mathbb{N}$. The ratio sequence provides accurate approximations of the probability at the origin. Additionally, we demonstrate that the mass probabilities of specific atoms before $x$ can be approximated.
\begin{propo} \label{theo:: class_2}
Let $X_1\in\mathcal{C}_2$ and $y_1\in X_1(\Omega)$ with $y_1\leq x$. If $y_2$ is the first atom following $y_1$ and $y_2-y_1 > x_{\min}$, then there exists $C_{x,y_1,y_2} > 0$ such that for $n$ large enough 
\begin{equation*}
    \left| \dfrac{c_{n+1,x} - \underset{x_i < y_1}{\sum} c_{n,x-x_i} \mathbb{P} [ X_1 = x_i] }{ c_{n,x-y_1} } - \mathbb{P} [ X_1 = y_1 ] \right| \leq \frac{C_{x,y_1,y_2}}{n} .
\end{equation*}
\end{propo}
\noindent
In particular, Proposition \ref{theo:: class_2} exhibits the link between the probabilities of small jumps and the mass probabilities. From a statistical perspective, this means that the distribution $F$ can be recursively estimated if the atoms are indexed by monotonic order and the gaps between them are large enough.\\

In the main result of this section, we present the limit of the ratio sequence when the distribution belongs to the class $\mathcal{C}_3$. We recall that in this case, zero represents an accumulated point of the support, meaning that the jumps of the renewal processes are as small as possible. This property introduces a degree of complexity to the analysis of the sequence $(c_{n,x})_{n\in\mathbb{N}}$, yet the limit of the ratio sequence remains unchanged.
\begin{theo} \label{theo:: class_3}
Let $X_1 \in  \mathcal C_3$, then for any $x>0$ we have
\begin{align*}
    \lim_{n \xrightarrow{} + \infty} \dfrac{c_{n+1,x} }{c_{n,x}} = \mathbb{P}[X_1=0].
\end{align*}
\end{theo}
\noindent
The challenge of attaining a comparable rate of convergence with $\mathcal{C}_3$ compared to $\mathcal{C}_1$ and $\mathcal{C}_2$ is attributed, once more, by the possibility of an infinite number of strictly positive jumps. This makes impossible the use of similar counting techniques as in the proof of Proposition \ref{theo:: class_c2}. However, under the regularity assumption of the function $F$, it is possible to regulate the sizes of the small jumps, thereby providing insight into the number of small jumps.

\subsection{Absolutely continuous random variables}

This section considers the analysis of the limits of ratio sequences with absolutely continuous distributions. In a similar manner, it is anticipated that solely the minor fluctuations will dictate the asymptotic characteristics of the ratio sequences. Nevertheless, conventional decomposition methodologies, such as those outlined in Equation (\ref{eq:discrete_case_spilting}), become unfeasible. Our analysis thus focuses on the density of the renewal processes and its behavior near the origin. In this perspective, we denote as $f^{*n}$ the density function of $S_n$. It is worth mentioning that the convexity of the distribution functions allows for the derivation of tighter convergence rates than those of Theorem \ref{theo:: class_3}.

As previously indicated, the convergence rates derived throughout this section for random variables belonging to the class $\mathcal{C}_4$ stem entirely from the convexity properties of the distribution functions of the variables $S_n$ in a neighborhood of zero. However, one should observe that the distribution function of a $\mathcal{C}_4$-class variable is not necessarily convex in such a neighborhood. As a counterexample, consider a random variable with a $C^1((0,+\infty))$ density $f$ satisfying $f(t) = t^{\alpha - 1}$ in a neighborhood of the origin, with $\alpha \in (0,1)$. The corresponding distribution function is therefore equal to the strictly concave function $F(t) = \alpha^{-1} t^{\alpha}$. Analytic properties on the sequence of self-convoluted density functions $(f^{*n})_{n \geq 1}$ can nevertheless be deduced for density function of a $\mathcal{C}_4$-class variable. The following proposition notably establishes that the operation of self-convolution increases the domain on which the function is non-decreasing.

\begin{propo}\label{propo:self_convolution}
When $f$ is the density function of a $\mathcal{C}_4$-class variable, for any $x>0$, there exists $N_x \geq 1$ such that for any $n \geq N_x$, $f^{*n}$ is non-decreasing on the domain $[0,x]$.
\end{propo}
Important remarks arise from this result. First, for any  i.i.d. random sequence $(X_n)_{n \geq 1}$ that lies in the class $\mathcal{C}_4$ and for any compact set $[0,x]$, there exist $N_x \geq 1$ such that for any $n \geq N_x$, the distribution function of $S_n$ is convex on the set $[0,x]$. Secondly, when $f$ is the density function of a $\mathcal{C}_4$-class variable, $f$ could be a decreasing function on $(0,+\infty)$ and its right limit at zero could also value the infinity. The convolution therefore regularizes the functions involved: after a sufficient number of self-convolutions, the resulting function becomes bounded at zero and, in addition, non-decreasing on an increasingly large interval.
Finally, as long as the density function exhibits sufficiently regular behavior in a neighborhood of the origin, the proposition applies. The function may, in particular, oscillate infinitely many times on an interval $[a,b]$, for some $0<a<b<\infty$. Such chaotic behavior disappears once the function has been convoluted sufficiently many times. \\
From the analytical and geometrical properties of the sequence of distribution functions $(F_{S_n})_{n \geq 1}$ established in Proposition \ref{propo:self_convolution} , more precise approximations can be derived for the sequence of renewal probability ratios. The following theorem establishes the convergence result for sequences of random variables that belong to $\mathcal{C}_4$.

\begin{theo} \label{theo:: class c_5}
Assume that $X_1 \in \mathcal{C}_4$. Then for all $ x > 0$, there exists $C_x > 0$, $N_x \in \mathbb{N}^*$ such that $\forall n \in \mathbb{N}$ such that $n \geq N_x$,
\begin{equation} \label{eq:first_ineq_thm_density}
   \sup_{t \in [0,x]} \left| \dfrac{c_{n+1,t} }{c_{n,t}} \right|  \leq C_x \mathbb{P} \left[ X_1 \leq 1/n \right].
\end{equation}
\end{theo}
\noindent
This result allows to lighten the relationship between the ratio limits and the small probabilities of $X_1$ with jumps of order $1/n$. Due to the regularly varying assumption of $f$ (see Equation \eqref{rem:limit_beha_slowly_var} in Section \ref{ssect:proof_class_C_5}), this affords upper bounds of polynomial order, given that $n$ is sufficiently large with 
$$\mathbb P \left[ X_1 \leq 1/n \right]\lesssim \dfrac{1}{n} f\left( \dfrac{1}{n}\right)  \lesssim \dfrac{1}{n^{\beta}}$$
for any $0 < \beta < \alpha$. Furthermore, a direct induction allows to deduce that
$$c_{n,x} \leq  \left( C_x \right)^{n-N_x} \prod_{k = N_x}^{n-1} \mathbb{P} \left[ X_1 \leq 1/k \right] \leq  ((N_x-1)!)^\beta \left( C_x C_\beta \right)^{n-N_x} \dfrac{1}{((n-1)!)^\beta} .$$
where $C_\beta$ is a positive constant. We finally extend the latter result to the class $\mathcal{C}_5$. The proof is derived from a similar analysis, given that the singular part of the support is isolated from zero. Consequently, this restricts the consideration to minor fluctuations from the absolutely continuous part of the distribution. The following proposition is thus derived.

\begin{propo} \label{prop:class_C_6}
Assume that $X_1 \in \mathcal{C}_5$. Then for all $x > 0$, there exists a constant $C_x > 0$, $N_x \in \mathbb N^*$ such that for all $n \geq N_x$
$$ \sup_{t \in [0,x]} \left| \dfrac{c_{n+1,t} }{c_{n,t}} \right| \leq C_x \mathbb{P} \left[ X^{\text{cont}} \leq 1/n \right].\\
$$
\end{propo}

In contrast to the discrete distributions, the tightness of the upper bound in Theorems \ref{theo:: class c_5} remains uncertain. Nevertheless, it is possible to obtain a lower bound on the likely order by dropping the uniform distance.
\begin{theo} \label{theo::lower_bound_continuous}
Assume that $X_1\in \mathcal{C}_4$. Then for any $x>0$, there exists $c_x >0$ such that for n sufficiently large,
\begin{align*}
    \dfrac{c_{n+1,x} }{c_{n,x}}  \geq c_x \mathbb{P} \left[ X_{1} \leq  1/n  \right].
\end{align*}

\end{theo}
\noindent
Together, Theorems \ref{theo:: class c_5} and \ref{theo::lower_bound_continuous} give the optimal order of the rate of convergence for the sequence  $(c_{n+1,x}/c_{n,x})_{n \in \mathbb{N}}$. It is noteworthy that the proofs of these theorems only merely demonstrate the existence of the constants $c_x$ and $C_x$ without providing explicit formulations. Nevertheless, it is possible to propose an asymptotic equivalent of the ratio sequence when $n$ increases.

\begin{theo} \label{theo:: equivalence_continuous_case}
Assume that $X_1\in\mathcal{C}_4$. Then for any $x>0$ and for any non-negative function $l$ that verifies $\displaystyle \lim_{t\to \infty} l(t) =  \infty$ and $l(t) \leq 2x \log(t)$ for any $t\in [a,\infty)$ for some $a\geq 0$, 
we have
\begin{align*}
    \dfrac{c_{n+1,x} }{c_{n,x}}  \underset{n \to + \infty}{\sim} \int_{0}^{\frac{l(n)}{n}} f(t) \left( 1 - \frac{t}{x} \right)^{\alpha n} dt.
\end{align*}
\end{theo}
\noindent
Several important remarks arise from the above result. First, we observe that the asymptotic equivalent is expressed exclusively via the parent density function $f$ and its regularly varying parameter $\alpha$. This allows for a convenient approximation of the ratio sequence from the model settings. It is not anticipated that a simpler expression can be obtained. Indeed, by assuming $f(t)=t^{\alpha-1}$ for $t$ small enough, one can show that the preceding integral yields $x^\alpha B (l(n)/(nx),\alpha,n\alpha+1)$ where $B(x,a,b)$ denotes the incomplete Beta function with parameters $(a,b)$. Note that it is possible to consider any function $l$ such that $\displaystyle l(t) \underset{t \to + \infty}{=} o ( \log (t) )$, provided $l$ tends to infinity when $t$ goes to infinity.  \\

Second, it turns out that it is impossible to consider a bounded function $l$. For any such function $l_0\leq M_0$ with $M_0>0$, the above conclusion would directly implies that for any function $l$ that satisfies the conditions of Theorem \ref{theo:: equivalence_continuous_case}, we have
\begin{align*}
    \int_{\frac{l_0(n)}{n}}^{\frac{l(n)}{n}} f(t) \left( 1 - \frac{t}{x} \right)^{\alpha n} dt \underset{n \to + \infty}{=} o \left( \frac{\mathbb{P} [S_{n+1} \leq x ] }{\mathbb{P} [S_n \leq x ] } \right) \underset{n \to + \infty}{=} o \left( \mathbb{P} \left[ X_1 \leq \frac{1}{n} \right] \right).
\end{align*}
However, for $n$ large enough we can show that
\begin{align*}
    \int_{\frac{l_0(n)}{n}}^{\frac{l(n)}{n}} f(t) \left( 1 - \frac{t}{x} \right)^{\alpha n} dt & \geq \int_{\frac{M_0}{n}}^{\frac{M_0 + 1}{n}} f(t) \left( 1 - \frac{t}{x} \right)^{\alpha n} dt \\
    & \geq \left( 1 - \frac{M_0 +1}{nx} \right)^{\alpha n} \int_{\frac{M_0}{n}}^{\frac{M_0 + 1}{n}} f(t)  dt \\
    & = \left( 1 - \frac{M_0 +1}{nx} \right)^{\alpha n} \left( \mathbb{P} \left[ X_1 \leq \frac{M_0 +1}{nx} \right] - \mathbb{P} \left[ X_1 \leq \frac{M_0 }{nx} \right] \right) \\ & \underset{n \to + \infty}{\sim} \exp \left( - \frac{\alpha ( M_0 +1)}{x} \right) \left[  \left( \frac{ M_0 +1}{x} \right)^\alpha - \left( \frac{ M_0}{x} \right)^\alpha \right] \mathbb{P} \left[ X_1 \leq \frac{1}{n} \right]
\end{align*}
where the asymptotic equivalence results from the regularity assumption of $f$, but contradicts the initial assertion. \\ 

In order to show the practical applicability of our result, we conclude this part with an example when $\alpha=1$ and $f$ is constant on a neighborhood of zero. Let us consider 
$f(t)=c$ for $t \in [0,\epsilon]$. 
Then we obtain for $l(n)/n\leq \epsilon$
\begin{eqnarray*}
\int_{0}^{\frac{l(n)}{n}} f(t) \left( 1 - \frac{t}{x} \right)^{n} dt&=&\dfrac{cx}{ n+1}\left(1-\left(1-\dfrac{l(n)}{nx}\right)^{\alpha n+1}\right)\underset{n \to + \infty}{\sim} \dfrac{cx}{n}=x\,\mathbb{P}[X_1\leq 1/n].
\end{eqnarray*}
Variable such as uniform distributions provide examples to constant density functions on a neighborhood of zero. However, one can notice that such random variables do not belong to the class $\mathcal{C}_4$ since the density function is not of class $C^1((0,+ \infty))$. It is nevertheless worth mentioning that density functions that are piece-wise continuous and derivable with bounded derivatives at the positive discontinuity points belong to the class $C^1((0,+ \infty))$ when sufficiently self-convoluted. In particular, one can prove that for such a density function $f$, $f*f$ is continuous on $(0,+ \infty)$ with piece-wise continuous derivative 
and $f^{*4}$ is of class $C^1((0,+ \infty))$. This property notably permits to extend the class $\mathcal{C}_4$ to a more general class of absolutely continuous random variable. Since this property results from similar arguments as in Lemma \ref{lem:: derivation_convo}, the proof is omitted.

\section{Ratio limits and large deviations} \label{sec:: LDT}

In this section, we explore the relationship between ratio limit theorems and the framework of large deviation theory. Cramér's theorem for large deviations provides asymptotic equivalents for the probability that empirical means will deviate from their expected values. However, these conclusions are insufficient for deriving similar ratio limit theorems. The probabilities $c_{n,x}$ are equivalent to $\mathbb{P} \left[ S_n/n \in \left[0, x/n \right] \right]$ which rather describes the probability for the empirical mean to remain in the neighborhood of zero, and not to a fixed subset. 
Therefore, the proposed ratio limit theorems serve as a complementary result to Cramér's theorem for decreasing subsets to the origin. In the sequel, we propose large deviations results on decreasing subsets for distributions supported on the real line and under specific conditions. In the final part, we compare and contrast the aforementioned results according to the frameworks of large deviations and ratio limit theorems, particularly for distributions in $\mathcal{C}_2$.\\  

\subsection{Large deviations inequalities} 
We first note that for any jump distribution with $- \infty \leq \mathbb{E} [X_1] \leq 0$, the central limit theorem yields
$$
    \lim_{n \xrightarrow{} + \infty } c_{n,x} = \lim_{n \xrightarrow{} + \infty } \mathbb{P} \left[ S_n/n \leq x/n  \right] = \begin{cases}   
    \frac{1}{2}, & \text{if} \; \mathbb{E} [X_1] = 0 , \\
    1, & \text{if} \; \mathbb{E} [X_1] < 0 
    \end{cases} 
$$
which trivially ensures that $\lim_{n\to +\infty}c_{n+1,x}/c_{n,x}=1$.
In this part, we thus focus our study to distributions with $\mathbb{P} [X_1 < 0] > 0$, $\mathbb{P} [X_1 > 0] > 0$ and $ 0 < \mathbb{E} [X_1] \leq +\infty$.  The results that follow are mostly grounded in Cramér’s findings as set forth in his pioneering article \cite{Cra38}. For further details on this work, as well as on the notations from the aforementioned article that we will reuse, we refer the reader to its translated version \cite{Cra22}. As such, we define $Y_n = \mathbb{E} [X_1] - X_n$ for any $n\geq1$ and denote $V$ their common distribution function. Additionally, we introduce the following model functions
$$R(h) = \int_{- \infty}^{+ \infty} e^{hy} dV(y) \quad\text{and}\quad \overline{m}(h) = \dfrac{R'(h)}{R(h)}$$
and suppose that
\begin{assumption} \label{ass:cramer_theo}
$\ $
\begin{enumerate}
    \item There exists a positive real number $a$ such that the moment generating function of $X_1$ is finite on the subset $]-a,a[$, i.e. $R(h)$ is well-defined for $|h| < a$. 
    \item $Y_1$ admits an absolutely continuous component.
    \item If $]-a_2, a_1[$ is the maximal convex subset wherein $R(h)$ converges, the next inequality holds
    \begin{equation} \label{eq:condition_C}
    \mathbb{E} [X_1] < C_1:=\lim_{h \uparrow a_1} \overline{m}(h).
    \end{equation}
\end{enumerate}
\end{assumption}
\noindent 
The first condition also induces the existence of a finite second order moment for the variables $Y_1$ and $X_1$. 
According to the same condition, it additionally appears that $a_1$ and $a_2$ must be positive numbers. Since $\overline{m}$ is an increasing function, as proven in the next theorem, $C_1$ and $C_2 := \lim_{h \downarrow - a_2} - \overline{m} (h)$ are well-defined and positive, possibly infinite. For any $h \in ]-a_2, a_1[$, we consider the \textit{i.i.d.} sequence of random variables $(\overline{Z}_n^h)_{n \in \mathbb{N}^*}$ with common distribution function $\overline{V}$ defined as
\begin{align*}
     \overline{V}(z) = \frac{1}{R(h)} \int_{- \infty}^z e^{hy} dV(y),\quad\forall z \in \mathbb R.
\end{align*}
Note that in this case $\overline{m}(h) = \mathbb{E} [\overline{Z}_n^h]$ and the variable $\overline{Z}_n^h$ admits a finite variance, here denoted $\overline{\sigma}^2 (h) = \mathbb{V} [ \overline{Z}_n^h ]$. The following result proposes asymptotic equivalents for the sequence $(c_{n,x})_{n\geq 1}$ from the same perspective than large deviations.
\begin{theo} \label{ratio_lim_ldp}
Let $x$ be any positive real number. Under Assumption \ref{ass:cramer_theo}, we have 
\begin{align*}
  c_{n,x} \underset{n \to + \infty}{\sim} \frac{e^{-\alpha n + xh_\infty}}{h_\infty \overline{\sigma}(h_\infty ) \sqrt{2 \pi n}}
\end{align*}
where $h_\infty = \arg \max \{ h \in ]-a_2, a_1 [ \mapsto h \mathbb{E} [X_1] - \log (R(h)) \}$ and $\alpha =  h_\infty \mathbb{E} [X_1] - \log (R(h_\infty))$.\\ 
In particular, we obtain that
\begin{align*}
    \lim_{n \to + \infty} \dfrac{c_{n+1,x}}{c_{n,x}} = e^{- \alpha}.
\end{align*}
\end{theo}

\begin{remark}
While this falls outside the scope of the present work, it is worth noting that the proof of the previous theorem can be extended to yield concentration results. Under Assumption \ref{ass:cramer_theo} and for any $c \in ]0,\mathbb{E} [X_1][$, we also have that
\begin{align*}
   \mathbb{P} \left[ S_n \leq cn + x  \right] \underset{n \to + \infty}{\sim} \frac{e^{-\alpha_c n + xh_{\infty,c}}}{h_{\infty,c} \overline{\sigma}(h_{\infty,c} ) \sqrt{2 \pi n}},
\end{align*}
where $h_{\infty,c} = \arg \max \{ h \in ]-a_2, a_1 [ \mapsto h c - \log (R(h)) \}$ and $\alpha_c = h_{\infty,c} c - \log (R(h_{\infty,c}))$.
\end{remark}

\paragraph{Discussion about \eqref{eq:condition_C} in Assumption \ref{ass:cramer_theo}:} 
the third condition of the aforementioned assumptions is essential to obtaining asymptotic equivalents. Although it may seem complicated at first, several noteworthy observations can be made. If one denote $R_{X_1}$ the moment generating function of $X_1$ and $\overline{m}_{X_1}=R'_{X_1}/R_{X_1}$ with
$$R_{X_1}(h) = \int_{- \infty}^{+ \infty} e^{hy} dF(y)$$
then we have for any $h \in ]-a_2,a_1[$
\begin{align*}
    \overline{m} (h) = \mathbb{E} [X_1] - \overline{m}_{X_1} (-h)
\end{align*}
so that \eqref{eq:condition_C} is equivalent to $\lim_{h \downarrow -a_1} \overline{m}_{X_1}(h)< 0$. In line with our initial settings, if $\mathbb{P} [X_1 < 0 ] = 0 $, straightforward algebra proves that  $\overline{m}_{X_1}$ is a non-negative function since for any $h \in ]-a_1,a_2[$
\begin{align*}
    \overline{m}_{X_1} (h) = \frac{1}{R_{X_1} (h) } \int_{- \infty}^{+ \infty} y e^{hy} dF(y) = \frac{1}{R_{X_1} (h) } \int_{0}^{+ \infty} y e^{hy} dF(y) \geq 0.
\end{align*}
This, in turn, demonstrates that the necessary assumptions for Theorem \ref{ratio_lim_ldp} are not met if $X_1$ is almost surely non-negative. Conversely, a number of cases persist that can be addressed independently. In the next lemma, we first demonstrate that the condition is met if $a_1=+\infty$.
\begin{lemma} \label{lem:Condition_C_true}
Assume that $\mathbb P(X_1 < 0) > 0$ and that $a_1 = + \infty$. Then 
$\lim_{h \downarrow -a_1} \overline{m}_{X_1}(h) <0$.
\end{lemma}
Similar analysis ensures that condition $(\ref{eq:condition_C})$ can still hold when $0 < a_1 < + \infty$. The function $R_{X_1}$ is a convex and log-convex function, so either $ \lim_{h \downarrow -a_1} R_{X_1} (h)$ is finite positive or infinite. Since $\overline{m}_{X_1} = \left(\log (R_{X_1})\right)'$, the function $\log R_{x_1}$ admits a continuous (even $\mathcal{C}^1$) extension at $-a_1$, and $\lim_{h \downarrow -a_1} \overline{m}_{X_1}(h) = - \infty$ if the limit of $R_{X_1}$ is finite. However, additional assumptions are required if the latter limit is infinite.  
\begin{exx}
Consider $X_1$ with the density function defined on $\mathbb{R}$ given by
\begin{align*}
    f_1 (x) = \frac{1}{2} e^{ -| x-1 |}\quad\text{or}\quad f_2 (x) & = \frac{c}{1 + (x-1)^4} e^{ -| x-1 |}
\end{align*}
where $c>0$ is a normalizing constant. In both cases, the random variable is supported on $\mathbb{R}$ with $\mathbb{E} [X_1] = 1$, $\mathbb{P} [X_1 \leq t] > 0$ for any $t<0$ and $a_1=-a_2=1$.\\

When the density is set to $f_1$, straightforward computations show that the functions $R_{X_1}$ and $-\overline{m}_{X_1}$ behave like $h\to (h-1)^{-1}$ near $-1$. Thus $\lim_{h \downarrow -a_1} R_{X_1} (h) = + \infty$ and Condition \eqref{eq:condition_C} is satisfied. If the density is given by $f_2$, 
$$ \lim_{h \downarrow -a_1} R_{X_1} (h) = c\int_{- \infty}^1 \frac{e^{-1}}{1 + (x-1)^4} dx +c\int_{1}^{+ \infty} \frac{1}{1 + (x-1)^4} e^{-2x+1} dx  < \infty$$ 
and $\overline{m}_{X_1}$ admits a finite limit at $-1$ with
\begin{align*}
   \overline{m}_{X_1}(-1) & = ce^{-1}\left(\int_{- \infty}^0 (x+1) \frac{1}{1 + x^4} dx + \int_{0}^{+ \infty} (x+1) \frac{e^{-2x}}{1 + x^4} dx\right).
\end{align*}
It then follows that $\int_{- \infty}^0  \frac{x+1}{1 + x^4} dx = \frac{\sqrt{2}-1}{4}\pi$ and $\int_{0}^{+ \infty} (x+1) \frac{e^{-2x}}{1 + x^4} dx \approx 0.57$ implying that $\overline{m}_{X_1}(-1) >0$ and Condition \eqref{eq:condition_C} is not satisfied.
\end{exx}

\subsection{Comparison with ratio limits} 
We finally discuss the discrepancy between the sharpness of the bounds for the ratio sequences when either the ratio limits or the large deviations methodologies are taken into consideration. We consider $X_1 \in \mathcal C_2$ and notice that in this case, the function $\Lambda : \lambda \mapsto \log \mathbb{E} \left[e^{\lambda X_1} \right]$ is well defined and finite on $]- \infty,0]$. Moreover
$ \Lambda^*(0) = - \lim_{\lambda \to - \infty} \Lambda ( \lambda) = - \log(\mathbb P(X_1=0))$.

\begin{lemma}[Using large deviation techniques] \label{lem:: large_dev_tech}
Assume $X_1\in\mathcal C_2$. For any $C>1$, we have for $n$ large enough
\begin{align*}
    \log \mathbb{P}[X_1=0] \leq \frac{1}{n} \log \mathbb{P} \left[ S_n \leq x  \right] &  
    \leq \log \mathbb{P}[X_1=0] +  C\frac{x}{x_{\min}} \frac{ \log n}{ n}.
\end{align*}
\end{lemma}

\begin{lemma}[Using the ratio limit results] \label{lem:: ratio_limit_tech}
Assume $X_1\in\mathcal C_2$. We have for $n$ large enough
$$
    \log \mathbb{P} [X_1 = 0 ] + \frac{ c_x}{ \mathbb{P} [X_1 = 0 ]} \frac{\log n}{n} \leq \frac{1}{n} \log \mathbb{P} [ S_n \leq x ]  \leq \log \mathbb{P} [X_1 = 0 ] + \frac{ C_x}{\mathbb{P} [X_1 = 0 ]} \frac{\log n}{n}
$$
where $C_x$ and $c_x$ are any arbitrary constants verifying $0< c_x < \mathbb{P} [X_1 = 0] M_x < C_x$, where $M_x$ is defined in Proposition \ref{theo:: class_c2}.
\end{lemma}
\noindent
It is evident that the two lemmas provide asymptotic bounds for $\left( \frac{1}{n} \log \mathbb{P} [ S_n \leq x ] \right)_{n \in \mathbb{N}^*}$.  In contrast to the approach with ratio limit theorems, the large deviation results do not allow the attainment of precise lower bounds for ratio sequences. Furthermore, the constants $C_x$ and $c_x$, via the definition of $M_x$, from Lemma \ref{lem:: ratio_limit_tech} provide further insights regarding the optimal constants in the inequalities. Finally, it should be noted that the aforementioned lower and upper limits do not permit the re-establishment of the ratio limit theorems that was proposed in Section  \ref{sect:: main_results}.

\bibliographystyle{plain}
\bibliography{bibli}

\begin{thebibliography}{10}

\bibitem{Bai15}
Jianming Bai, Yun Chen, Chun Yuan, and Xiaoling Yin.
\newblock Limit theorems for local cumulative shock models with cluster shock
  structure.
\newblock {\em Mathematical Problems in Engineering}, 2015:1--11, 03 2015.

\bibitem{bing:gold:teug:89}
Nicholas~H. Bingham, Charles.~M. Goldie, and Jozef~L. Teugels.
\newblock {\em Regular variation}, volume~27 of {\em Encyclopedia of
  Mathematics and its Applications}.
\newblock Cambridge University Press, Cambridge, 1989.

\bibitem{Cra70}
Harald Cram\'er.
\newblock {\em Random Variables and Probability Distributions}.
\newblock Cambridge Tracts in Mathematics. Cambridge University Press, 1970.

\bibitem{Cra38}
Harald Cramér.
\newblock Sur un nouveau théorème-limite de la théorie des probabilités.
\newblock {\em Colloque consacré à la théorie des probabilités}, 736:2--23,
  1938.

\bibitem{Cra22}
Harald Cramér and Hugo Touchette.
\newblock On a new limit theorem in probability theory (translation of '{S}ur
  un nouveau th\'eor\`eme-limite de la th\'eorie des probabilit\'es'), 2018.

\bibitem{Dem09}
Amir Dembo and Ofer Zeitouni.
\newblock {\em Large Deviations Techniques and Applications}.
\newblock Stochastic Modelling and Applied Probability. Springer, 2009.

\bibitem{Den15}
Denis Denisov and Vitali Wachtel.
\newblock Random walks in cones.
\newblock {\em The Annals of Probability}, 43(3):992--1044, 2015.

\bibitem{Eri70}
K.~Bruce Erickson.
\newblock Strong renewal theorems with infinite mean.
\newblock {\em Transactions of the American Mathematical Society},
  151(1):263--291, 1970.

\bibitem{Sah25}
Mikael Escobar-Bach, Alexandre Popier, and Malo Sahin.
\newblock A dependent and censored first hitting-time model with compound
  {P}oisson processes, 2025.

\bibitem{Fel71}
William Feller.
\newblock {\em An introduction to probability theory and its applications}.
\newblock Wiley, 1971.

\bibitem{Fre83}
David Freedman.
\newblock {\em Markov chain}.
\newblock Springer, 1 edition, 1983.

\bibitem{Gar62}
Adriano Garsia and John Lamperti.
\newblock A discrete renewal theorem with infinite mean.
\newblock {\em Commentarii Mathematici Helvetici}, 37:221–234, 1962.

\bibitem{Gon18}
Ming Gong, Min Xie, and Yaning Yang.
\newblock Reliability assessment of system under a generalized run shock model.
\newblock {\em Journal of Applied Probability}, 55:1249--1260, 12 2018.

\bibitem{Goy23}
Dheeraj Goyal, Nil~Kamal Hazra, and Maxim Finkelstein.
\newblock A general class of shock models with dependent inter-arrival times.
\newblock {\em TEST}, 32(3):1079--1105, Sep 2023.

\bibitem{Gut90}
Allan Gut.
\newblock Cumulative shock models.
\newblock {\em Advances in Applied Probability}, 22(2):504--507, 1990.

\bibitem{Gut83}
Allan Gut and Svante Janson.
\newblock The limiting behaviour of certain stopped sums and some applications.
\newblock {\em Scandinavian Journal of Statistics}, 10(4):281--292, 1983.

\bibitem{KesII63}
Harry Kesten.
\newblock Ratio theorems for random walks ii.
\newblock {\em Journal d'Analyse Math{\'e}matique}, 11(1):323--379, Dec 1963.

\bibitem{KesI63}
Harry Kesten and Frank Spitzer.
\newblock Ratio theorems for random walks i.
\newblock {\em Journal d'Analyse Math{\'e}matique}, 11(1):285--322, Dec 1963.

\bibitem{Li19}
Junxiang Li, Jianqiao Chen, and Zhiqiang Chen.
\newblock A new cumulative damage model for time-dependent reliability analysis
  of deteriorating structures.
\newblock {\em Proceedings of the Institution of Mechanical Engineers, Part O:
  Journal of Risk and Reliability}, 234:1748006X1988615, 11 2019.

\bibitem{Mit14}
Kosto~V. Mitov and Edward Omey.
\newblock {\em Renewal processes}.
\newblock Springer International Publishing, 2014.

\bibitem{Mon15}
Delia Montoro-Cazorla and Rafael Pérez-Ocón.
\newblock A reliability system under cumulative shocks governed by a bmap.
\newblock {\em Applied Mathematical Modelling}, 39(23):7620--7629, 2015.

\bibitem{Ome11}
Edward Omeya and Rein Vesilob.
\newblock Local limit theorems for shock models.
\newblock {\em Brazilian Journal of Probability and Statistics}, 30, No.
  2:221–247, 2016.

\bibitem{Ran19}
Somayeh~Hamed Ranjkesh, Ali~Zeinal Hamadani, and Safieh Mahmoodi.
\newblock A new cumulative shock model with damage and inter-arrival time
  dependency.
\newblock {\em Reliability Engineering \& System Safety}, 192:106047, 2019.
\newblock Complex Systems RAMS Optimization: Methods and Applications.

\bibitem{Sawyer1978}
Stanley Sawyer.
\newblock Martin boundaries and ratio limit theorems for markov chains.
\newblock {\em Annals of Probability}, 6:299--313, 1978.

\bibitem{Sawyer1980}
Stanley Sawyer.
\newblock Martin boundaries and ratio limit theorems for markov chains. ii.
\newblock {\em Annals of Probability}, 8:312--327, 1980.

\bibitem{Sha83}
J.~George Shanthikumar and Ushio Sumita.
\newblock General shock models associated with correlated renewal sequences.
\newblock {\em Journal of Applied Probability}, 20(3):600--614, 1983.

\bibitem{Sha85}
J.~George Shanthikumar and Ushio Sumita.
\newblock A class of correlated cumulative shock models.
\newblock {\em Advances in Applied Probability}, 17(2):347--366, 1985.

\bibitem{Sto66}
Charles Stone.
\newblock Ratio limit theorems for random walks on groups.
\newblock {\em Transactions of the American Mathematical Society},
  125(1):86--100, 1966.

\bibitem{Hof08}
Remco van~der Hofstad and Wouter Kager.
\newblock Pattern theorems, ratio limit theorems and gumbel maximal clusters
  for random fields.
\newblock {\em Journal of Statistical Physics}, 130(3):503--522, Feb 2008.

\bibitem{Woess2021}
Wolfgang Woess.
\newblock Ratio limits and {M}artin boundary.
\newblock {\em Documenta Mathematica}, 26:1501--1528, 01 2021.

\end{thebibliography}

\newpage

\section{Proofs of the results} \label{sec:: main_proofs}

\subsection{Proofs for discrete random variables}

In order to prove the different Theorems in the cases of discrete random variables, we introduce several common notations that will be useful along the different proofs. 
\paragraph{Notations}
For any random variable $X_1 \in \mathcal{C}_2$, the atom $0$ is an isolated point of the support and we consequently define $$x_{\min} = \inf \{  X_1( \Omega) \backslash \{ 0\} \}, $$ the infinimum of the positive atoms. $x_{\min}$ is necessarily a positive quantity. \\
Recall that for such random variable and for any $t>0$, the quantity $M_{\max,t}$ that denotes the maximal number of variables $(X_i,i=1,\ldots,n)$ that take positive values, when studying the event $ \{ S_n \leq t \}$ for some $n \in \mathbb{N}^*$.\\
More generally, for any random variable from the classes $\mathcal{C}_2$ and $\mathcal{C}_3$ and any $t>0$,
\begin{align*}
    \Diamond_t  =  \left\{ (x_{i_1}, \dots ,x_{i_k}) \in (X_1(\Omega)\backslash \{ 0 \})^k; \; k \in \mathbb{N}^* \; \text{and} \; \sum_{j=1}^k x_{i_j} \leq t \right\}.
\end{align*}
In the particular case where $X_1 \in \mathcal{C}_2$, the value of $k$ in $\Diamond_t$ can not exceed the value $M_{\max,t}$. \\
Finally, for any $k \in \mathbb{N}^*$ such that $(x_{i_1}, \dots ,x_{i_k}) \in \Diamond_t$, we define the event
\begin{align*}
    B_{n,t,(x_{i_1}, \dots ,x_{i_k})} & = \{ \exists j_1<\ldots<j_k \leq n; X_{j_1} = x_{i_1}, \ldots ,X_{j_k} = x_{i_k} \} \cap \{ \forall i \leq n;\forall h \leq k; i \neq j_h; X_i = 0 \}.
\end{align*}

\subsubsection*{Proof of Proposition \ref{theo:: class_c2}}
To prove the result, we will partition the event $\{ S_n \leq x \}$ on the number of variables in the sequence $(X_i, i=1,\ldots,n)$ that value positive atoms. 
Let $n \geq M_x$, we have
\begin{align*}
    \mathbb{P} [  S_{n} \leq x ] & = \mathbb{P} [ S_{n} \leq x , \exists 0 \leq k \leq M_x, \exists \{i_1,\ldots,i_{k} \} \subset \{1,\ldots,n\}, X_{i_j} > 0, \forall 1 \leq j \leq k  ] \\
    & = \sum_{k=0}^{M_x} \mathbb{P} [ S_{n} \leq x , \exists \{i_1,\ldots,i_{k} \} \subset \{1,\ldots,n\}, X_{i_j} > 0, \forall 1 \leq j \leq k  ] \\
    & = \sum_{k=0}^{M_x} \binom{n}{k} \mathbb{P} [X_1 = 0 ]^{n-k} \mathbb{P} [ S_{k} \leq x , X_{j} > 0, \forall 1 \leq j \leq k  ].
\end{align*}
Therefore
\begin{align*}
    \dfrac{\mathbb{P} [  S_{n+1} \leq x ]}{\mathbb{P} [  S_{n} \leq x ]} & = \dfrac{\sum_{k=0}^{M_x} \binom{n+1}{k} \mathbb{P} [X_1 = 0 ]^{n+1-k} \mathbb{P} [ S_{k} \leq x , X_{j} > 0, \forall 1 \leq j \leq k  ]}{\sum_{k=0}^{M_x} \binom{n}{k} \mathbb{P} [X_1 = 0 ]^{n-k} \mathbb{P} [ S_{k} \leq x , X_{j} > 0, \forall 1 \leq j \leq k  ]} \\
    & = \dfrac{\mathbb{P} [X_1 = 0 ]^{n+1-M_x}}{\mathbb{P} [X_1 = 0 ]^{n- M_x}} . \dfrac{\sum_{k=0}^{M_x} \binom{n+1}{k} \mathbb{P} [X_1 = 0 ]^{M_x-k} \mathbb{P} [ S_{k} \leq x , X_{j} > 0, \forall 1 \leq j \leq k  ]}{\sum_{k=0}^{M_x} \binom{n}{k} \mathbb{P} [X_1 = 0 ]^{M_x-k} \mathbb{P} [ S_{k} \leq x , X_{j} > 0, \forall 1 \leq j \leq k  ]} \\
    & = \mathbb{P} [X_1 = 0 ] \dfrac{\sum_{k=0}^{M_x} \binom{n+1}{k} \mathbb{P} [X_1 = 0 ]^{M_x-k} \mathbb{P} [ S_{k} \leq x , X_{j} > 0, \forall 1 \leq j \leq k  ]}{\sum_{k=0}^{M_x} \binom{n}{k} \mathbb{P} [X_1 = 0 ]^{M_x-k} \mathbb{P} [ S_{k} \leq x , X_{j} > 0, \forall 1 \leq j \leq k  ]}.
\end{align*}
Consequently
\begin{align} \nonumber
    &\dfrac{\mathbb{P} [  S_{n+1} \leq x ]}{\mathbb{P} [  S_{n} \leq x ]} - \mathbb{P} [X_1 = 0 ]\\ 
    \nonumber
    & = \mathbb{P} [X_1 = 0 ] \dfrac{\sum_{k=0}^{M_x} \left[ \binom{n+1}{k} - \binom{n}{k} \right] \mathbb{P} [X_1 = 0 ]^{M_x-k} \mathbb{P} [ S_{k} \leq x , X_{j} > 0, \forall 1 \leq j \leq k  ]}{\sum_{k=0}^{M_x} \binom{n}{k} \mathbb{P} [X_1 = 0 ]^{M_x-k} \mathbb{P} [ S_{k} \leq x , X_{j} > 0, \forall 1 \leq j \leq k  ]} \\ \label{eq:: proof_c2}
    & = \mathbb{P} [X_1 = 0 ] \dfrac{\sum_{k=1}^{M_x} \binom{n}{k-1}  \mathbb{P} [X_1 = 0 ]^{M_x-k} \mathbb{P} [ S_{k} \leq x , X_{j} > 0, \forall 1 \leq j \leq k  ]}{\sum_{k=0}^{M_x} \binom{n}{k} \mathbb{P} [X_1 = 0 ]^{M_x-k} \mathbb{P} [ S_{k} \leq x , X_{j} > 0, \forall 1 \leq j \leq k  ]}.
\end{align}
The numerator and the denominator of the fraction in (\ref{eq:: proof_c2}) can be expressed as polynomial function in n with respective degree equal to $M_x-1$ and $M_x$. The dominating terms are respectively the terms where $k=M_x$. We conclude that 
\begin{align*}
    \dfrac{\mathbb{P} [  S_{n+1} \leq x ]}{\mathbb{P} [  S_{n} \leq x ]} - \mathbb{P} [X_1 = 0 ] & \underset{n \to + \infty}{\sim} \mathbb{P} [X_1 = 0 ] \dfrac{ \binom{n}{M_x-1}  \mathbb{P} [ S_{M_x} \leq x , X_{j} > 0, \forall 1 \leq j \leq M_x  ]}{ \binom{n}{M_x}  \mathbb{P} [ S_{M_x} \leq x , X_{j} > 0, \forall 1 \leq j \leq M_x  ]} \\
    & \underset{n \to + \infty}{\sim} \mathbb{P} [X_1 = 0 ] \dfrac{M_x}{n+1 - M_x} \\
    & \underset{n \to + \infty}{\sim} \mathbb{P} [X_1 = 0 ] \dfrac{M_x}{n}.
\end{align*}

\subsubsection*{Proof of Proposition \ref{theo:: class_2}}
The proof of this result uses the same arguments as the proof of part 2 in \cite[Theorem 4.6]{Sah25}.
For any $n \geq M_x$ as defined in Proposition \ref{theo:: class_c2}, we have 
\begin{align*} \nonumber
    \mathbb{P} [ S_{n+1} \leq x] = &  \sum_{x_i \in X_1(\Omega) } \mathbb{P} [ S_{n+1} \leq x; X_{n+1} = x_i] \\ \nonumber
    = & \sum_{x_i \in X_1(\Omega)} \mathbb{P} [ S_{n} \leq x - x_i] \mathbb{P} [X_1 = x_i]  \\
    = &  \sum_{x_i \leq y_1} \mathbb{P} [ S_{n} \leq x - x_i] \mathbb{P} [X_1 = x_i]  + \sum_{ x_i \geq y_2} \mathbb{P} [ S_{n} \leq x - x_i] \mathbb{P} [X_1 = x_i].
\end{align*}
This implies that
\begin{align*}
 \mathbb{P} [ S_{n+1} \leq x] - \sum_{ x_i \leq y_1} \mathbb{P} [ S_{n} \leq x - x_i] \mathbb{P} [X_1 = x_i]&=\sum_{x_i \geq y_2} \mathbb{P} [ S_{n} \leq x - x_i] \mathbb{P} [X_1 = x_i]\\ 
 & \leq  \mathbb{P} [ S_{n} \leq x - y_2 ] \sum_{ x_i \geq y_2} \mathbb{P} [X_1 = x_i] \\
    &  \leq  \mathbb{P} [ S_{n} \leq x - y_2 ].
\end{align*}
We deduce that 
$$    \frac{\mathbb{P} [ S_{n+1} \leq x] - \sum_{x_i < y_1} \mathbb{P} [ S_{n} \leq x - x_i] \mathbb{P} [X_1 = x_i]}{\mathbb{P} [ S_n \leq x - y_1]} - \mathbb{P} [ X_1 = y_1 ] \leq \frac{ \mathbb{P} [ S_{n} \leq x - y_2] }{ \mathbb{P} [ S_n \leq x - y_1]}.$$
We know from Proposition \ref{theo:: class_c2} that 
\begin{align*}
    \dfrac{\mathbb{P} [  S_{n+1} \leq x - y_1 ]}{\mathbb{P} [  S_{n} \leq x - y_1 ]} - \mathbb{P} [X_1 = 0 ] & \underset{n \to + \infty}{\sim} \mathbb{P} [X_1 = 0 ] \dfrac{M_{x-y_1}}{n}.
\end{align*}
Moreover,
\begin{align*}
    \dfrac{\mathbb{P} [  S_{n+1} \leq x - y_1 ]}{\mathbb{P} [  S_{n} \leq x - y_1 ]} - \mathbb{P} [X_1 = 0 ] = \sum_{x_i \in X_1(\Omega) \backslash \{0\}} \frac{\mathbb{P} [  S_{n} \leq x - y_1 - x_i ]}{\mathbb{P} [  S_{n} \leq x - y_1 ]} \mathbb{P} [X_1 = x_i].
\end{align*}
From the initial assumptions on $y_1$ and $y_2$, we have $y_2 - y_1 > x_{\min}$. Due to the definition of $x_{\min}$, there necessarily exists an index $i_0$ such that $x_{i_0} \in X_1(\Omega) \backslash \{0\}$ and $x_{\min} \leq x_{i_0} < y_2 - y_1 $. We consequently deduce that 
\begin{align*}
    \dfrac{\mathbb{P} [  S_{n+1} \leq x - y_1 ]}{\mathbb{P} [  S_{n} \leq x - y_1 ]} - \mathbb{P} [X_1 = 0 ] \geq \frac{\mathbb{P} [  S_{n} \leq x - y_1 - x_{i_0} ]}{\mathbb{P} [  S_{n} \leq x - y_1 ]} \mathbb{P} [X_1 = x_{i_0}] \geq \mathbb{P} [X_1 = x_{i_0}] \frac{ \mathbb{P} [ S_{n} \leq x - y_2] }{ \mathbb{P} [ S_n \leq x - y_1]}.
\end{align*}
We finally obtain that for $n$ sufficiently large,
\begin{align*}
    \frac{ \mathbb{P} [ S_{n} \leq x - y_2] }{ \mathbb{P} [ S_n \leq x - y_1]} \leq  2 \dfrac{\mathbb{P} [X_1 = 0 ]}{\mathbb{P} [X_1 = x_{i_0}]} \dfrac{M_{x-y_1}}{n},
\end{align*}
which ends the proof.

\subsubsection*{Proof of Theorem \ref{theo:: class_3}}

\begin{proof}
Note that no assumption is made regarding whether 0 belongs to the support of the variable $X_1$. We will see during the proof that the case where 0 is an atom requires applying the result obtained when 0 is not an atom. Therefore, we first propose to prove the result in the specific case where 0 is not an atom of the random variable. \\
\underline{Assume that $\mathbb{P} [ X_1 = 0 ] = 0$.}
Since the distribution function $F$ of $X_1$ is right-continuous at zero, with $F(0)=0$, there exists $y > 0$ such that $F(y) < \frac{\varepsilon}{2}$. Define $s = \sup \left\{ y > 0, F(y) < \frac{\varepsilon}{2} \right\}$.
 In order to apply Inequality \eqref{eq:discrete_case_spilting}, we consider $J_1 = \{ x_i \in J, x_i < s  \}$ and $J_2 = J \backslash J_1$. When $J_2$ admits a minimal atom, it will be denoted as $x_{\min}$. In the contrary case, set $x_{\min} = s$. We deduce that
\begin{align*}
    \mathbb{P} [S_{n+1} \leq x] 
    & \leq \mathbb{P} [S_{n} \leq x - x_{\min} ] + \mathbb{P} [S_{n} \leq x]F(s) \\
    & \leq \mathbb{P} [S_{n} \leq x - x_{\min} ] + \mathbb{P} [S_{n} \leq x]\frac{\varepsilon}{2}. 
\end{align*}    
Hence
$$0 \leq \frac{\mathbb{P} [S_{n+1} \leq x] }{\mathbb{P} [S_{n} \leq x]}  \leq \frac{\mathbb{P} [S_{n} \leq x - x_{\min} ]}{\mathbb{P} [S_{n} \leq x]} + \frac{\varepsilon}{2}.$$
We can now prove that $ (\mathbb{P} [S_{n} \leq x - x_{\min} ]/\mathbb{P} [S_{n} \leq x])_{n \in \mathbb{N}^*}$ admits 0 as limit when $n$ goes to infinity. 

\medskip

To do so, we recall that $F(t) = \underset{t \xrightarrow{} 0^+}{ O } (t^{\alpha})$ for some $\alpha > 0$. Thus, define $k_\alpha = \lceil \frac{1}{\alpha} \rceil$. Then for any $n > k_\alpha$, we pose $n = q k_\alpha + r$, the Euclidean division of $n$ by $k_\alpha$. Now denote $S_{k_\alpha}^j = X_{j k_\alpha + 1 } + \ldots + X_{ (j+1) k_\alpha } $ for $j \in \{0, \ldots , q-1 \}$. We can re-write
\begin{align*}
    \mathbb{P} [S_{n} \leq x - x_{\min} ] & = \mathbb{P} \left[ S_{n} \leq x - x_{\min} ; \min \{ S_{k_\alpha}^0 ; S_{k_\alpha}^1 ; \ldots ; S_{k_\alpha}^{q-1} + X_{q k_\alpha + 1} + \ldots + X_n \} \leq \frac{x - x_{\min}}{ q } \right] \\
    & \leq \mathbb{P} \left[ S_{n} \leq x - x_{\min} ; \min \{ S_{k_\alpha}^0 ; S_{k_\alpha}^1 ; \ldots ; S_{k_\alpha}^{q-1} \} \leq \frac{x - x_{\min}}{ q } \right] \\
    & = \mathbb{P} \left[ \{ S_{n} \leq x - x_{\min} \} \cap \left( \bigcup_{i=0}^{q-1} \left\{ S_{k_\alpha}^i \leq \frac{x - x_{\min}}{ q } \right\}  \right) \right]. 
\end{align*}
We then use the i.i.d. property of the sequence $X_n$.
\begin{align*}
    \mathbb{P} [S_{n} \leq x - x_{\min} ] 
    & \leq \sum_{i = 0}^q \mathbb{P} \left[  S_{n} \leq x - x_{\min} ; S_{k_\alpha}^i \leq \frac{x - x_{\min}}{ q } \right] \\
    & \leq \sum_{i = 0}^q \mathbb{P} \left[  S_{n} - S_{k_\alpha}^i \leq x - x_{\min} ; S_{k_\alpha}^i \leq \frac{x - x_{\min}}{ q } \right] \\
    & \leq \sum_{i = 0}^q \mathbb{P} \left[  S_{n} - S_{k_\alpha}^i \leq x - x_{\min} \right] \mathbb{P} \left[ S_{k_\alpha}^i \leq \frac{x - x_{\min}}{ q } \right] \\
    & = q \mathbb{P} \left[  S_{n - k_\alpha} \leq x - x_{\min} \right] \mathbb{P} \left[ S_{k_\alpha} \leq \frac{x - x_{\min}}{ q } \right] \\
    & \leq q \mathbb{P} \left[  S_{n - k_\alpha} \leq x - x_{\min} \right] \mathbb{P} \left[ X_1 \leq \frac{x - x_{\min}}{ q } \right]^{k_\alpha}.
\end{align*}
Since $ (n-1)/k_\alpha \leq q  \leq n/k_\alpha $, we consequently obtain
\begin{align*}
    \mathbb{P} [S_{n} \leq x - x_{\min} ] & \leq \frac{n}{k_\alpha} \mathbb{P} \left[  S_{n - k_\alpha} \leq x - x_{\min} \right] \mathbb{P} \left[ X_1 \leq \frac{k_\alpha (x - x_{\min})}{ n-1 } \right]^{k_\alpha} \\
    & = \frac{n}{k_\alpha} \frac{\mathbb{P} \left[  S_{n - k_\alpha} \leq x - x_{\min} \right] \mathbb{P} [ S_{k_\alpha} \leq x_{\min} ] }{\mathbb{P} [ S_{k_\alpha} \leq x_{\min} ]} \mathbb{P} \left[ X_1 \leq \frac{k_\alpha (x - x_{\min})}{ n-1 } \right]^{k_\alpha} \\
    & \leq \frac{n}{k_\alpha} \frac{\mathbb{P} \left[  S_{n } \leq x \right]}{\mathbb{P} [ S_{k_\alpha} \leq x_{\min} ]} \mathbb{P} \left[ X_1 \leq \frac{k_\alpha (x - x_{\min})}{ n-1 } \right]^{k_\alpha} \\
    & = \frac{\mathbb{P} [S_{n} \leq x]}{ \mathbb{P} [ S_{k_\alpha} \leq x_{\min} ] } \frac{n}{k_\alpha} \frac{k_\alpha (x - x_{\min})}{ n-1 }  \frac{\mathbb{P} \left[ X_1 \leq \frac{k_\alpha (x - x_{\min})}{ n-1 } \right]^{k_\alpha}}{ \frac{k_\alpha (x - x_{\min})}{ n-1 } }.
\end{align*}
The second inequality is justified as follows.
\begin{align*}
    \mathbb{P} \left[  S_{n - k_\alpha} \leq x - x_{\min} \right] \mathbb{P} [ S_{k_\alpha} \leq x_{\min} ] = & \mathbb{P} \left[  S_{n - k_\alpha} \leq x - x_{\min} \right] \mathbb{P} [ S_n - S_{n - k_\alpha} \leq x_{\min} ] \\
    = & \mathbb{P} \left[  S_{n - k_\alpha} \leq x - x_{\min} ; S_n - S_{n - k_\alpha} \leq x_{\min} \right] \\
    \leq & \mathbb{P} \left[  S_{n } \leq x \right]   . 
\end{align*}
Finally, 
\begin{align*}
    \frac{\mathbb{P} [S_{n} \leq x - x_{\min} ]}{\mathbb{P} [S_{n} \leq x]} & \leq \frac{n}{n-1} \frac{x-x_{\min}}{ \mathbb{P} [ S_{k_\alpha} \leq x_{\min} ]} \frac{ F \left( \frac{k_\alpha (x - x_{\min})}{ n-1 } \right)^{k_\alpha}}{\frac{k_\alpha (x - x_{\min})}{ n-1 }},
\end{align*}
where $t \mapsto F(t)^{k_\alpha} = \underset{t \xrightarrow{} 0^+}{ O } (t^{\alpha k_\alpha})$, with $\alpha k_\alpha >1 $, which proves the claimed statement on the limit of $ \mathbb{P} [S_{n} \leq x - x_{\min} ]/\mathbb{P} [S_{n} \leq x] $. Thus for $n$ sufficiently large, $ \mathbb{P} [S_{n} \leq x - x_{\min} ]/\mathbb{P} [S_{n} \leq x] \leq \varepsilon/2$. Finally, for $n$ sufficiently large
\begin{align*}
    \frac{\mathbb{P} [S_{n+1} \leq x] }{\mathbb{P} [S_{n} \leq x]} \leq \varepsilon
\end{align*}
which achieves the proof of the result. \\
\underline{Now assume that $\mathbb{P} [ X_1 = 0 ] > 0$.}
The beginning of the proof in such configuration begins as in the previous case and we keep the same notations.  
Let $\varepsilon >0$. Define $s = \sup \left\{ y > 0, F(y) - F(0) < \frac{\varepsilon}{3} \right\}$, $J_1 = \{ x_i \in J, x_i < s  \}$, $J_2 = J \backslash J_1$ and $x_{\min}$ the minimal atom from the set $J_2$ when it exists and otherwise $x_{\min} = s$. Again with Inequality \eqref{eq:discrete_case_spilting}, we now obtain 
 \begin{align*}
     \mathbb{P} [S_{n+1} \leq x] 
     & \leq  \mathbb{P} [S_{n} \leq x ]  \mathbb{P} [ X_{1} = 0] + \mathbb{P} [S_{n} \leq x - x_{\min} ] + \mathbb{P} [S_{n} \leq x] \frac{\varepsilon}{3}. 
\end{align*}
Hence,
\begin{equation} \label{eq:estim_0_class_C4}
 0 \leq \frac{\mathbb{P} [S_{n+1} \leq x] }{\mathbb{P} [S_{n} \leq x]} - \mathbb{P} [ X_{1} = 0]  \leq \frac{\mathbb{P} [S_{n} \leq x - x_{\min} ]}{\mathbb{P} [S_{n} \leq x]} + \frac{\varepsilon}{3}.
\end{equation}
We can now prove that $ \mathbb{P} [S_{n} \leq x - x_{\min} ]/\mathbb{P} [S_{n} \leq x] $ admits 0 as limit when $n$ goes to infinity. \\
We define $(X_n')_{n \in \mathbb{N}^*}$ an i.i.d sequence of positive random variables such that ${\rm d} \mathbb{P}_{X_1'} = {\rm d} \mathbb{P}_{X_1 | X_1 > 0} $, i.e. for all $A \in \sigma ( X_1 \mathbb{1}_{X_1 > 0} )$, $ \mathbb{P} [ X_1' \in A ] = \mathbb{P} [ X_1 \in A ]/ \mathbb{P} [X_1 > 0 ] $. We denote by $(S_n')_{n \in \mathbb{N}^*}$ the associated partial sums of the sequence.  Then, $(X_n')_{n \in \mathbb{N}^*}$ belongs to Class $\mathcal C_3$ with $\mathbb{P} [X_1' = 0] = 0$. In the first part of the proof, it has been proven that $ \mathbb{P} [S_{n}' \leq x - x_{\min} ]/\mathbb{P} [S_{n}' \leq x] $ converges to zero. Consequently, there exists $N \in \mathbb{N}^*$ such that $\forall n \geq N$, $ \mathbb{P} [S_{n}' \leq x - x_{\min} ]/\mathbb{P} [S_{n}' \leq x] \leq \frac{\varepsilon}{3}$.\\ 
Thus, for $n > N$, by categorizing on the number of terms of $S_n$ equal to zero, we obtain
\begin{align} \nonumber
     & \mathbb{P} [S_{n} \leq x - x_{\min} ] \\ \nonumber
     & \quad = \sum_{k= 0}^n \mathbb{P} [S_{n} \leq x - x_{\min} ; \exists i_1 < \ldots < i_k, X_{i_j} = 0 \; \forall j \in \{ 1 , \ldots , k \}; X_\ell > 0 \; \forall \ell \notin \{ i_1 , \ldots , i_k \}  ] \\ \nonumber
     & \quad =  \sum_{k= 0}^{n-N} \mathbb{P} [S_{n} \leq x - x_{\min} ; \exists i_1 < \ldots < i_k, X_{i_j} = 0 \; \forall 1\leq  j \leq  k ; X_\ell > 0 \; \forall \ell \neq i_j]
     \\ \label{eq:estim_1_C_4}
     & \quad + \sum_{k= n-N}^n \mathbb{P} [S_{n} \leq x - x_{\min} ; \exists i_1 < \ldots < i_k, X_{i_j} = 0 \; \forall 1\leq j \leq k; X_\ell > 0 \; \forall \ell \neq i_j].
\end{align}
The first term can be properly dominated: 
\begin{align} \nonumber
    & \sum_{k= 0}^{n-N} \mathbb{P} [S_{n} \leq x - x_{\min} ; \exists i_1 < \ldots < i_k, X_{i_j} = 0 \; \forall j \in \{ 1 , \ldots , k \}; X_\ell > 0 \; \forall \ell \notin \{ i_1 , \ldots , i_k \}  ] \\ \nonumber
     & \quad = \sum_{k= 0}^{n-N} \sum_{i_1 < \ldots < i_k} \mathbb{P} [S_{n} \leq x - x_{\min}; X_{i_1} = \ldots = X_{i_k} = 0 ; X_\ell > 0 \; \forall \ell \notin \{ i_1 , \ldots , i_k \} ] \\ \nonumber
     & \quad =  \sum_{k= 0}^{n-N} \sum_{i_1 < \ldots < i_k} \mathbb{P} [ X_1 = 0 ]^k \mathbb{P} [S_{n-k} \leq x - x_{\min} ; X_\ell = X_\ell' \; \forall 1 \leq \ell \leq n-k \} ] \\ \nonumber
     & \quad =  \sum_{k= 0}^{n-N} \sum_{i_1 < \ldots < i_k} \mathbb{P} [ X_1 = 0 ]^k \mathbb{P} [ X_1 > 0 ]^{n-k} \mathbb{P} [S_{n-k}' \leq x - x_{\min} ] \\ \nonumber
     & \quad =  \sum_{k= 0}^{n-N} \sum_{i_1 < \ldots < i_k} \mathbb{P} [ X_1 = 0 ]^k \mathbb{P} [ X_1 > 0 ]^{n-k} \mathbb{P} [S_{n-k}' \leq x  ] \frac{\mathbb{P} [S_{n-k}' \leq x - x_{\min} ]}{ \mathbb{P} [S_{n-k}' \leq x  ] } \\ \nonumber
     & \quad \leq  \frac{\varepsilon}{3}  \sum_{k= 0}^{n-N} \sum_{i_1 < \ldots < i_k} \mathbb{P} [ X_1 = 0 ]^k \mathbb{P} [ X_1 > 0 ]^{n-k} \mathbb{P} [S_{n-k}' \leq x  ] \\ \label{eq:estim_2_class_C_4}
     & \quad\leq  \frac{\varepsilon}{3}  \sum_{k= 0}^{n} \sum_{i_1 < \ldots < i_k} \mathbb{P} [ X_1 = 0 ]^k \mathbb{P} [ X_1 > 0 ]^{n-k} \mathbb{P} [S_{n-k}' \leq x  ] .
\end{align}
But we also have that
\begin{align*}
    \mathbb{P} [S_{n} \leq x ] = & \sum_{k= 0}^n \mathbb{P} [S_{n} \leq x  ; \exists i_1 < \ldots < i_k, X_{i_j} = 0 \; \forall j \in \{ 1 , \ldots , k \}; X_\ell > 0 \; \forall \ell \notin \{ i_1 , \ldots , i_k \}  ] \\ 
    = & \sum_{k= 0}^{n} \sum_{i_1 < \ldots < i_k} \mathbb{P} [S_{n} \leq x ; X_{i_1} = \ldots = X_{i_k} = 0 ; X_\ell > 0 \; \forall \ell \notin \{ i_1 , \ldots , i_k \} ] \\
    = & \sum_{k= 0}^{n} \sum_{i_1 < \ldots < i_k} \mathbb{P} [ X_1 = 0 ]^k \mathbb{P} [S_{n-k} \leq x ; X_\ell = X_\ell' \; \forall 1 \leq \ell\leq n-k \} ] \\
    = & \sum_{k= 0}^{n} \sum_{i_1 < \ldots < i_k} \mathbb{P} [ X_1 = 0 ]^k \mathbb{P} [ X_1 > 0 ]^{n-k} \mathbb{P} [S_{n-k}' \leq x ]. 
\end{align*}
And combining this equality with the estimate \eqref{eq:estim_2_class_C_4} justifies that 
\begin{align} \nonumber
    & \sum_{k= 0}^{n-N} \mathbb{P} [S_{n} \leq x - x_{\min} ; \exists i_1 < \ldots < i_k, X_{i_j} = 0 \; \forall j \in \{ 1 , \ldots , k \}; X_\ell > 0 \; \forall \ell \notin \{ i_1 , \ldots , i_k \}  ] \\ \label{eq:estim_5_class_C_4}
    & \quad \leq  \frac{\varepsilon}{3} \mathbb{P} [S_{n} \leq x ].
\end{align}
Recall that from the notations introduced at the beginning of the section, for any positive $t$, the family of set $\{  B_{n, t, (x_{i_1}, \dots ,x_{i_k})} ,  (x_{i_1}, \dots ,x_{i_k}) \in \Diamond_t, k \in \{ 0 , \ldots , n \}  \}$ constitute a partition of the event $\{ S_n \leq t \}$. \\
Then, the second term can be re-written as
\begin{align*}
    & \sum_{k= n-N}^n \mathbb{P} \Bigg[ S_{n} \leq x - x_{\min} \; ; \; \exists\, i_1 < \ldots < i_k,\; X_{i_j} = 0 \; \forall j \in \{1 , \ldots , k\},\; X_\ell > 0 \; \forall \ell \notin \{i_1, \ldots, i_k\} \Bigg] \\
  & \quad =  \sum_{k=0}^{N} \sum_{(x_{i_1}, \ldots, x_{i_k}) \in \Diamond_{x - x_{\min}}} 
    \mathbb{P} \Big[ B_{n, x - x_{\min}, (x_{i_1}, \dots ,x_{i_k})} \Big].
\end{align*}
Since $N$ is a fixed integer such that $N <n$, using the same proof as in \cite[Theorem 4.6]{Sah25} (we omit the details of the proof here), we obtain the following estimate: for any fixed atom of $X_1$ $y \in ]0,x_{\min} [$, $k \leq N$, $ (x_{i_1}, \ldots, x_{i_k}) \in \Diamond_{x - x_{\min}}$,
\begin{align*}
    \frac{\mathbb{P} \left[ B_{n,x - x_{\min},(x_{i_1}, \ldots ,x_{i_k})} \right]}{\mathbb{P} \left[ B_{n,x-x_{\min} + y,(x_{i_1}, \dots ,x_{i_k}, y)} \right]} 
    & = \frac{k}{n-k} \cdot \frac{\mathbb{P} [X_1 = 0]}{\mathbb{P} [X_1 = y]} 
    \leq \frac{ N }{n- N} \cdot \frac{\mathbb{P} [X_1 = 0]}{\mathbb{P} [X_1 = y]}.
\end{align*}
We finally obtain
\begin{align*}
    & \sum_{k= n-N}^n \mathbb{P} \Bigg[ S_{n} \leq x - x_{\min} \; ; \; \exists\, i_1 < \ldots < i_k,\; X_{i_j} = 0 \; \forall j \in \{1 , \ldots , k\},\; X_\ell > 0 \; \forall \ell \notin \{i_1, \ldots, i_k\} \Bigg] \\
    & \leq \frac{ N }{n- N} \cdot \frac{\mathbb{P} [X_1 = 0]}{\mathbb{P} [X_1 = y]} 
    \sum_{k=0}^{N} \sum_{(x_{i_1}, \ldots, x_{i_k}) \in \Diamond_{x - x_{\min}}} 
    \mathbb{P} \left[ B_{n,x-x_{\min} + y,(x_{i_1}, \dots ,x_{i_k}, y)} \right] \\
    & \leq \frac{ N }{n- N} \cdot \frac{\mathbb{P} [X_1 = 0]}{\mathbb{P} [X_1 = y]} \cdot \mathbb{P} [ S_n \leq x],
\end{align*}
where the last inequality holds since all the terms in the obtained sum are disjoint sub-events of the event $\{ S_n \leq x \}$. Finally, for $n$ sufficiently large, 
$$\frac{ N }{n- N} \cdot \frac{\mathbb{P} [X_1 = 0]}{\mathbb{P} [X_1 = y]} \cdot \mathbb{P} [ S_n \leq x] \leq \frac{\varepsilon}{3} \cdot \mathbb{P} [ S_n \leq x].$$
Coming back to \eqref{eq:estim_0_class_C4}, with \eqref{eq:estim_1_C_4} and \eqref{eq:estim_5_class_C_4}, we obtain
\begin{align*}
0 \leq \frac{\mathbb{P} [S_{n+1} \leq x] }{\mathbb{P} [S_{n} \leq x]} - \mathbb{P} [ X_{1} = 0] 
    & \leq   \frac{\varepsilon}{3} + \frac{\varepsilon}{3} + \frac{\varepsilon}{3} = \varepsilon.
\end{align*}
This achieves the proof of the result.
\end{proof}

\subsection{Proofs for absolutely continuous variables}  \label{ssect:proof_class_C_5}

In this section, we prove Theorems \ref{theo:: class c_5}, \ref{theo::lower_bound_continuous} and \ref{theo:: equivalence_continuous_case} and Propositions \ref{propo:self_convolution} and \ref{prop:class_C_6}.  

In \cite[Theorem 4.6]{Sah25} a more restrictive theorem establishes a similar result of convergence when $X_1$ is absolutely continuous with respect to the Lebesgue measure, with density function $f$ such that for some $N$, $f^{*N}$ ($N$ times self-convolution of $f$ or density of $S_N$) is a non-decreasing function on $[0,x]$. Along this section, we prove that sequences in $\mathcal{C}_4$ verify such a property (Proposition \ref{propo:self_convolution}).

Regular variation has been evoked in Definition \ref{defi:regular_variation}. Note that if $L$ is slowly varying, then for any $\beta>0$,
\begin{equation}\label{rem:limit_beha_slowly_var}
     \lim_{x \xrightarrow{} 0^+} x^{\beta} L(x) = 0 \quad \quad \text{and} \quad \quad \lim_{x \xrightarrow{} 0^+} x^{- \beta} L(x) = + \infty .
\end{equation}

\begin{propo} \label{prop:slowly_varying}
Let $f$ be a regularly varying at zero, with index of regular variation $\beta \in \mathbb{R}$. Assume that $f$ is of class $C^1$ in a neighborhood of zero, such that $f$ is non-null and $f'$ is monotone in such neighborhood. We have
\begin{eqnarray*}
\lim_{x \xrightarrow{} 0+} \frac{xf'(x)}{f(x)} = \beta .
\end{eqnarray*}
\end{propo}
\begin{proof}
When $\beta \neq 0$, $f$ is not a slowly varying function and the monotone density theorem justifies the theorem (see. \cite[Proposition 1.5.8]{bing:gold:teug:89}). \\
When $\beta =0$, $f$ is a slowly varying function.
Due to the assertions made on $f$, there exists $\varepsilon > 0$ such that $f$ is non-null, of class $C^1$ and $f'$ is monotone on $]0,\varepsilon]$. Let $x \in ]0,\frac{3}{4} \varepsilon[$. 
Since $f'$ is monotone on $]0,\varepsilon[$,  then $f$ is either concave or convex. The concavity/convexity inequalities ensure that
\begin{align*}
    \frac{f(x) - f(x - x/4)}{x/4} & \leq f'(x) \leq \frac{f(x + x/4) - f(x)}{x/4} \quad \quad \quad \text{if $f$ is convex}, \\
    \frac{f(x + x/4) - f(x)}{x/4} & \leq f'(x) \leq \frac{f(x) - f(x-x/4)}{x/4} \quad \quad \quad \text{if $f$ is concave}.
\end{align*}
Consequently,
\begin{align*}
    4\frac{f(x)}{f(x)} - 4\frac{f((3/4)x)}{f(x)} & \leq \frac{xf'(x)}{f(x)} \leq 4\frac{f((5/4)x)}{f(x)} - 4\frac{f(x)}{f(x)} \quad \quad \quad \text{if $f$ is convex}, \\
     4\frac{f((5/4)x)}{f(x)} - 4\frac{f(x)}{f(x)}& \leq \frac{xf'(x)}{f(x)} \leq 4\frac{f(x)}{f(x)} - 4\frac{f((3/4)x)}{f(x)} \quad \quad \quad \text{if $f$ is concave}.
\end{align*}
Since $f$ is slowly varying, the left and right bounds converge to $0= \beta $.
\end{proof}

\begin{coro} \label{coro:: regular_var}
Let $f$ be a regularly varying function at zero, with index of regular variation $\beta \in \mathbb{R}$. Assume that $f$ is of class $C^1$ in a neighborhood of zero, such that $f$ is non-null and $f'$ is monotone in such neighborhood. Then, for any $\theta  \in \mathbb{R}$, $x \mapsto x^{\theta} f(x)$ is non-decreasing in a neighborhood of 0 when $\theta + \beta > 0$ and non-increasing in a neighborhood of 0 when $\theta + \beta < 0$.

\end{coro}

\begin{proof}
Immediate by differentiating the functions and applying Proposition \ref{prop:slowly_varying}.
\end{proof}

\subsubsection{Self-convolution for functions in $\mathcal C_4$}

The first idea to prove Proposition \ref{propo:self_convolution} is to justify that any regularly varying function in a neighborhood of zero with locally monotone derivative becomes a non-decreasing function in this neighborhood when sufficiently self-convoluted. The proof is based on Lemma \ref{lem:: regular_to_increas}.

However the function is only non-decreasing on a subset $[0,\varepsilon]$ that could be smaller than $[0,x]$. Actually, Lemma \ref{lemm: increas_propagation} will allow to extend the domain of monotony of the sequence of functions $(f_{S_n})_{n \in \mathbb{N}^*}$ as $n$ grows. The proof of this lemma is based on two other technical Lemmata \ref{lem:: derivation_convo} and \ref{lem:: convo_of_increas_func}.

\medskip

\begin{lemma} \label{lem:: regular_to_increas}
Let $(X_n)_{n \in \mathbb{N}^*}$ be a sequence of i.i.d. positive random variables such that $X_1$ is absolutely continuous with respect to the Lebesgue measure. Denote $f$ the density function of $X_1$, $F$ its cumulative distribution function and $f_{S_n}$ the density function of $S_n = \sum_{k=1}^{n} X_k$. Define the assertions
\begin{enumerate}
    \item[(i)] There exist an integer $N_1 > 0$ and $\varepsilon>0$ such that $f_{S_{N_1}}$ is non-decreasing on $]0,\varepsilon]$.
    \item[(ii)] There exist an integer $N_2 > 0$,  $\varepsilon>0$ and $\beta > -1$, such that: $\forall c \in ]0;1[$, $\forall t \in ]0;\varepsilon]$, we have $f_{S_{N_2}}(ct) \leq c^{\beta} f_{S_{N_2}}(t)$. 
\end{enumerate}
Assertions (i) and (ii) are equivalent.
\end{lemma}

\begin{proof}
$\ $
\noindent \underline{(i) $\Rightarrow$ (ii):} 
If there exist an integer $N_1$ and $\varepsilon>0$ such that $f_{S_{N_1}}$ is non-decreasing on $]0;\varepsilon]$, we have 
\begin{eqnarray*}
    \forall c \in ]0;1[, \forall t \in [0;x], \quad f_{S_{N_1}} (ct) \leq f_{S_{N_1}} (t) = c^0 f_{S_{N_1}} (t),
\end{eqnarray*}
and (ii) is satisfied with $\beta = 0$.

\smallskip

\noindent \underline{(ii) $\Rightarrow$ (i):} 
Assume that there exist an integer $N_2 > 0$, $\varepsilon > 0$ and $ \beta > -1$ such that: $\forall c \in ]0;1[$, $\forall t \in ]0;\varepsilon]$, we have $f_{S_{N_2}}(ct) \leq c^{\beta} f_{S_{N_2}}(t)$. \\
One can prove prove by induction that $\forall k \in \mathbb{N}$, $\forall c \in ]0;1[$ and $\forall t \in ]0;\varepsilon]$,
\begin{eqnarray*}
    f_{S_{2^k N_2}}(ct) \leq c^{2^k(1+ \beta) -1} f_{S_{2^k N_2}}(t).
\end{eqnarray*}
The case where $k=0$ is true due to the assumption that we made. Now assume the hypothesis true at rank $k \geq 0$. Note that
$$S_{2^{k+1}N_2} = X_1 + \ldots + X_{2^k N_2} + X_{2^k N_2 +1} + \ldots + X_{2^{k+1}N_2},$$ 
where the sum of the $2^k N_2$ first terms and the sum of the $2^k N_2$ last terms are independent. We have
\begin{eqnarray*}
    f_{S_{2^{k+1}N_2}}(t) = \int_0^t f_{S_{2^k N_2}} ( t-y) f_{S_{2^k N_2}} (y) dy, \quad \quad \forall t \in \mathbb{R}_+^* .
\end{eqnarray*}
The upper bound of the interval of integration is $t$ since $(t-y) < 0$ if $y>t$ and $f_{S_{2^k N_2}}$ is null on the half-line $]- \infty ,0[$. 
For any $c \in ]0;1[$ and $t \in ]0;\varepsilon[$, we have
\begin{align*}
    f_{S_{2^{k+1}N_2}}(ct) & = \int_0^{ct} f_{S_{2^k N_2}} ( ct-y) f_{S_{2^k N_2}} (y) dy \\
& = \int_0^{t} f_{S_{2^k N_2}} ( ct-cz) f_{S_{2^k N_2}} (cz) c dz = c \int_0^{t} f_{S_{2^k N_2}} ( c(t-z)) f_{S_{2^k N_2}} (cz) dz \\
    & \leq c \int_0^{t} c^{2^k(1+ \beta) -1} f_{S_{2^k N_2}} ( t-z) c^{2^k(1+ \beta) -1} f_{S_{2^k N_2}} (z) dz = c^{2^{k+1}(1+ \beta) -1} f_{S_{2^{k+1} N_2}}(t),
\end{align*}
where the domination is provided by the induction assumption. Hence the claim is verified. \\
Now since $1 + \beta > 0$, there exists an integer $k_0$ such that $2^{k_0}(1+ \beta) -1 > 0$ and for all $c \in ]0;1[$, $c^{2^{k_0}(1+ \beta) -1} \leq 1$. Therefore for all $t \in ]0;\varepsilon]$ and $c \in ]0;1[$, $f_{S_{2^{k_0}N_2}}(ct) \leq f_{S_{2^{k_0}N_2}}(t)$, which proves that $f_{S_{2^{k_0}N_2}}$ is non-decreasing on $]0;\varepsilon]$.
\end{proof}

\begin{remark}[On Lemma \ref{lem:: regular_to_increas}]$\ $
\begin{enumerate}
    \item This lemma offers a new criterion which is less restrictive on the density functions of the sequence $(S_n)_{n \in \mathbb{N}^*}$ and allows us to free ourselves from the hypothesis of growth on $[0,x]$ of the density $f_{S_N}$ for a certain rank $N \in \mathbb{N}^*$. \\
    For example, we can considerate the sequence of random variables $(X_n)_{n \in \mathbb{N}^*}$ where the law of $X_1$ is the Gamma law $\Gamma(k, \theta)$ where $(k,\theta) \in ]0,1[\times \mathbb{R}^{*}_{+}$, which admits a density $f$ given by: $f(x) = \frac{x^{k-1} e^{-\frac{x}{\theta}}}{\Gamma(k+1) \theta^k}$. It is easy to see that (ii) from lemma \ref{lem:: regular_to_increas} is satisfied in a neighborhood of 0. It justifies the existence of $N>0$ such that $f_{S_N}$ is non-decreasing on this neighborhood.
    \item In the previous lemma, we considered non-decreasing functions on some subset $]0, \varepsilon]$. If the density function $f$ of a random variables from Class $\mathcal{C}_4$ verifies assumption $(ii)$ and consequently $(i)$, for $n$ sufficiently large, $f^{*n}$ will be non-decreasing on some subset $]0,\varepsilon]$. Such density function is continuous on $\mathbb{R}^*$. Since the convolution preserves the domain of continuity of the convoluted functions, $f^{*n}$ can be extended as a continuous non-decreasing function on $[0,\varepsilon]$.    
    \end{enumerate}
\end{remark}

Let us recall standard results about the differentiation of the convolution of functions. If $f$ and $g$ are functions of class $C_0^{\infty}( \mathbb{R} )$ (infinitely differentiable functions supported on compact sets) or if they are of class $C^1 ( \mathbb{R} )$, integrable, with integrable derivatives, then $f*g$ is well defined, of class $C^1$ and we have
$$        (f*g)' = f'*g = f*g'. $$
Nevertheless, some density functions from Class $\mathcal{C}_4$ have non-integrable derivatives. As a counter-example, let $X$ be a Gamma random variable, with density function:
\begin{eqnarray*}
    f_X ( t ) = \frac{1}{\Gamma(1/2)}  t^{-1/2} e^{-t} \mathbb{1}_{ \{ t > 0\} }.
\end{eqnarray*}
$f_X$ is of class $C^1(]0,+\infty[)$ and is integrable. But we have
\begin{eqnarray*}
    f_X' ( t ) = - \frac{1}{\Gamma(1/2)}  (t^{-1/2} + \frac{1}{2} t^{-3/2}) e^{-t} \mathbb{1}_{ \{ t > 0\} }
\end{eqnarray*}
and $f_X'$ is not integrable at 0, inducing that $f_X * f_X' $ is not defined on $\mathbb{R}_+^*$. 

It is nonetheless possible to differentiate the convolution of two functions of class $C^1(]0,+\infty[)$ integrable on $]0,+\infty[$, with non-integrable derivatives in a neighborhood of 0.
The next lemma 
proves this result and gives an expression of the derivative of convolutions for such functions.

\begin{lemma} \label{lem:: derivation_convo}
Let $g$ be a non-negative function of class $C^1(]0,+\infty[)$, null on $\mathbb{R}_-^*$, integrable on $]0,+\infty[$. Then for all $n \in \mathbb{N}^*$, $g^{*n}$ is of class $\mathcal{C}^1(]0,+\infty[)$, null on $\mathbb{R}_-^*$. Moreover for all $n \geq 2$, $j \in \{ 1, \ldots , n-1 \}$ and $x > 0$,
\begin{align}  \nonumber
    (g^{*n})' (x) & = \int_0^{\frac{x}{2}} g^{*j} (t) (g^{*(n-j)})'(x-t) dt + \int_0^{\frac{x}{2}} g^{*(n-j)} (t)  (g^{*j})' (x-t) dt \\ \label{eq:deriv_conv}
    & +  g^{*j} \left( \frac{x}{2} \right) g^{*(n-j)} \left( \frac{x}{2} \right).
\end{align}
\end{lemma}

\begin{remark}
Since in our case we study self-convoluted density functions, the previous lemma provides formulas of the derivative of such functions. Using the same arguments of the proof below, it is possible to generalize the lemma as follows. \\ 
\medskip
Let $g_1$, $g_2$ be non-negative functions of class $C^1(]0,+\infty[)$, null on $\mathbb{R}_-^*$, integrable on $]0,+\infty[$. Then for any $x>0$, $g_1*g_2$ is differentiable at $x$ and for any $y \in ]0,x[$, we have
\begin{align*}
    (g_1*g_2)'(x) & = \int_0^y g_1 (t) g_2'(x-t)dt + \int_0^{x-y} g_2 (t) g_1'(x-t)dt \\
    & + \frac{1}{2} (g_1(y)g_2(x-y) + g_1(x-y)g_2(y)).
\end{align*}
Lemma \ref{lem:: derivation_convo} is obtained considering $y=\frac{x}{2}$ and $g_1 = g_2 = g$.
\end{remark}

\begin{proof}
Note that the $n$-th convolution is always well defined on $]0, + \infty[$ because convolution of $L^1(]0,+\infty[) \cap C^1(]0,+\infty[)$ functions and for all $n \in \mathbb{N}^*$, $ g^{*n} \in  L^1(]0,+\infty[)$. We prove the result by induction on the number of iteration of the convolution. 
Let $n \in \mathbb{N^*}$. Assume that for any $k \in \{1, \ldots,n\}$, $g^{*k}$ is of class $C^1(]0,+\infty[)$ and that equality \eqref{eq:deriv_conv} holds. Let's prove that $g^{*(n+1)}$ of class $C^1(]0,+\infty[)$ and that equality \eqref{eq:deriv_conv} holds. 

\smallskip 
We firstly prove the continuity of the convoluted function.
Here, we cannot directly apply the continuity theorem of parametrized integrals since the upper bound of the integrals also depends on the variable of differentiation. Let's prove the right-continuity. The left-continuity can be proved the same way. 
Let $j \in \{ 1, \ldots , n \}$. Then $g^{*(n+1)}$ is defined on $]0,+ \infty[$ as
$$
    g^{*(n+1)}(x) = \int_0^x g^{*j} (t) g^{*(n-j+1)}(x-t) dt = \int_0^{+ \infty} g^{*j} (t) g^{*(n-j+1)}(x-t) \mathbb{1}_{\{ 0<t<x \}} dt.  
$$
In order to lighten the forthcoming formulas, we denote $g^{*j}$ as $g_1$ and $g^{*(n-j+1)}$ as $g_2$. 
Let $x>0$ and $0<h<\frac{x}{4}$.
\begin{align*}
    & g^{*(n+1)} (x+h) -  g^{*(n+1)} (x)  = \int_0^{x+h} g_1 (t) g_2(x+h-t) dt - \int_0^x g_1 (t) g_2(x-t) dt \\
    &\quad  = \underset{A(h)}{ \underbrace{\int_0^x g_1 (t) g_2(x+h-t) dt - \int_0^x g_1 (t) g_2(x-t) dt }} + \underset{B(h)}{\underbrace{ \int_x^{x+h} g_1 (t) g_2(x+h-t) dt } }.
\end{align*}
In the term $B(h)$, the function $g_1=g^{*j} $ is integrated on $]x,x+h[ \subset [x,x+\frac{x}{4}]$ (because $0<h<\frac{x}{4}$). But under the induction hypothesis, $g^{*j} \in C^1(]0,+\infty[)$. We especially have $g^{*j} \in C^0([x,x+\frac{x}{4}])$ and $g^{*j}$ is moreover bounded on $[x,x+\frac{x}{4}]$. We obtain
\begin{align*}
    0 \leq \int_x^{x+h} g_1 (t) g_2(x+h-t) dt \leq & ||g_1||_{\infty,[x,x+\frac{x}{4}]}  \int_x^{x+h} g_2(x+h-t) dt \\
    = & ||g_1||_{\infty,[x,x+\frac{x}{4}]}  \int_0^h g_2(t) dt.
\end{align*}
And since $g_2=g^{*(n-j+1)}$ is integrable, it is clear that the integral converges to 0 as $h$ converges to 0. \\
The term $A(h)$ can be re-written as follows.
\begin{align*}
    A(h) & = \int_0^x g_1 (t) g_2(x+h-t) dt - \int_0^x g_1 (t) g_2(x-t) dt    \\
    & = \int_0^{\frac{x}{2}} g_1 (t) g_2(x+h-t) dt - \int_0^{\frac{x}{2}} g_1 (t) g_2(x-t) dt \\
    & \quad +  \int_{\frac{x}{2}}^x g_1 (t) g_2(x+h-t) dt -\int_{\frac{x}{2}}^x g_1 (t) g_2(x-t) dt .
\end{align*}
By grouping the integral terms with respect to the interval of integration, we obtain
$$
    A(h) = \underset{C(h)}{ \underbrace{\int_0^{\frac{x}{2}} g_1 (t) [g_2(x+h-t) - g_2(x-t)] dt } }   + \underset{D(h)}{ \underbrace{ \int_{\frac{x}{2}}^x g_1 (t) [g_2(x+h-t)  -    g_2(x-t) ]dt } }.     
$$
To study $C(h)$, note that $| g^{*(n-j+1)}(x+h-t) - g^{*(n-j+1)}(x-t) | \leq g^{*(n-j+1)}(x+h-t) + g^{*(n-j+1)}(x-t) $, where for any $t \in ]0, \frac{x}{2}]$, $ x+h-t \in [\frac{x}{2}, x +\frac{x}{4}]  $ and $ x-t \in [\frac{x}{2}, x] \subset [\frac{x}{2}, x +\frac{x}{4}] $. But under induction hypothesis, $g_2 = g^{*(n-j+1)} \in C^1(]0,+\infty[)$ which implies that $g_2 \in C^0([\frac{x}{2},x+\frac{x}{4}])$ the function is bounded on $[\frac{x}{2},x+\frac{x}{4}]$. We can apply the dominated convergence to prove that the first integral converges to 0 as $h$ converges to 0. 
\begin{enumerate}
    \item   $t \in ]0,\frac{x}{2}] \mapsto g_1 (t) [g_2(x+h-t) - g_2(x-t)]$ is measurable (because continuous). 
    \item  $\left| g_1 (t) [g_2(x+h-t) - g_2(x-t)] \right| \leq g_1 (t) \times 2 ||g_2||_{\infty,[\frac{x}{2},x+\frac{x}{4}]} $, for any $t \in  ]0,\frac{x}{2}]$. And $g_1$ is integrable. 
    \item $ \lim_{h \to 0+}  g_1 (t) [g_2(x+h-t) - g_2(x-t)]  = 0$, for any $t \in ]0,\frac{x}{2}] $, because $g_2 = g^{*(n-j+1)}$ is continuous on $]0,+\infty[$.
\end{enumerate}
Since $0<h<\frac{x}{4} $ the following decomposition of $D(h)$ is well defined and
\begin{align*}
    D(h)& = 
    \int_{\frac{x}{2}}^{\frac{x}{2}+h}  g_1 (t) g_2(x+h-t) dt + \int_{\frac{x}{2}+h}^x g_1 (t) g_2(x+h-t) dt \\
    & \quad - \int_{\frac{x}{2}}^{x-h} g_1 (t) g_2(x-t) dt - \int_{x-h}^x g_1 (t) g_2(x-t) dt \\
    & =  \int_{\frac{x}{2}}^{\frac{x}{2}+h}  g_1 (t) g_2(x+h-t) dt + \int_{\frac{x}{2}}^{x-h} h_1 (t+h) h_2(x-t) dt \\
    & - \int_{\frac{x}{2}}^{x-h} g_1 (t) g_2(x-t) dt - \int_{x-h}^x g_1 (t) g_2(x-t) dt.
\end{align*}
By grouping the second and third integrals, we obtain
$$
     D(h)= \int_{\frac{x}{2}}^{x-h} [g_1 (t+h) - g_1 (t)] g_2(x-t) dt + \int_{\frac{x}{2}}^{\frac{x}{2}+h}  g_1 (t) g_2(x+h-t) dt  - \int_{x-h}^x g_1 (t) g_2(x-t) dt.
$$
To conclude, the convergence to 0 of the first integral is treated as we did with the term $C(h)$. The convergence to 0 of the last two terms is treated as we did with the term $B(h)$. 
We finally conclude that $g^{*(n+1)}$ is right-continuous on $]0,+ \infty[$.
The proof of the left-continuity uses the same arguments.
We deduce that $g^{*(n+1)}$ is of class $C^0(]0,+\infty[)$. 

\medskip

We can now prove that $g^{*(n+1)}$ is differentiable on $]0,+\infty[$, with continuous derivative. 
As explained in the proof of the continuity, since the bound of the integral depends on the variable of differentiation x, the usual theorem of differentiation of parametrized integrals can not be used. 
Let's prove that for any $x>0$, $g^{*(n+1)}$ is left and right differentiable and the left and right derivatives coincide. 

\smallskip
Let $x>0$ and $0<h<\frac{x}{4}$. 
\begin{align*}
    & \frac{1}{h} [ g^{*(n+1)} (x+h) - g^{*(n+1)} (x)  ] = \frac{1}{h} \left[ \int_0^{x+h} g_1 (t) g_2(x+h-t) dt - \int_0^x g_1 (t) g_2(x-t)dt \right] \\
    & \quad =  \frac{1}{h} \left[ \int_0^{\frac{x+h}{2}} g_1 (t) g_2(x+h-t) dt + \int_{\frac{x+h}{2}}^{x+h} g_1 (t) g_2(x+h-t)dt \right] \\
    & \qquad - \frac{1}{h} \left[ \int_0^{\frac{x}{2}} g_1 (t) g_2(x-t) dt + \int_{\frac{x}{2}}^x g_1 (t) g_2(x-t) dt \right] \\
    & \quad =  \frac{1}{h} \left[ \int_0^{\frac{x+h}{2}} g_1 (t) g_2(x+h-t) dt + \int_0^{\frac{x+h}{2}} g_1 (x+h -t) g_2(t) dt \right] \\
    &\qquad  - \frac{1}{h} \left[ \int_0^{\frac{x}{2}} g_1 (t) g_2(x-t) dt + \int_0^{\frac{x}{2}} g_1 (x-t) g_2(t) dt \right] \\
    & \quad =  \underset{E(h)}{ \underbrace{ \int_0^{\frac{x}{2}} g_1 (t) \frac{g_2(x+h-t) - g_2(x-t)}{h} dt } } +  \underset{F(h)}{ \underbrace{ \int_0^{\frac{x}{2}} g_2 (t) \frac{g_1 (x+h-t) - g_1 (x-t)}{h} dt } }  \\
    & \qquad +\dfrac{1}{2} \underset{G(h)}{ \underbrace{ \frac{2}{h} \int_{\frac{x}{2}}^{\frac{x}{2} + \frac{h}{2}}  [g_1 (t) g_2(x+h-t) + g_1 (x+h -t) g_2(t)] dt } }.
\end{align*}
Concerning the terms $E(h)$ and $F(h)$, since $g_2 = g^{*(n-j+1)}$ and $g_1 = g^{*j}$ are of $C^1(]0,+\infty[)$ by induction assumption, they necessarily are of class $C^1([\frac{x}{2},x+\frac{x}{4}])$ and the mean-value theorem states that $\forall x>0, \forall t \in ]0;\frac{x}{2}[, \forall h \in ]0;\frac{x}{4}[ $,
\begin{align*}
    & \left| \frac{g^{*(n-j+1)}(x+h-t) - g^{*(n-j+1)}(x-t)}{h}  \right| \leq || (g^{*(n-j+1)})' ||_{\infty , [\frac{x}{2},x+\frac{x}{4}]} \\
    & \left| \frac{g^{*j}(x+h-t) - g^{*j}(x-t)}{h}  \right| \leq || (g^{*j})' ||_{\infty , [\frac{x}{2},x+\frac{x}{4}]}
\end{align*}
Consequently, the integrated function of the two first integrals are bounded by integrable functions. And by the dominated convergence theorem, we deduce that
\begin{align*}
    & \lim_{h \xrightarrow{} 0^+} \int_0^{\frac{x}{2}} g_1 (t) \frac{g_2(x+h-t) - g_2(x-t)}{h} dt = \int_0^{\frac{x}{2}} g_1 (t) g_2'(x-t) dt , \\
    &  \lim_{h \xrightarrow{} 0^+} \int_0^{\frac{x}{2}} g_2 (t) \frac{g_1 (x+h-t) - g_1 (x-t)}{h} dt = \int_0^{\frac{x}{2}} g_2 (t)  g_1' (x-t) dt.
\end{align*}
Concerning the term $G(h)$, 
since $t \mapsto g_1 (t) g_2(x+h-t) + g_1 (x+h -t) g_2(t) = g^{*j} (t) g^{*(n-j+1)}(x+h-t) + g^{*j} (x+h -t) g^{*(n-j+1)}(t)$ is of class $C^0([\frac{x}{2}; \frac{x}{2} + \frac{h}{2}])$, we immediately deduce that 
\begin{align*}
    \lim_{h \xrightarrow{} 0^+}  \frac{2}{h} \int_{\frac{x}{2}}^{\frac{x}{2} + \frac{h}{2}}  [g_1 (t) g_2(x+h-t) + g_1 (x+h -t) g_2(t)] dt = g_1 \left( \frac{x}{2} \right) g_2 \left( \frac{x}{2} \right) + g_1 \left( \frac{x}{2} \right) g_2 \left( \frac{x}{2} \right).
\end{align*}
We can conclude that $g^{*(n+1)}$ is right differentiable and for any $x>0$,
\begin{align*}
    & \lim_{h \xrightarrow{} 0^+} \frac{1}{h} [ g^{*(n+1)} (x+h) - g^{*(n+1)} (x)  ] =  \int_0^{\frac{x}{2}} g^{*j} (t) (g^{*(n-j+1)})'(x-t) dt \\
    & \qquad + \int_0^{\frac{x}{2}} g^{*(n-j+1)} (t)  (g^{*j})' (x-t) dt
    +  g^{*j} \left( \frac{x}{2} \right) g^{*(n-j+1)} \left( \frac{x}{2} \right).
\end{align*}
The determination of the left derivative uses the same arguments (we separate the integral in two parts by cutting the interval in the middle). It is easy to see that the left derivative is the same as the right one. To end the proof we need to verify that the derivative of $(g^{*(n+1)})'$ is of class $C^0(]0,+\infty[)$, which is proved the same way as the continuity of $(g^{*(n+1)})$. Hence $g^{*(n+1)}$ is of class $C^1(]0,+\infty[)$. 
\end{proof}

We recall a classical result which states that the convolution of functions preserves the interval on which one of them is non-decreasing.
\begin{lemma}  \label{lem:: convo_of_increas_func}
Let $f,g$ be two non-negative and continuous functions on $\mathbb{R}_+^*$, equal to 0 on $\mathbb{R}_-^*$, integrable. Assume that $f$ is non-decreasing on $]0,\varepsilon]$, for some $\varepsilon>0$. Then $f*g$ is also non-decreasing on $]0,\varepsilon]$.
\end{lemma}

\begin{proof}
Note that $f*g$ is well defined. $f$ is continuous on $\mathbb{R}_+^*$ and non-decreasing in a neighborhood of zero. Consequently, it can be extended by continuity at zero. Thus, we obtain the following domination: for any $x>0$
$$\left| \int_0^x g(t) f(x-t) dt  \right| \leq \sup_{t \in [0,x]} | f(t) | \left| \int_0^x g(t) dt \right|.$$ \\
Now let $x_1,x_2 \in ]0,\varepsilon] $, with $x_1 \leq x_2$. Then
\begin{align*}
    f*g(x_2) - f*g(x_1) & = \int_0^{x_2} g(t) f(x_2 - t) dt - \int_0^{x_1} g(t) f(x_1 -t) dt \\
    & = \int_0^{x_1} g(t) [f(x_2 -t)-f(x_1 -t)] dt + \int_{x_1}^{x_2} g(t) f(x_2 -t) dt.
\end{align*}
$f$ and $g$ are non-negative. Then the second integral in the right-hand side is the integral of a non-negative function. $f$ is non-decreasing on $]0,\varepsilon]$, then $ f(x_2 -t)-f(x_1 -t) \geq 0 $, $\forall t \in [0,x_1[$. Consequently, the first term of the right-hand side is also the integral of a non-negative function. 
Finally $f*g(x_2) - f*g(x_1) \geq 0$, that is $f*g$ is non-decreasing on $]0,\varepsilon]$.
\end{proof}

\begin{lemma}  \label{lemm: increas_propagation}
Let $g$ be a non-negative function of class $ C^0 ([0, + \infty[) \cap  C^1 (]0, + \infty[)$, identically null on $\mathbb{R}^*_-$. Assume the existence of $\varepsilon > 0$ such that $g$ is non-decreasing on $[0,\varepsilon]$. 
Then for any $k \in \mathbb{N}$, $g^{*2^{2k}}$ is non-decreasing on $\left[ 0, \left( \frac{3}{2} \right)^k \varepsilon \right]$.
\end{lemma}

\begin{proof}
We will prove this result by induction on $k \in  \mathbb{N}$. The claim holds for $k=0$ by assumptions on $g$. 
In order to lighten the notations along the proof we pose for any $n \in \mathbb{N}$, $h_n = g^{*2^{n}}$ ; in particular $h_0 = g$.  
Before beginning the proof, note that Lemma \ref{lem:: derivation_convo} ensures that for any $n \in \mathbb{N}^*$, $g^{*n}$ is of class $ C^0 ([0, + \infty[) \cap  C^1 (]0, + \infty[)$. It is worth mentioning that Lemma \ref{lem:: derivation_convo} involves that $g^{*n}$ is of class $  C^1 (]0, + \infty[)$ and the continuity at zero can be deduced be observing that since $g$ is continuous on $[0, + \infty[$, we have for $n \geq 2$, $\displaystyle \lim_{x\to 0} g^{*n}(x) =  0$, which allows us to extend these functions to zero.
Assume that our statement holds for some $k \in \mathbb{N}$ and let's show the result for $k+1$ 
by proving that
$$\forall x \in \left[ 0, \left( \frac{3}{2} \right)^{k+1} \varepsilon \right], \quad (h_{2(k+1)})'(x) \geq 0  .$$ 
For any $x \in \left[ 0, \left( \frac{3}{2} \right)^{k+1} \varepsilon \right]$, we have
\begin{eqnarray*}
    h_{2(k+1)}(x) = ( h_{2k+1} * h_{2k+1} )(x) = \int_0^x h_{2k+1} (t) h_{2k+1} (x-t) dt = 2 \int_0^{\frac{x}{2}} h_{2k+1} (t) h_{2k+1} (x-t) dt.
\end{eqnarray*}
Lemma \ref{lem:: derivation_convo} tells that $x > 0$
\begin{align*}
     & (h_{2(k+1)})'(x)  =  2 \int_0^{\frac{x}{2}} h_{2k+1} (t) \left( h_{2k+1} \right)' (x-t) dt + \left[ h_{2k+1} \left( \frac{x}{2} \right) \right]^2  \\
     & \quad = 2 \left[ h_{2k+1} (t) \left( - h_{2k+1} (x-t) \right)  \right]_0^{\frac{x}{2}} + 2 \int_0^{\frac{x}{2}} \left( h_{2k+1} \right)' (t) h_{2k+1} (x-t) dt + \left[ h_{2k+1} \left( \frac{x}{2} \right) \right]^2  \\
     & \quad =  2   h_{2k+1} (0) h_{2k+1} (x) + 2 \int_0^{\frac{x}{2}} \left( h_{2k+1} \right)' (t) h_{2k+1} (x-t) dt - \left[ h_{2k+1} \left( \frac{x}{2} \right) \right]^2.
\end{align*}
Since $g$ is non-decreasing and continuous on $[0,\varepsilon]$, it induces  that $\forall n \geq 2 $, $g^{*n}(0) = 0$. Then
\begin{align} \nonumber
    (h_{2(k+1)})'(x) & = 2 \int_0^{\frac{x}{2}} \left( h_{2k+1} \right)' (t) h_{2k+1} (x-t) dt - \left[ h_{2k+1} \left( \frac{x}{2} \right) \right]^2 \\ \nonumber
    & = 2 \int_0^{\frac{x}{2}} \left( h_{2k+1} \right)' (t) h_{2k+1} (x-t) dt - \int_0^{\frac{x}{2}} 2 \left( h_{2k+1} \right)' (t) h_{2k+1} (t) dt \\ \label{eq:tech_positive_property_2}
    & =  2 \int_0^{\frac{x}{2}} \left( h_{2k+1} \right)' (t) \left[ h_{2k+1} (x-t) - h_{2k+1} (t) \right] dt.
\end{align}
We now verify that $\forall x \in \left[ 0, \left( \frac{3}{2} \right)^{k+1} \varepsilon \right]$, $\forall t \in \left[ 0 , \frac{x}{2}  \right]$, $\left( h_{2k+1} \right)' (t) \left[ h_{2k+1} (x-t) - h_{2k+1} (t) \right] \geq 0 $. 
Under induction hypothesis, $h_{2k}$ is non-decreasing on $\left[ 0, \left( \frac{3}{2} \right)^k \varepsilon \right]$. 
Since $h_{2k+1} = h_{2k} * h_{2k} $, we deduce from Lemma \ref{lem:: convo_of_increas_func} that $h_{2k+1}$ is non-decreasing on $\left[ 0, \left( \frac{3}{2} \right)^k \varepsilon \right]$. 
Consequently, if $x \in \left[ 0, \left( \frac{3}{2} \right)^{k+1} \varepsilon \right]$,  then $ \frac{x}{2} \in \left[ 0, \left( \frac{3}{2} \right)^k \varepsilon \right]$. Then for all $t \in \left[ 0 , \frac{x}{2}  \right]$, $\left( h_{2k+1} \right)' (t) \geq 0$. 

It remains to prove that 
\begin{equation} \label{eq:tech_positive_property_1}
  \forall x \in \left[ 0, \left( \frac{3}{2} \right)^{k+1} \varepsilon \right], \ \forall t \in \left[ 0 , \frac{x}{2}  \right], \quad  h_{2k+1} (x-t) - h_{2k+1} (t) \geq 0.  
\end{equation}
We proceed in two steps. 
\begin{enumerate}
    \item First we prove that
$$\forall t \in \left[ 0 , \frac{1}{2}  \left( \frac{3}{2} \right)^{k+1} \varepsilon  \right], \quad h_{2k+1} \left( \left( \frac{3}{2} \right)^{k+1} \varepsilon -t \right) - h_{2k+1} (t)  \geq 0.$$
Indeed if $ \left( \left( \frac{3}{2} \right)^{k+1} \varepsilon -t \right) \leq  \left( \frac{3}{2} \right)^{k} \varepsilon $, since $h_{2k+1}$ is non-decreasing on  $\left[ 0, \left( \frac{3}{2} \right)^k \varepsilon \right]$ and since $t \leq \left( \frac{3}{2} \right)^{k+1} \varepsilon -t$, we deduce the desired inequality. 
But if $ \left( \left( \frac{3}{2} \right)^{k+1} \varepsilon -t \right) >  \left( \frac{3}{2} \right)^{k} \varepsilon $, then $ t < \left( \frac{3}{2} \right)^{k+1} \varepsilon - \left( \frac{3}{2} \right)^{k} \varepsilon = \frac{1}{2} \left( \frac{3}{2} \right)^{k} \varepsilon  $. And because $ h_{2k+1} = h_{2k} * h_{2k} $, we have
\begin{align*}
    & h_{2k+1} \left( \left( \frac{3}{2} \right)^{k+1} \varepsilon -t \right) - h_{2k+1} (t) \\
    & \quad = \int_0^{\left( \frac{3}{2} \right)^{k+1} \varepsilon -t} h_{2k} (y)  h_{2k} \left( 
    \left( \frac{3}{2} \right)^{k+1} \varepsilon -t - y \right) d y   - \int_0^t h_{2k} (y) h_{2k} (t-y) d y \\
    &\quad  \geq \int_{\frac{1}{2} \left( \frac{3}{2} \right)^{k} \varepsilon - t}^{ \frac{1}{2} \left( \frac{3}{2} \right)^{k} \varepsilon } h_{2k} (y)  h_{2k} \left( \left( \frac{3}{2} \right)^{k+1} \varepsilon -t - y \right) dy  - \int_0^t h_{2k} (y) h_{2k} (t-y) dy \\
    & \quad = \int_0^t h_{2k} \left( y + \frac{1}{2} \left( \frac{3}{2} \right)^{k} \varepsilon - t \right)  h_{2k} \left( \left( \frac{3}{2} \right)^{k} \varepsilon - y \right) dy  - \int_0^t h_{2k} (y) h_{2k} (t-y) dy.
\end{align*}
Finally, since $t < \frac{1}{2} \left( \frac{3}{2} \right)^{k} \varepsilon $ and since $ h_{2k}$ is non-decreasing on  $\left[ 0, \left( \frac{3}{2} \right)^k \varepsilon \right]$, it induces that
\begin{eqnarray*}
    0 & \leq h_{2k} (y)  \leq h_{2k} \left( y + \frac{1}{2} \left( \frac{3}{2} \right)^{k}  \varepsilon - t \right),  \quad \quad \quad \forall y \in [0,t], \\
    0 & \leq h_{2k} (t-y)  \leq h_{2k} \left( \left( \frac{3}{2} \right)^{k} \varepsilon - y \right), \quad \quad \quad \forall y \in [0,t].
\end{eqnarray*}
Thus $ h_{2k+1} \left( \left( \frac{3}{2} \right)^{k+1} \varepsilon -t \right) - h_{2k+1} (t)$ is greater than a non-negative integral, which concludes the proof of this first step. 
\item  Now we deduce Inequality \eqref{eq:tech_positive_property_1}.
Let $ x \in \left[ 0, \left( \frac{3}{2} \right)^{k+1} \varepsilon \right], \; t \in \left[ 0 , \frac{x}{2}  \right] $. If $ x-t \leq \left( \frac{3}{2} \right)^{k} \varepsilon$, since $ h_{2k+1} $ is non-decreasing on  $\left[ 0, \left( \frac{3}{2} \right)^k \varepsilon \right]$, and since $t \leq x-t$,  $h_{2k+1} (x-t) - h_{2k+1} (t)  \geq 0$. If $ x-t \geq \left( \frac{3}{2} \right)^{k} \varepsilon$, we pose $t' = \left( \frac{3}{2} \right)^{k+1} \varepsilon - x + t \geq t  $. We have: $t' = \left( \frac{3}{2} \right)^{k+1} \varepsilon - x + t = \left( \frac{3}{2} \right)^{k+1} \varepsilon - (x - t) \leq \left( \frac{3}{2} \right)^{k+1} \varepsilon - \left( \frac{3}{2} \right)^{k} \varepsilon = \frac{1}{2} \left( \frac{3}{2} \right)^{k} \varepsilon$.
Again since $h_{2k+1}$ is non-decreasing on $\left[ 0, \left( \frac{3}{2} \right)^k \varepsilon \right]$, we obtain:
\begin{align*}
    h_{2k+1} (x-t) - h_{2k+1} (t) & = h_{2k+1} \left( \left( \frac{3}{2} \right)^{k+1} \varepsilon - t' \right) - h_{2k+1} (t) \\
      & \geq  h_{2k+1} \left( \left( \frac{3}{2} \right)^{k+1} \varepsilon - t' \right) - h_{2k+1} (t').
\end{align*} 
From the first step, we deduce that 
\begin{align*}
    h_{2k+1} (x-t) - h_{2k+1} (t) &\geq  h_{2k+1} \left( \left( \frac{3}{2} \right)^{k+1} \varepsilon - t' \right) - h_{2k+1} (t') \geq 0.
\end{align*} 
Hence Inequality \eqref{eq:tech_positive_property_1} is proved. 
\end{enumerate}
Combining \eqref{eq:tech_positive_property_1} with \eqref{eq:tech_positive_property_2}, we conclude that $(h_{2(k+1)})'(x) \geq 0$, $\forall x \in \left[ 0, \left( \frac{3}{2} \right)^{k+1} \varepsilon \right] $, because integral of non-negative functions. This achieves the proof. 
\end{proof}

\paragraph{Proof of Proposition \ref{propo:self_convolution}}
This result appears as a consequence of Lemmas \ref{lem:: regular_to_increas} and \ref{lemm: increas_propagation}. If the density function $f$ of a $\mathcal{C}_4$-variable verifies assertion $(ii)$ from Lemma \ref{lem:: regular_to_increas}, the existence of $\varepsilon > 0$ and $n \in \mathbb{N}^*$ such that $f^{*n}$ is non-decreasing on $]0,\varepsilon]$ can be deduced. Since $f^{*n}$ (for such $n$) can be extended as a non-decreasing continuous function on $[0,\varepsilon]$, we deduce from Lemma \ref{lemm: increas_propagation} the existence of $m \geq n$ such $f^{*m}$ is non-decreasing on $[0,x]$. It consequently suffices to prove that the regularity assumptions made on $f$ imply assertion $(ii)$ from Lemma \ref{lem:: regular_to_increas}.

Let us detail the two steps. First 
we know the existence of $\alpha > 0$ such that $f$ is regularly varying of index $\alpha -1$ in a neighborhood of 0 and that $f'$ is monotone in such neighborhood.
We consider different cases depending on the values of the index of regular variations. Let us fix $\beta = 1 - \frac{\alpha}{2}$. Then $\beta + 1 -\alpha = \frac{\alpha }{2} > 0$. 
From Corollary \ref{coro:: regular_var}, we deduce the existence of $\varepsilon >0$ such that $x \mapsto x^{\beta} f(x)$ is non-decreasing on $[0,\varepsilon]$.
Then for any $t \in ]0,\varepsilon]$, $c \in ]0,1]$:
$$
    (ct)^{\beta} f(ct) \leq t^{\beta}f(t) \quad \quad \iff \quad \quad f(ct) \leq c^{- \beta} f(t).
$$
From Lemma \ref{lem:: regular_to_increas}, we deduce the existence of an integer $N_1 \in \mathbb{N}^*$ such that $f_{S_{N_1}}$ is non-decreasing on $]0,\varepsilon]$. Since $f_{S_{N_1}}$ is of class $ C^1 (]0, + \infty[) $ and is non-decreasing on $]0,\varepsilon]$, it can be extended by continuity as a function of class $C^0 ([0, + \infty[) \cap C^1 (]0, + \infty[) $. W.l.o.g. we still denote this function by $f_{S_{N_1}}$.  

Secondly, applying Lemma \ref{lemm: increas_propagation} to $f_{S_{N_1}}$ leads to the property that for any $k \in \mathbb{N}$, $f_{S_{2^{2k} N_1}} $ is non-decreasing on $\left[ 0, \left( \frac{3}{2} \right)^k \varepsilon \right]$. We define $n_0 = \lceil \log \left( \frac{x}{\varepsilon} \right) / \log \left( \frac{3}{2} \right)  \rceil$. Then for $N_0 = 2^{2n_0} N_1$, $f_{S_{N_0}}$ is non-decreasing on $[0,x]$. From Lemma \ref{lemm: increas_propagation}, we deduce that it is true for any $n \geq N_0$. 

This achieves the proof of the Proposition.

\subsubsection{Proofs of Theorem \ref{theo:: class c_5} and Proposition \ref{prop:class_C_6}}

Here we prove the upper bounds for the ratio sequences when the random variables belong to the classes $\mathcal C_4$ and $\mathcal C_5$. Let us start the proof of Theorem \ref{theo:: class c_5}.
\begin{proof}
The proof of the theorem relies on the non-decreasing property of $f$ sufficiently convoluted established in Proposition \ref{propo:self_convolution} to deduce convexity inequalities that will induce the convergence result. \\
\medskip 
There exists $\alpha > 0$ such that $f$ is regularly varying of index $\alpha -1$ in a neighborhood of 0. 
 We deduce from \cite[Proposition 1.5.8]{bing:gold:teug:89} that
$
    \mathbb{P} [X_1 \leq t] \underset{t \xrightarrow{} 0^+}{\sim} \frac{1}{\alpha} t f(t).
$
Thus 
\begin{equation} \label{eq:mono_dens_equiv}
    \mathbb{P} \left[ X_1 \leq \frac{1}{n} \right] \underset{n \xrightarrow{} \infty}{\sim} \dfrac{1}{\alpha} f \left( \frac{1}{n} \right) \frac{1}{n}.
\end{equation}
Hence the behavior of this probability depends on the behavior of $f$ close to zero. When $\alpha > 1$, then the limit of $f$ at zero is zero; when $0 < \alpha < 1$, this limit is $\infty$. When $\alpha = 1$, since $f'$ is monotone on a neighborhood of zero, this property is also true for $f$, therefore the limit of $f$ at zero still exists in $[0,\infty]$.

Now using simple rules of integration, we obtain
$$
    \frac{\mathbb{P} \left[ S_{n+1} \leq x  \right]}{ \mathbb{P} \left[ S_n \leq x  \right] } = \int_0^x  f(y_1) \frac{F_{S_n} (x - y_1)}{F_{S_n}(x)} dy_1 = \int_0^x  f(y_1) \frac{ \mathbb{P} [S_n \leq x - y_1 ]}{\mathbb{P} [S_n \leq x ]} dy_1.
$$
As a consequence of Proposition \ref{propo:self_convolution}, if we set $K_x$ the minimal integer for which, $f_{S_{K_x}}$ is non-decreasing on $[0,x]$, 
from the proof of \cite[Theorem 4.6]{Sah25}, for any $n \geq K_x$, and for any $y_1 \in [0,x]$,
\begin{align} \label{eq::expo_bound_ratio}
    \frac{ \mathbb{P} [S_n \leq x - y_1 ]}{\mathbb{P} [S_n \leq x ]} \leq \left( \frac{x-y_1}{x} \right)^{ \frac{n}{K_x }  - 1  }.
\end{align}
Consequently
\begin{align} \nonumber
    \frac{\mathbb{P} \left[ S_{n+1} \leq x  \right]}{ \mathbb{P} \left[ S_n \leq x  \right] }  &  \leq \int_0^x  f(y_1) \left( \frac{x-y_1}{x} \right)^{ \frac{n}{K_x }  - 1  } dy_1 \\ \nonumber
    & = \int_0^{1/n}  f(y_1) \left( \frac{x-y_1}{x} \right)^{ \frac{n}{K_x }  - 1  } dy_1 + \int_{1/n}^{x}  f(y_1) \left( \frac{x-y_1}{x} \right)^{ \frac{n}{K_x }  - 1  } dy_1 \\ \nonumber
    & \leq \int_0^{1/n}  f(y_1)  dy_1 + \int_{1/n}^{x}  f(y_1) \left( \frac{x-y_1}{x} \right)^{ \frac{n}{K_x }  - 1  } dy_1 \\ \label{eq:density_case_upper_bound_1}
     & =  \mathbb{P} \left[ X_1 \leq \frac{1}{n} \right]
    + \int_{1/n}^{x}  f(y_1) \left( \frac{x-y_1}{x} \right)^{ \frac{n}{K_x }  - 1  } dy_1 .
\end{align} 
To end the proof, we compare the last integral with the probability $\mathbb P(X_1 \leq 1/n)$. Let us start with the easiest case.

\medskip 
\noindent \underline{Assume that $\lim_{x \xrightarrow{} 0^+} f(x)= + \infty$.} Evoke that this case is possible only if $0 < \alpha \leq 1$. From \eqref{eq:mono_dens_equiv},
$$\dfrac{1}{n}  \underset{n \xrightarrow{} \infty}{ = } o \left( \mathbb{P} \left[ X_1 \leq \frac{1}{n} \right] \right).$$
Since $f'$ and consequently $f$ are ultimately monotone\footnote{Indeed $f'$ is continuous and either non-decreasing or non-increasing on a neighborhood of zero. Thus its sign remains constant of such neighborhood.}, we know the existence of $\varepsilon_1 >0$ such that $f$ and $f'$ are monotone functions on $]0,\varepsilon_1[$. The limit at zero of the density function implies that $f$ is non-increasing on $]0,\varepsilon_1[$.
Then denote $N = \max \{ \left\lfloor \frac{1}{\varepsilon} \right\rfloor +1; \left\lfloor \frac{1}{\varepsilon_1} \right\rfloor +1 ;  K_x   \}$.
Using \eqref{eq:density_case_upper_bound_1}, for any $n \geq N$,
\begin{align*}
    \frac{\mathbb{P} \left[ S_{n+1} \leq x  \right]}{ \mathbb{P} \left[ S_n \leq x  \right] } & \leq \mathbb{P} \left[ X_1 \leq \frac{1}{n} \right]+ \int_{\frac{1}{n}}^x  f(y_1) \left( \frac{x-y_1}{x} \right)^{ \frac{n}{K_x }  - 1  } dy_1 \\
    & \leq  \mathbb{P} \left[ X_1 \leq \frac{1}{n} \right] + ||f||_{\infty , [1/n;x]} \int_{\frac{1}{n}}^x \left( \frac{x-y_1}{x} \right)^{ \frac{n}{K_x } n - 1  } dy_1 \\
    & \leq \mathbb{P} \left[ X_1 \leq \frac{1}{n} \right] + ||f||_{\infty , [1/n;x]} \frac{ x K_x}{n} .
\end{align*}
Since $f$ is non-increasing on $]0,\varepsilon_1[$, we have that $ ||f||_{\infty , [1/n;x]} \leq f(1/n) + ||f||_{\infty , [\varepsilon_1 ;x]}  $ and then:
\begin{align*}
    \frac{\mathbb{P} \left[ S_{n+1} \leq x  \right]}{ \mathbb{P} \left[ S_n \leq x  \right] } & \leq  \mathbb{P} \left[ X_1 \leq \frac{1}{n} \right] + f\left( \dfrac{1}{n}\right) \frac{ x K_x}{n} + \frac{ ||f||_{\infty , [\varepsilon_1 ;x]}  x K_x}{n} \\
    & \leq \left( 1 + 2 x \alpha K_x \right) \mathbb{P} \left[ X_1 \leq \frac{1}{n} \right] +  \frac{ ||f||_{\infty , [\varepsilon_1 ;x]}  xK_x}{n} \\
    & \leq \left( 1 + 2 x \alpha K_x \right) \mathbb{P} \left[ X_1 \leq \frac{1}{n} \right] + o \left( \mathbb{P} \left[ X_1 \leq \frac{1}{n} \right] \right).
\end{align*}
The last inequality is verified for $n$ sufficiently large. It is due to the obtained monotone density equivalence \eqref{eq:mono_dens_equiv}. 
Hence for $n$ sufficiently large
\begin{align*}
    \frac{\mathbb{P} \left[ S_{n+1} \leq x  \right]}{ \mathbb{P} \left[ S_n \leq x  \right] } \leq \left( 2 + 2 x \alpha K_x  \right) \mathbb{P} \left[ X_1 \leq \frac{1}{n} \right] .
\end{align*}
Finally Inequality \eqref{eq:first_ineq_thm_density} is obtained by considering  $C_x = 2 + 2 \alpha x K_x $. Define $N_x$ the minimal integer for which the previous inequality is verified for any greater integer. If we considered $\frac{\mathbb{P} \left[ S_{n+1} \leq t  \right]}{ \mathbb{P} \left[ S_n \leq t  \right] }$ for some $t\in [0,x]$, the dominating constant would have been smaller than the obtained constant, when $t=x$, leading to the uniform upper-bound.  

\medskip 
\noindent \underline{Assume that $f$ admits a finite positive limit at 0.} This case is possible only if $\alpha = 1$ and from \eqref{eq:mono_dens_equiv}, the sequences
$\dfrac{f(0)}{n}$ and $\mathbb{P} \left[ X_1 \leq \frac{1}{n} \right]$ are equivalent. The arguments are quite similar to the previous case. Indeed
\begin{align*}
    \int_{1/n}^{x}  f(y_1) \left( \frac{x-y_1}{x} \right)^{ \frac{n}{K_x }  - 1  } dy_1 \leq & \| f \|_{\infty , [0,x]} \dfrac{x K_x}{n} \left( 1 - \dfrac{1}{nx} \right)^{\frac{n}{K_x}}  \\
    \underset{n \to + \infty}{ \sim } & \| f \|_{\infty , [0,x]}  \dfrac{x K_x}{n} \exp \left( - \dfrac{K_x}{x} \right).
\end{align*}
Hence coming back to \eqref{eq:density_case_upper_bound_1}, we deduce that Inequality \eqref{eq:first_ineq_thm_density} holds.

\medskip 
\noindent \underline{Assume that the limit of $f$ at 0 is zero.} It implies that $\alpha \geq 1$ and from \eqref{eq:mono_dens_equiv}, we obtain that 
\begin{align*}
    \mathbb{P} \left[ X_1 \leq \frac{1}{n} \right] \underset{n \xrightarrow{} \infty}{=} o \left( \dfrac{1}{n} \right).
\end{align*}
Hence we cannot apply the previous arguments to get the desired result.
To end the proof, we now justify that 
\begin{align}\label{eq:density_case_upper_bound_4}
    \int_{1/n}^{x}  f(y_1) \left( \frac{x-y_1}{x} \right)^{ \frac{n}{K_x }  - 1  } dy_1  \underset{n \xrightarrow{} \infty}{=}  O \left( \mathbb{P} \left[ X_1 \leq \frac{1}{n} \right] \right).
\end{align}
To prove it, we fix some sequence $(u_n, \ n \geq 1)$ converging to zero, greater than $1/n$, and we split the integral into two terms:
\begin{align} \nonumber
    \int_{1/n}^{x}  f(y_1) \left( \frac{x-y_1}{x} \right)^{ \frac{n}{K_x }  - 1  } dy_1 &  = \int_{1/n}^{u_n}  f(y_1) \left( \frac{x-y_1}{x} \right)^{ \frac{n}{K_x }  - 1  } dy_1 \\ \label{eq:density_case_upper_bound_2}
    & + \int_{u_n}^{x}  f(y_1) \left( \frac{x-y_1}{x} \right)^{ \frac{n}{K_x }  - 1  } dy_1.
\end{align}

Since $f'$ is supposed to be monotone in a neighborhood of $0$, we deduce from Corollary  \ref{coro:: regular_var}, we deduce that if $ \alpha - 1 < \delta$, then $t \mapsto  f(t) t^{- \delta}$ is a non-increasing function in a certain neighborhood $]0,\epsilon[$ of the origin. Thus for $n$ sufficiently large: 
\begin{align*}
    & \int_{1/n}^{u_n}  f(y_1) \left( \frac{x-y_1}{x} \right)^{ \frac{n}{K_x }  - 1  } dy_1 = \int_{1/n}^{u_n}  f(y_1) y_1^{- \delta} y_1^{ \delta}  \left( \frac{x-y_1}{x} \right)^{ \frac{n}{K_x }  - 1  } dy_1 \\
    & \quad \leq f \left( \frac{1}{n} \right) n^{\delta} \int_{1/n}^{u_n} y_1^{ \delta}  \left( \frac{x-y_1}{x} \right)^{ \frac{n}{K_x }  - 1  } dy_1 \leq f \left( \frac{1}{n} \right) n^{\delta} \int_{1/n}^{x} y_1^{ \delta}  \left( \frac{x-y_1}{x} \right)^{ \frac{n}{K_x }  - 1  } dy_1.
\end{align*}
We claim that 
\begin{equation} \label{eq:integral_estim_proof_thm_density_case}
    \int_{1/n}^{x} y_1^{ \delta}  \left( \frac{x-y_1}{x} \right)^{ \frac{n}{K_x }  - 1  } dy_1 \underset{n \to \infty}{=} O (n^{-\delta-1}).
\end{equation}
The proof is set up just after the proof of the theorem. Then there exists a constant $C$
such that 
\begin{align*}
    & \int_{1/n}^{u_n}  f(y_1) \left( \frac{x-y_1}{x} \right)^{ \frac{n}{K_x }  - 1  } dy_1 \leq C f \left( \frac{1}{n} \right) n^{\delta} n^{-\delta-1} = C f \left( \frac{1}{n} \right)\dfrac{1}{n}.
\end{align*}
Using \eqref{eq:mono_dens_equiv}, we obtain that:
\begin{equation} \label{eq:density_case_upper_bound_3}
    \int_{1/n}^{u_n}  f(y_1) \left( \frac{x-y_1}{x} \right)^{ \frac{n}{K_x }  - 1  } dy_1 \leq C \mathbb{P} \left[ X_1 \leq \frac{1}{n} \right].
\end{equation}
For the second integral, we have
\begin{align*}
    & \int_{u_n}^x  f(y_1) \left( \frac{x-y_1}{x} \right)^{ \frac{n}{K_x }  - 1  } dy_1 \leq \|f\|_{\infty,[0,x]} \int_{u_n}^x   \left( \frac{x-y_1}{x} \right)^{ \frac{n}{K_x }  - 1  } dy_1 \\
    & \quad =\|f\|_{\infty,[0,x]} \dfrac{K_x}{n}  \left( 1 - \frac{u_n}{x} \right)^{ \frac{n}{K_x }    } = \|f\|_{\infty,[0,x]} \dfrac{K_x}{n}  \exp \left[ \dfrac{n}{K_x} \ln \left( 1 - \frac{u_n}{x} \right)\right].
\end{align*}
Choosing $u_n = xK_x \alpha \dfrac{\ln n}{n}$ we have 
$$\dfrac{1}{n}  \exp \left[\dfrac{n}{K_x} \ln \left( 1 - \frac{u_n}{x} \right) \right] \underset{n\to \infty}{\sim}\dfrac{1}{n^{\alpha+1}} \underset{n\to \infty}{=} o \left(\mathbb{P} \left[ X_1 \leq \frac{1}{n} \right] \right).$$
This estimate, together with \eqref{eq:density_case_upper_bound_2} and \eqref{eq:density_case_upper_bound_3}, leads to \eqref{eq:density_case_upper_bound_4}.
Again the uniform upper bound has not been detailed but the same proof considering $\frac{\mathbb{P} \left[ S_{n+1} \leq t  \right]}{ \mathbb{P} \left[ S_n \leq t  \right] }$ with $0 < t \leq x $ gives the same upper bound with a constant $C_t \leq C_x$. 

This achieves the proof of Theorem \ref{theo:: class c_5}.
\end{proof}

\paragraph{Proof of Estimate \eqref{eq:integral_estim_proof_thm_density_case}}
\begin{proof}
Recall that $\delta > 0$. Hence its integer part $\lfloor\delta \rfloor$ is non-negative. 
Let us first assume that $\delta$ is not an integer. Then we prove by induction that for any $k \in \{0, \ldots , \lfloor \delta \rfloor \}$, 
\begin{align} \nonumber
    & \int_{1/n}^{x} y_1^{ \delta} \left(  \frac{x-y_1}{x} \right)^{ \frac{n}{K_x }  - 1  } dy_1  =  \sum_{j=0}^k \left( \prod_{i=0}^j (\delta - i) \left( \frac{n}{K_x} + i \right)^{-1} \right) \frac{x^{j+1}}{\delta - j}  \left( \frac{1}{n} \right)^{\delta - j} \left( 1 - \frac{1}{nx}  \right)^{\frac{n}{K_x} + j}  \\ \label{induc:: integration_formula}
    & \qquad \qquad + x^{k+1} \left( \prod_{i=0}^k (\delta - i) \left( \frac{n}{K_x} + i \right)^{-1} \right)  \int_{1/n}^{x} y_1^{ \delta - k - 1} \left(  \frac{x-y_1}{x} \right)^{ \frac{n}{K_x }  + k  } dy_1.
\end{align}
Indeed when $k=0$, we have
\begin{align*}
     \int_{1/n}^{x} y_1^{ \delta} \left(  \frac{x-y_1}{x} \right)^{ \frac{n}{K_x }  - 1  } dy_1 = & - \frac{xK_x}{n} \left[ y_1^{ \delta}  \left(  \frac{x-y_1}{x} \right)^{ \frac{n}{K_x }  }  \right]_{\frac{1}{n}}^{x} + \frac{\delta x K_x}{n} \int_{1/n}^{x} y_1^{ \delta - 1 } \left(  \frac{x-y_1}{x} \right)^{ \frac{n}{K_x }   } dy_1 \\
     = & \frac{xK_x}{n} \left( \frac{1}{n} \right)^{\delta} \left( 1 - \frac{1}{nx}  \right)^{\frac{n}{K_x}} + \frac{\delta x K_x}{n} \int_{1/n}^{x} y_1^{ \delta - 1 } \left(  \frac{x-y_1}{x} \right)^{ \frac{n}{K_x }   } dy_1.
\end{align*}
With an integration by part, we also have for any $k$:
\begin{align*}
    & \int_{1/n}^{x} y_1^{ \delta - k - 1} \left(  \frac{x-y_1}{x} \right)^{ \frac{n}{K_x }  + k  } dy_1 = x \left( \frac{n}{K_x} +k+1 \right)^{-1} \left( \dfrac{1}{n}\right)^{ \delta - (k +1)}  \left( 1- \dfrac{1}{nx} \right)^{ \frac{n}{K_x } + (k +1)  }   \\
    & \qquad \qquad +  x  \left( \frac{n}{K_x} +k+1 \right)^{-1}( \delta - (k + 1) ) \int_{1/n}^{x} y_1^{ \delta - (k+1) - 1 } \left(  \frac{x-y_1}{x} \right)^{ \frac{n}{K_x }  +k+1 } dy_1.
\end{align*}
The result is then deduced from this equality and the induction hypothesis. \\
In particular when $k = \lfloor \delta \rfloor$ we deduce that:
\begin{align}  \nonumber
    & \int_{1/n}^{x} y_1^{ \delta}  \left(  \frac{x-y_1}{x} \right)^{ \frac{n}{K_x }  - 1  } dy_1 
    = \sum_{j=0}^{\lfloor \delta \rfloor} \left( \prod_{i=0}^j (\delta - i) \left( \frac{n}{K_x} + i \right)^{-1} \right) \frac{x^{j+1}}{\delta - j} \left( \frac{1}{n} \right)^{\delta - j} \left( 1 - \frac{1}{nx}  \right)^{\frac{n}{K_x} + j} \\ \label{eq:formula_delta_estim}
    & \qquad \qquad + x^{\lfloor \delta \rfloor +1} \left( \prod_{i=0}^{\lfloor \delta \rfloor} (\delta - i) \left( \frac{n}{K_x} + i \right)^{-1} \right)  \int_{1/n}^{x} y_1^{ \delta - \lfloor \delta \rfloor - 1} \left(  \frac{x-y_1}{x} \right)^{ \frac{n}{K_x }  + \lfloor \delta \rfloor  } dy_1. 
\end{align}
Observe that for any $j \in \{0, \ldots , \lfloor \delta \rfloor  \}$,
$$
    \left( \frac{1}{n} \right)^{\delta - j} \left( 1 - \frac{1}{nx}  \right)^{\frac{n}{K_x} + j}  \leq \left( \frac{1}{n} \right)^{\delta - j} \left( 1 - \frac{1}{nx}  \right)^{\frac{n}{K_x} }  \leq \left( \frac{1}{n} \right)^{\delta - j} \exp \left( - \dfrac{1}{xK_x} \right)
$$
and 
$$
\prod_{i=0}^j (\delta - i) \left( \frac{n}{K_x} + i \right)^{-1} = n^{-j-1} \prod_{i=0}^j (\delta - i) \left( \frac{1}{K_x} + \dfrac{i}{n} \right)^{-1}  \leq (K_x)^{j+1} n^{-j-1}\prod_{i=0}^j (\delta - i).
$$
One can notice that $\delta - \lfloor \delta \rfloor - 1 < 0 $ and that $y_1 \mapsto y_1^{ \delta - \lfloor \delta \rfloor - 1} $ is a decreasing function. Thus
\begin{align*}
    \int_{1/n}^{x} y_1^{ \delta - \lfloor \delta \rfloor - 1} \left(  \frac{x-y_1}{x} \right)^{ \frac{n}{K_x }  + \lfloor \delta \rfloor  } dy_1 &  \leq \left( \dfrac{1}{n} \right)^{\delta - \lfloor \delta \rfloor - 1} \int_{1/n}^{x} \left(  \frac{x-y_1}{x} \right)^{ \frac{n}{K_x }  + \lfloor \delta \rfloor  } dy_1 \\
    & = \left( \dfrac{1}{n} \right)^{\delta - \lfloor \delta \rfloor - 1} \left( \frac{n}{K_x} + \lfloor \delta \rfloor +1 \right)^{-1} \left( 1 - \dfrac{1}{nx}\right)^{\frac{n}{K_x} + \lfloor \delta \rfloor +1}
\end{align*}
Therefore we deduce that all terms in \eqref{eq:formula_delta_estim} are controlled by $C_x n^{-\delta-1}$. 
If $\delta = \lfloor \delta \rfloor$ is an positive integer, then Formula \eqref{induc:: integration_formula} holds only for $k\in \{0,\ldots,\delta-1\}$ and with $k=\delta-1$, Equation \eqref{eq:formula_delta_estim} becomes:
\begin{align*} \nonumber
    & \int_{1/n}^{x} y_1^{ \delta} \left(  \frac{x-y_1}{x} \right)^{ \frac{n}{K_x }  - 1  } dy_1  =  \sum_{j=0}^{\delta-1} \left( \prod_{i=0}^j (\delta - i) \left( \frac{n}{K_x} + i \right)^{-1} \right) \frac{x^{j+1}}{\delta - j}  \left( \frac{1}{n} \right)^{\delta - j} \left( 1 - \frac{1}{nx}  \right)^{\frac{n}{K_x} + j}  \\ 
    & \qquad \qquad + x^{\delta} \left( \prod_{i=0}^{\delta-1} (\delta - i) \left( \frac{n}{K_x} + i \right)^{-1} \right)  \int_{1/n}^{x}  \left(  \frac{x-y_1}{x} \right)^{ \frac{n}{K_x }  + \delta-1  } dy_1\\
    & \qquad = \sum_{j=0}^{\delta} \left( \prod_{i=0}^{j-1} (\delta - i) \left( \frac{n}{K_x} + i \right)^{-1} \right) x^{j+1} \left( \frac{n}{K_x} + j \right)^{-1}  \left( \frac{1}{n} \right)^{\delta - j} \left( 1 - \frac{1}{nx}  \right)^{\frac{n}{K_x} + j}.
\end{align*}
 The arguments and the conclusion remain the same. 
\end{proof}

\paragraph{Proof of Proposition \ref{prop:class_C_6}}

\begin{proof}
We fix $x>0$.
Let $Atom = \{ x_i,i\in D \}$ be the set of all the atoms of the random variable $X_1$. Let $x_{\min} = \inf \{ x_i, i \in D \}$ and $M_{\max} = \left\lceil \frac{x}{x_{\min}} \right\rceil $. Under our setting $x_{\min} > 0$ and $M_{\max}$ is well defined. The density function of the absolutely continuous part of $X_1$ is denoted by $f$. 
For any $k \leq M_{\max}$, define $\Diamond_{k,x} = \{ (x_{i_1}, \dots ,x_{i_k}) \in Atom^k; \;  \; \sum_{j=1}^k x_{i_j} < x \}$, the set of all $k$-tuples of atoms of the distribution of $X_1$ such that $\forall n \geq M_{\max},\forall k \leq M_{\max}; \forall (x_{i_1}, \dots ,x_{i_k}) \in \Diamond_x$: 
$$
    \mathbb{P}[S_n \leq x; \exists j_1<\ldots<j_k \leq n; X_{j_1} = x_{i_1}, \ldots ,X_{j_k} = x_{i_k}] > 0 .
$$
Then it is clear, using the total probabilities formula that for all $n \geq M_{\max}$:
\begin{align*}
    \mathbb{P}[S_{n+1} \leq x] = \sum_{k=0}^{M_{\max}} \sum_{z \in \Diamond_{k,x}} \mathbb{P} \big[ \{ S_{n+1} \leq x\} & \cap \{ \exists j_1<\ldots<j_k \leq n+1; (X_{j_1}, \ldots ,X_{j_k}) = z \} \\
    & \cap \{ \forall i \leq n+1;\forall h \leq k; i \neq j_h; X_i \not\in Atom \} \big]
\end{align*}
Note the necessity to order the indexes $j_1<\ldots<j_k \leq n+1$. If not, since every element of $\Diamond_{k,x}$ on which we apply a permutation action of the terms of the tuple, is a different element of $\Diamond_{k,x}$, almost every term in the sum would be counted more than twice. \\
To simplify the formulas, let us re-write $\forall n \geq M_{\max},\forall k \leq M_{\max}; \forall z \in \Diamond_{k,x}$:
$$
    A_{n,k,z} = \{ \exists j_1<\ldots<j_k \leq n; (X_{j_1}  \ldots ,X_{j_k}) = z \} \cap \{ \forall i \leq n+1;\forall h \leq k; i \neq j_h; X_i \not\in Atom \}
$$
Then 
\begin{align} \nonumber
    \mathbb{P}[S_{n+1} \leq x] = & \sum_{k=0}^{M_{\max}} \sum_{z \in \Diamond_{k,x}} \mathbb{P} \big[ \{ S_{n+1} \leq x\} \cap A_{(n+1),k,z}  \big] \\ \label{eq:thm_2_9_eq_1}
    = & \sum_{k=0}^{M_{\max}} \sum_{z \in \Diamond_{k,x}}  \frac{\mathbb{P} \big[ \{ S_{n+1} \leq x\} \cap A_{(n+1),k,z}  \big]}{\mathbb{P} \big[ \{ S_{n} \leq x\} \cap A_{n,k,z}  \big]} \mathbb{P} \big[ \{ S_{n} \leq x\} \cap A_{n,k,z}  \big]
\end{align}
We want to obtain a uniform upper bound for $\frac{\mathbb{P} \big[ \{ S_{n+1} \leq x\} \cap A_{(n+1),k,z}  \big]}{\mathbb{P} \big[ \{ S_{n} \leq x\} \cap A_{n,k,z}  \big]}$ that does not depend on the tuple $z$. If $z =(x_{i_1}, \dots ,x_{i_k})$, the sum $\sum z$ of $z$ is equal $\sum_{j=1}^k x_{i_j}$. Remark that 
$$\mathbb{P} \big[ S_{n+1} \leq x | A_{(n+1),k,z}  \big] = \mathbb{P} \left[\sum_{i=1}^{n+1-k} X^{\text{cont}}_i \leq x - \sum z \right]$$
where $(X^{\text{cont}}_n, \ n\geq 1)$ is a sequence of i.i.d. absolutely continuous random variables with density $f$. We denote the corresponding sum by $S_n^{\text{cont}}$. Thus
\begin{align} \nonumber
    \frac{\mathbb{P} \big[ \{ S_{n+1} \leq x\} \cap A_{(n+1),k,z}  \big]}{\mathbb{P} \big[ \{ S_{n} \leq x\} \cap A_{n,k,z}  \big]}&  =  \frac{\mathbb{P} \big[ S_{n+1} \leq x | A_{(n+1),k,z}  \big] \mathbb{P} [A_{(n+1),k,z}] }{\mathbb{P} \big[  S_{n} \leq x | A_{n,k,z}  \big] \mathbb{P} [A_{n,k,z}]} \\ \nonumber
    & =  \frac{\mathbb{P} [S_{n+1-k}^{\text{cont}} \leq x - \sum z]}{\mathbb{P} [S_{n-k}^{\text{cont}} \leq x - \sum z]} \frac{ \mathbb{P} [A_{(n+1),k,z}] }{ \mathbb{P} [A_{n,k,z}]} \\ \nonumber
    & \leq  \sup_{t \in [0;x]} \left| \frac{\mathbb{P} [S_{n+1-k}^{\text{cont}} \leq t]}{\mathbb{P} [S_{n-k}^{\text{cont}} \leq t]} \right| \frac{ \mathbb{P} [A_{(n+1),k,z}] }{ \mathbb{P} [A_{n,k,z}]}  \\ \label{eq:thm_2_9_eq_2}
    & \leq  \max_{k \in \{1,...,M_{\max} \} } \left(  \sup_{t \in [0;x]} \left| \frac{\mathbb{P} [S_{n+1-k}^{\text{cont}} \leq t]}{\mathbb{P} [S_{n-k}^{\text{cont}} \leq t]} \right|  \frac{ \mathbb{P} [A_{(n+1),k,z}] }{ \mathbb{P} [A_{n,k,z}]}\right). 
\end{align}
The first quotient can be dominated, since $(X_n^{\text{cont}})_{n \in \mathbb{N}^*}$ lies in class $\mathcal{C}_4$ (using Theorem \ref{theo:: class c_5}): there exist a constant $C_x' > 0$ and an integer $N_x \in \mathbb N^*$ such that for all $n \geq N_x$
$$
    \sup_{t \in [0,x]} \left| \frac{\mathbb{P} \left[ S_{n+1}^{\text{cont}} \leq t  \right]}{ \mathbb{P} \left[ S_n^{\text{cont}} \leq t  \right] } \right|  \leq C_x'  \mathbb P \left[ X^{\text{cont}}_1 \leq \dfrac{1}{n} \right].
$$
Concerning the second quotient, the probability of the events $A_{n,k,z}$ and $A_{(n+1),k,z}$ can be computed as it follows: if $z=(x_{i_1}, \dots ,x_{i_k})$,
\begin{align*}
   & \mathbb{P} [A_{n,k,z}] 
     =  \sum_{1 \leq j_1<\ldots<j_k \leq n } \mathbb{P} [ \{ (X_{j_1}  \ldots ,X_{j_k}) = z \} \cap \{ \forall i \leq n;\forall h \leq k; i \neq j_h; X_i \not\in Atom \} ] \\
    & \quad =  \sum_{1 \leq j_1<\ldots<j_k \leq n } \mathbb{P} [X_1 = X_1^{\text{cont}}]^{n-k} \prod_{h=1}^k \mathbb{P} [X_1 = x_{i_h}] = \binom{n}{k}  \mathbb{P} [X_1 = X_1^{\text{cont}}]^{n-k} \prod_{h=1}^k \mathbb{P} [X_1 = x_{i_h}]. 
\end{align*}
Consequently for all $n\geq M_{\max}$ and $k\leq M_{\max}$
\begin{eqnarray*}
    \frac{ \mathbb{P} [A_{(n+1),k,z}] }{ \mathbb{P} [A_{n,k,z}]} = \frac{n+1}{n} \frac{n-k}{n+1-k} \mathbb{P} [X_1 = X_1^{\text{cont}}] \leq 2
\end{eqnarray*}
and for all  $n \geq N_x + M_{\max} $
$$ \max_{k \in \{1,...,M_{\max} \} } \left(  \sup_{t \in [0;x]} \left| \frac{\mathbb{P} [S_{n+1-k}^{\text{cont}} \leq t]}{\mathbb{P} [S_{n-k}^{\text{cont}} \leq t]} \right|  \frac{ \mathbb{P} [A_{(n+1),k,z}] }{ \mathbb{P} [A_{n,k,z}]}\right) \leq 2 C_x'  \mathbb P \left[ X^{\text{cont}}_1 \leq \dfrac{1}{n-M_{\max}} \right]  .$$
Coming back to \eqref{eq:thm_2_9_eq_1}, together with \eqref{eq:thm_2_9_eq_2}, we obtain
for any $n \geq N_x + M_{\max} $
\begin{align*} \nonumber
    \mathbb{P}[S_{n+1} \leq x] & \leq 2 C_x' \mathbb P \left[ X^{\text{cont}}_1 \leq \dfrac{1}{n-M_{\max}} \right]   \sum_{k=0}^{M_{\max}} \sum_{z \in \Diamond_{k,x}}  \mathbb{P} \big[ \{ S_{n} \leq x\} \cap A_{n,k,z}  \big] \\
    & \leq 2 C_x' \mathbb P \left[ X^{\text{cont}}_1 \leq \dfrac{1}{n-M_{\max}} \right] \mathbb{P} \big[ S_{n} \leq x  \big] \\
    & \leq 2 C_x' \mathbb P \left[ X^{\text{cont}}_1 \leq \dfrac{N_x + M_{\max}}{N_x} \dfrac{1}{n} \right] \mathbb{P} \big[  S_{n} \leq x  \big] \\
    & \leq 4 C_x' \left(\dfrac{N_x + M_{\max}}{N_x} \right)^\alpha \mathbb P \left[ X^{\text{cont}}_1 \leq \dfrac{1}{n} \right] \mathbb{P} \big[  S_{n} \leq x  \big],
\end{align*}
where the last inequality is provided by the regular variation property.
This achieves the proof of the proposition. 
\end{proof}

\subsubsection{Lower bound and equivalence for the class $\mathcal C_4$}

Now we prove that the ratio sequences are bounded from below (Theorem \ref{theo::lower_bound_continuous}) and show the equivalence announced in Theorem \ref{theo:: equivalence_continuous_case}. Let us begin with the proof of the lower bound.

\begin{proof}
The remainder of the proof is given as follows: We first use the regular variation property of the density $f$ and the continuity of $f$ in $]0,+ \infty[$ to prove the existence of $\kappa_x >0$ such that $\forall (t,c) \in ]0,x] \times ]0,1[ $, $\mathbb{P} [X_1 \leq ct] \geq c^{\kappa_x} \mathbb{P} [X_1 \leq t] $. We then use this inequality to prove the proposed lower bound for the sequence $(c_{n,x})_{n \in \mathbb{N}}$. In this proof, we fix $x >0$.

\smallskip

\noindent \underline{\textbf{Step 1:}}
Since $(X_n)_{n \in \mathbb{N}^*} \in \mathcal C_4$, the density function $f$ is of class $C^{1} (]0,+\infty[) $ and there exists $\alpha > 0$, such that $f$ is regularly varying of index $\alpha -1$ at $0$. Due to Karamata's theorem (see \cite[Proposition 1.5.8]{bing:gold:teug:89}) we deduce that the distribution function $F$ of $X_1$ is regularly varying of index $\alpha$ at 0. Since $f'$ is supposed to be monotone in a neighborhood of $0$, it is also the case for $f$ and due to Corollary \ref{coro:: regular_var}, we deduce that for any $\beta > \alpha$, there exists $\varepsilon  >0$ such that $t \mapsto t^{-\beta}F(t) $ is non-increasing on $]0,\varepsilon]$. Thus
\begin{equation*}
    \forall t \in ]0,\varepsilon], \forall c \in ]0,1[, \quad \quad  t^{-\beta}F(t) \leq (ct)^{-\beta}F(ct) \quad \iff \quad  c^{\beta}F(t) \leq F(ct).
\end{equation*}
Consider such $\beta$ and $\varepsilon$. Since $f$ is of class $C^{1}([\varepsilon,x])$, we have $||f||_{\infty, [\varepsilon,x]} < \infty$. We prove that $c^{\kappa_x}F(t) \leq F(ct)$, $\forall (t,c) \in ]0,x] \times ]0,1[ $,  with $\kappa_x = \max \left\{ \beta, \frac{x ||f||_{\infty, [\varepsilon,x]}}{ F(\varepsilon)} + 1 \right\}$. To do so, we define the following functions for a fixed  $t \in [\varepsilon,x]$: 
\begin{align*}
    a_t : [\varepsilon,t] \xrightarrow{} & [F(\varepsilon), F(t)] \\
     y \mapsto & \left\{ \begin{matrix} F(\varepsilon) & \quad \quad \text{if} \quad \quad y \leq \varpi_{\varepsilon,t} \\
     F(t) + (y-t) ||f||_{\infty, [\varepsilon,x]} & \quad \quad \text{if} \quad \quad y > \varpi_{\varepsilon,t}
    \end{matrix} \right. 
    \\
    b_t : [0,t] \xrightarrow{} & [0, F(t)] \\
     y \mapsto & \left( \frac{y}{t} \right)^{\kappa_x} F(t)
\end{align*}
where
$$\varpi_{\varepsilon,t} = t + \dfrac{F(\varepsilon) - F(t)}{ ||f||_{\infty, [\varepsilon,x]}} =  \dfrac{F(\varepsilon) - F(t) + t ||f||_{\infty, [\varepsilon,x]}}{ ||f||_{\infty, [\varepsilon,x]}} \leq t.$$
Due to the mean value theorem, it is easy to verify that 
$\varepsilon \leq\varpi_{\varepsilon,t}$ and 
$F(y) \geq a_t(y), \forall y \in [\varepsilon,t]$. We now prove that $b_t(y) \leq a_t(y), \forall y \in [\varepsilon,t]$. $a_t$ is a piecewise linear function and $b_t$ is a strictly convex function. Due to those analytical properties, the graphs of $a_t$ and $b_t$ have at most two points of intersection. We have $a_t(t) = b_t(t)$. Let us assume for a while that $b_t \left(\varpi_{\varepsilon,t} \right) \leq  a_t \left(\varpi_{\varepsilon,t} \right)$. Since $a_t$ is constant before $\varpi_{\varepsilon,t}$ and $b_t$ is strictly increasing on its domain of definition, we deduce that $b_t(y) \leq a_t(y), \forall y \in \left[ \varepsilon, \kappa_{\varepsilon,t} \right]$. The inequality $b_t(y) \leq a_t(y)$ holds on $\left[ \kappa_{\varepsilon,t} , t \right]$ since the first point of intersection of $a_t$ and $b_t$ is $t$. We need to prove:
$$
b_t \left(\varpi_{\varepsilon,t} \right) 
    \leq  a_t \left(\varpi_{\varepsilon,t} \right) = F(\varepsilon).
$$
We have
\begin{align*}
    \log \left( \frac{b_t \left(\varpi_{\varepsilon,t} \right)}{F(\varepsilon)} \right) & = \kappa_x \log \left( 1 - \frac{F(t) - F(\varepsilon)}{t ||f||_{\infty, [\varepsilon,x]}} \right) + \log \left( \frac{F(t)}{F(\varepsilon)} \right) \\
    & = \kappa_x \log \left( 1 - \frac{F(t) - F(\varepsilon)}{t ||f||_{\infty, [\varepsilon,x]}} \right) + \log \left(1 + \frac{F(t) - F(\varepsilon)}{F(\varepsilon)} \right).
\end{align*}
Since $\log(1+x) \leq x, \forall x > 0$, we obtain
\begin{align*}
    \log \left( \frac{b_t \left(\varpi_{\varepsilon,t} \right)}{F(\varepsilon)} \right) & \leq - \kappa_x \frac{F(t) - F(\varepsilon)}{t ||f||_{\infty, [\varepsilon,x]}}  + \frac{F(t) - F(\varepsilon)}{F(\varepsilon)} \\
    & \leq -\left( \frac{x ||f||_{\infty, [\varepsilon,x]}}{ F(\varepsilon)} + 1 \right) \frac{F(t) - F(\varepsilon)}{t ||f||_{\infty, [\varepsilon,x]}} + \frac{F(t) - F(\varepsilon)}{F(\varepsilon)} \\
    & = \left( 1 - \frac{x}{t} \right) \frac{F(t) - F(\varepsilon)}{F(\varepsilon)} - \frac{F(t) - F(\varepsilon)}{t ||f||_{\infty, [\varepsilon,x]}} \leq 0.
\end{align*}
Therefore we have: $b_t(y) \leq a_t(y) \leq F(y), \forall \varepsilon \leq y \leq t$. To end the proof of the proposed inequality, we need to verify if $b_t(y) \leq F(y), \forall 0 < y \leq \varepsilon$. We already have $c^{\beta}F(t) \leq F(ct), \forall t \in ]0,\varepsilon], \forall c \in ]0,1[$, which can be reformulated as follows
\begin{equation*}
     \left( \frac{y_1}{y_2} \right)^{\beta} \leq \frac{F(y_1)}{F(y_2)}, \quad \quad \forall 0 < y_1 \leq y_2 \leq \varepsilon.
\end{equation*}
Let $y \in ]0,\varepsilon]$, $t \in [\varepsilon,x]$. We have
$$
    \frac{b_t(y) }{F(t)} = \left( \frac{y}{t} \right)^{\kappa_x} = \left( \frac{y}{\varepsilon} \right)^{\kappa_x} \left( \frac{\varepsilon}{t} \right)^{\kappa_x} \leq \left( \frac{y}{\varepsilon} \right)^{\beta} \left( \frac{\varepsilon}{t} \right)^{\kappa_x} \leq \frac{F(y)}{F(\varepsilon)} \frac{F(\varepsilon)}{F(t)},
$$
which proves that $b_t(y) \leq  F(y)$ on $[0,t]$. 

\smallskip

\noindent \underline{\textbf{Step 2:}} 
We prove by induction on $k \in \mathbb{N}^*$, that for any $0< t_1 < t_2 < x$, 
$$\frac{ \mathbb{P} [S_k \leq t_2 - t_1 ]}{\mathbb{P} [S_k \leq t_2 ]} \geq \left(1 - \frac{t_1}{t_2} \right)^{k \kappa_x}.$$
Let $0< t_1 < t_2 < x$. For $k=1$ this inequality follows from Step 1. 
Assume that the property holds for any $1 \leq m \leq k$. Then
\begin{align*}
    \mathbb{P} [S_{k+1} \leq t_2 - t_1 ] = 
    & \int_0^{t_2 - t_1} f(t) \mathbb{P} [S_{k} \leq t_2 - t_1 - t ] dt 
    = \int_0^{t_2 - t_1} f(t) \mathbb{P} \left[ S_{k} \leq \frac{t_2 - t_1}{t_2} \left( t_2 - \frac{t_2 }{t_2- t_1} t \right) \right] dt \\
    \geq 
    & \int_0^{t_2 - t_1} f(t) \mathbb{P} \left[ S_{k} \leq t_2 - \frac{t_2 }{t_2- t_1}  t  \right]  
    \left( \frac{t_2 - t_1}{t_2} \right)^{k \kappa_x} dt
\end{align*}
where we used the property for $S_k$ with $0<  \frac{t_1}{t_2} \left( t_2 - \frac{t_2 }{t_2- t_1} t \right) <  t_2 - \frac{t_2 }{t_2- t_1}  t < x$. 
By an integration by part, since $\mathbb{P}[X_1 \leq 0] = \mathbb{P}[S_k \leq 0] = 0 $, we have
\begin{align*}
    \int_0^{t_2 - t_1} f(t) \mathbb{P} & \left[ S_{k} \leq t_2 -   \frac{t_2 }{t_2- t_1}  t  \right]  \left( \frac{t_2 - t_1}{t_2} \right)^{k \kappa_x} dt \\
    & \quad = \int_0^{t_2 - t_1} \mathbb{P} [X_1 \leq t ] f_{S_k} \left(t_2 - \frac{t_2 }{t_2- t_1}  t  \right) \left( \frac{t_2 - t_1}{t_2} \right)^{k \kappa_x - 1} dt \\
    & \quad \geq \int_0^{t_2 - t_1} \mathbb{P} \left[ X_1 \leq \frac{t_2 }{t_2- t_1} t \right] f_{S_k} \left(t_2 - \frac{t_2 }{t_2- t_1}  t  \right) \left( \frac{t_2 - t_1}{t_2} \right)^{(k+1) \kappa_x - 1} dt \\
    & \quad =  \int_0^{t_2 } \mathbb{P} \left[ X_1 \leq u \right] f_{S_k} \left(t_2 - u  \right) \left( \frac{t_2 - t_1}{t_2} \right)^{(k+1) \kappa_x } du  \\
    & \quad = \left( \frac{t_2 - t_1}{t_2} \right)^{(k+1) \kappa_x } \mathbb{P} [S_{k+1} \leq t_2 ]
\end{align*}
where the inequality follows from the property applied for $k=1$ and $0 < \frac{t_1}{t_2-t_1} t < \frac{t_2}{t_2-t_1} t< x$. 
We finally obtain
\begin{align*}
    \frac{\mathbb{P} [S_{k+1} \leq t_2 - t_1 ]}{\mathbb{P} [S_{k+1} \leq t_2 ]} \geq \left( 1 -  \frac{ t_1}{t_2} \right)^{(k+1) \kappa_x } ,
\end{align*}
which ends the proof of the induction. 

\smallskip

\noindent \underline{\textbf{Step 3:}} 
Let $n \in \mathbb{N}^*$. 
To conclude the proof, we have
\begin{align*}
    \frac{\mathbb{P} \left[ S_{n+1} \leq x  \right]}{ \mathbb{P} \left[ S_n \leq x  \right] } & \geq \frac{\mathbb{P} \left[ S_{n} \leq x - \frac{1}{n}   \right] \mathbb{P} \left[ X_{n+1} \leq  \frac{1}{n}   \right]}{ \mathbb{P} \left[ S_n \leq x  \right] } \\
    & \geq \mathbb{P} \left[ X_{1} \leq  \frac{1}{n}   \right] \left( 1 -  \frac{ 1}{n x} \right)^{(n+1) \kappa_x }
\end{align*}
We have $\underset{n \xrightarrow{} \infty}{ \lim } \left( 1 -  \frac{ 1}{n x} \right)^{(n+1) \kappa_x } = \exp \left( - \frac{\kappa_x}{x} \right)$. We finally obtain for n sufficiently large
\begin{align*}
    \frac{\mathbb{P} \left[ S_{n+1} \leq x  \right]}{ \mathbb{P} \left[ S_n \leq x  \right] } \geq \mathbb{P} \left[ X_{1} \leq  \frac{1}{n}   \right] \frac{1}{2} \exp \left( - \frac{\kappa_x}{x} \right).
\end{align*}
\end{proof}

\paragraph{Proof of Theorem \ref{theo:: equivalence_continuous_case}}

In order to enlighten the proof, we introduce the following lemma.
\begin{lemma} \label{lem: equivalence_inequ}
Consider two sequences of positive real numbers $(a_n)_{n \in \mathbb{N}}$ and $(b_n)_{n \in \mathbb{N}}$. Assume that $a_n \underset{n \to + \infty}{\sim} b_n$. Moreover assume that there exists two more sequences of positive real numbers $(c_n)_{n \in \mathbb{N}}$ and $(d_n)_{n \in \mathbb{N}}$ such that $c_n \underset{n \to + \infty}{\sim} d_n$ and for $n$ sufficiently large, 
$c_n \leq a_n$ and $b_n \leq d_n$.
Then, $a_n \underset{n \to + \infty}{\sim} c_n$.
\end{lemma}

\begin{proof}
Due to the assumptions made on the sequences, we deduce the existence of $N \in \mathbb{N}$ such that for any $n \geq N$, $c_n \leq a_n$ and $b_n \leq d_n$. We consequently deduce from these inequalities that for $n \geq N$,
$1 \; \leq  \; \dfrac{a_n}{c_n} \; \leq \; \dfrac{a_n}{b_n} \dfrac{d_n}{c_n}$.
Finally, $a_n \underset{n \to + \infty}{\sim} b_n$ and $c_n \underset{n \to + \infty}{\sim} d_n$, which implies that the right side term converges to $1$, which ends the proof.
\end{proof}

\begin{proof}[Proof of the theorem]
The remainder of the proof is given as follows. We first prove that 
\begin{align*}
    \mathbb{P} [S_n \leq x ] \underset{n \to + \infty}{\sim} \mathbb{P} [S_n \leq x, X_i \leq \varepsilon_n, \forall i \in \{1, \ldots , n \} ],
\end{align*}
where $(\varepsilon_n)_{n \in \mathbb{N}^*} = \left( \dfrac{(\alpha + 2) x K_x \log (n)}{ n} \right)_{n \in \mathbb{N}^*}$. Recall that  $K_x$ is the minimal integer for which, $f_{S_{K_x}}$ is non-decreasing on $[0,x]$. 

In a second step, from the regular variation assertion on the density function of $X_1$, we deduce the existence of two sequences $(\beta_n^+)_{n \in \mathbb{N}^*}$ and $(\beta_n^-)_{n \in \mathbb{N}^*}$ respectively non-decreasing and non-increasing that converge to $\alpha$ such that for $n$ sufficiently large:
\begin{align} \label{eq:lestim_equivalence1}
    \int_0^{\varepsilon_{n}} f(t) \left( \dfrac{x-t}{x} \right)^{n \beta_n^-} dt & \leq \dfrac{\mathbb{P} [S_{n+1} \leq x, X_i \leq \varepsilon_{n}, \forall i \in \{1, \ldots , n +1 \} ]}{\mathbb{P} [S_n \leq x, X_i \leq \varepsilon_n, \forall i \in \{1, \ldots , n \} ]} \\
    \label{eq:uestim_equivalence1}
 \int_0^{\varepsilon_{n}} f(t) \left( \dfrac{x-t}{x} \right)^{n \beta_n^+} dt & \geq \dfrac{\mathbb{P} [S_{n+1} \leq x, X_i \leq \varepsilon_n, \forall i \in \{1, \ldots , n +1 \} ]}{\mathbb{P} [S_n \leq x ]}.
\end{align}
In the third step, we prove that
\begin{equation} \label{eq:equivalence2}
    \int_0^{\varepsilon_{n}} f(t) \left( \dfrac{x-t}{x} \right)^{n \beta_n^-} dt \underset{n \to + \infty}{\sim} \int_0^{\varepsilon_{n}} f(t) \left( \dfrac{x-t}{x} \right)^{n \beta_n^+} dt \underset{n \to + \infty}{\sim}  \int_0^{\varepsilon_{n}} f(t) \left( \dfrac{x-t}{x} \right)^{n \alpha} dt.
\end{equation}
In the last step, applying Lemma \ref{lem: equivalence_inequ}, one first deduce that 
$$\dfrac{\mathbb{P} [S_{n+1} \leq x ]}{\mathbb{P} [S_n \leq x ]}   \underset{n \to + \infty}{\sim} \int_0^{\varepsilon_{n}} f(t) \left( \dfrac{x-t}{x} \right)^{n \alpha} dt,$$
and then 
that for any function $l$ that diverges to the infinity at infinity and which verifies $l(n)/n \underset{n \to + \infty}{=} o(\varepsilon_n)$, we have 
\begin{align*}
    \int_0^{\varepsilon_{n}} f(t) \left( \dfrac{x-t}{x} \right)^{n \alpha} dt \underset{n \to + \infty}{\sim} \int_0^{\frac{l(n)}{n}} f(t) \left( \dfrac{x-t}{x} \right)^{n \alpha} dt.
\end{align*}

\medskip
\noindent \underline{\textbf{Step 1:}} As explained in the remainder of the proof, the first key step in deriving the asymptotic equivalent of the sequence $(c_{n+1,x}/c_{n,x})_{n \in \mathbb{N}}$ consists in showing that, for any $x > 0$, 
\begin{align*}
    \mathbb{P} [S_n \leq x ] \underset{n \to + \infty}{\sim} \mathbb{P} [S_n \leq x, X_i \leq \varepsilon_n, \forall i \in \{ 1, \ldots , n \} ],
\end{align*}
with $\varepsilon_n = \frac{A \log (n)}{n}$ and $A = (\alpha+2)xK_x$. Let $n \geq 2$. We can split the probability $\mathbb{P} [S_{n+1} \leq x ]$ as the sum
\begin{align*}
    \mathbb{P} [S_{n+1} \leq x, X_i \leq \varepsilon_{n+1}, \forall i \in \{1, \ldots , n +1 \} ] + \mathbb{P} [S_{n+1} \leq x, \exists i \in \{1, \ldots , n +1 \}, X_i > \varepsilon_{n+1} ].
\end{align*}
We also have
\begin{align*}
    \dfrac{\mathbb{P} [S_{n+1} \leq x, \exists i \in \{1, \ldots , n +1 \}, X_i > \varepsilon_{n+1} ] }{\mathbb{P} [S_n \leq x ] } & \leq \sum_{k=1}^{n+1} \dfrac{\mathbb{P} [S_{n+1} \leq x, X_k > \varepsilon_{n+1} ] }{\mathbb{P} [S_n \leq x ] } \\ 
    & \leq \sum_{k=1}^{n+1} \dfrac{\mathbb{P} [S_{n+1} - X_k \leq x - \varepsilon_{n+1} ] \mathbb{P} [X_k > \varepsilon_{n+1} ] }{\mathbb{P} [S_n \leq x ] } \\
    & \leq (n+1) \dfrac{\mathbb{P} [S_{n} \leq x - \varepsilon_{n+1} ] }{\mathbb{P} [S_n \leq x ] }.
\end{align*}
It has been justified in the proof\footnote{It is proved for a fixed $0<t<x$ instead of $\varepsilon_{n+1}$. But straightforward modifications can be done to extend the result, since the threshold $K_x$ just depends on $x$, neither on $t$ nor on $\varepsilon_{n+1}$.} of Theorem \ref{theo:: class c_5} (see Inequality (\ref{eq::expo_bound_ratio})) that there exists $K_x \in \mathbb{N}^*$ such that for $n \geq K_x$ which also verifies $\varepsilon_{n+1} \leq x/2$,
$$  \dfrac{\mathbb{P} [S_{n} \leq x - \varepsilon_{n+1} ] }{\mathbb{P} [S_n \leq x ] } \leq \left( \dfrac{x - \varepsilon_{n+1} }{x} \right)^{n/K_x -1} . $$
We define 
$$N_1 = \min \{ n \geq K_x \; | \; \forall k \geq n , \varepsilon_{n} \leq x/2 \}.$$ 
We consequently deduce that for any $n \geq N_1$,
\begin{align*}
    \dfrac{\mathbb{P} [S_{n+1} \leq x, \exists i \in \{1, \ldots , n +1 \}, X_i > \varepsilon_{n+1} ] }{\mathbb{P} [S_n \leq x ] } \leq (n+1) \left( \dfrac{x - \varepsilon_{n+1}}{x} \right)^{n/K_x -1} .
\end{align*}
Note that with $A=(\alpha+2)xK_s$
\begin{align*}
    \left( \dfrac{x - \varepsilon_{n+1}}{x} \right)^{n/K_x -1} & =  \left( 1 - \dfrac{A \log (n+1) }{(n+1)x} \right)^{n/K_x -1} \\
    & \leq \left( 1 - \dfrac{A \log (n) }{nx} \right)^{- 1/K_x -1} \exp \left( - \dfrac{ A \log(n+1) }{ xK_x}  \right) \\
    & \leq 2^{ 1/K_x + 1} (n+1)^{-\alpha - 2}.
\end{align*}
We know from \cite[Proposition 1.5.8]{bing:gold:teug:89} that $t \mapsto \mathbb{P} [X_1 \leq t ]$ is a regularly varying function of index $\alpha$ (see also Estimate \eqref{eq:mono_dens_equiv}), so that 
\begin{align*}
    (n+1) \left( \dfrac{x - \varepsilon_{n+1}}{x} \right)^{n/K_x -1}  \leq  2^{ 1/K_x + 1} (n+1)^{-\alpha - 1} \underset{n \to + \infty}{ = } o \left( \mathbb{P} \left[ X_{1} \leq  \frac{1}{n}   \right]  \right).
\end{align*}

To conclude this first step, recall that due to Theorems \ref{theo:: class c_5} and \ref{theo::lower_bound_continuous} that there exist two constants $0< c_x < C_x$ such that for $n$ sufficiently large, 
\begin{align*}
    c_x \mathbb{P} \left[ X_{1} \leq  \frac{1}{n}   \right] \leq  \frac{\mathbb{P} [S_{n+1} \leq x ] }{\mathbb{P} [S_n \leq x ] }  \leq C_x \mathbb{P} \left[ X_{1} \leq  \frac{1}{n}   \right].
\end{align*}
Hence
\begin{align*}
    \frac{\mathbb{P} [S_{n+1} \leq x ] }{\mathbb{P} [S_n \leq x ] } & = \frac{\mathbb{P} [S_{n+1} \leq x, X_i \leq \varepsilon_{n+1}, \forall i \in \{1, \ldots , n +1 \} ] }{\mathbb{P} [S_n \leq x ] } \\
    & +\dfrac{ \mathbb{P} [S_{n+1} \leq x, \exists i \in \{1, \ldots , n +1 \}, X_i > \varepsilon_{n+1} ] }{\mathbb{P} [S_n \leq x ] } \\
    & = \frac{\mathbb{P} [S_{n+1} \leq x, X_i \leq \varepsilon_{n+1}, \forall i \in \{1, \ldots , n +1 \} ] }{\mathbb{P} [S_n \leq x ] } + o \left( \mathbb{P} \left[ X_{1} \leq  \frac{1}{n}   \right]  \right) \\
    & =  \frac{\mathbb{P} [S_{n+1} \leq x, X_i \leq \varepsilon_{n+1}, \forall i \in \{1, \ldots , n +1 \} ] }{\mathbb{P} [S_n \leq x ] } + o \left( \frac{\mathbb{P} [S_{n+1} \leq x ] }{\mathbb{P} [S_n \leq x ] } \right).
\end{align*}
We deduce that
\begin{align*}
    \dfrac{\mathbb{P} [S_{n+1} \leq x, \exists i \in \{1, \ldots , n+1 \}, X_i > \varepsilon_{n+1} ] }{\mathbb{P} [S_n \leq x ] } \underset{n \to + \infty}{ = } o \left(  \frac{\mathbb{P} [S_{n+1} \leq x ] }{\mathbb{P} [S_n \leq x ] } \right).
\end{align*}
By multiplying both terms by $\mathbb{P} [S_n \leq x ] \neq 0$, we obtain the wanted asymptotic equivalent
\begin{align*}
    \mathbb{P} [S_{n+1} \leq x ] \underset{n \to + \infty}{\sim} \mathbb{P} [S_{n+1} \leq x, X_i \leq \varepsilon_{n+1}, \forall i \in \{1, \ldots , n+1 \} ].
\end{align*}
Using the same arguments of proof, one can consider $\varepsilon_n$ instead of $\varepsilon_{n+1}$ in the considered events to prove that
\begin{align*}
    \mathbb{P} [S_{n+1} \leq x ] \underset{n \to + \infty}{\sim} \mathbb{P} [S_{n+1} \leq x, X_i \leq \varepsilon_{n}, \forall i \in \{1, \ldots , n+1 \} ].
\end{align*}
The equivalence with a shift on $-1$ in the index of the sequence $(\varepsilon_n)_{n \geq N_1}$ will later simplify the arguments of proof. 

\medskip
\noindent \underline{\textbf{Step 2:}} 
Consider now $n \geq N_1$. We define the real numbers
\begin{align*}
    \beta_n^+ = \alpha - \sup_{t \in \left] 0 , \frac{x - \varepsilon_{n}}{x} \varepsilon_{n} \right] } \left( \left| \frac{tf'(t)}{f(t)} - (\alpha - 1)  \right| \right) \quad \quad \text{and} \quad \quad \beta_n^- = \alpha + \sup_{t \in \left] 0 , \frac{x - \varepsilon_{n}}{x} \varepsilon_{n} \right] } \left( \left| \frac{tf'(t)}{f(t)} - (\alpha - 1)  \right| \right).
\end{align*}
Proposition \ref{prop:slowly_varying} justified the limit $\frac{tf'(t)}{f(t)} \underset{t \to 0^+}{\to} \alpha - 1$. Since $ \frac{x - \varepsilon_{n}}{x} \varepsilon_{n}$ converges to zero at infinity, the sequences $(\beta_n^+)_{n \geq N_1}$ and $(\beta_n^-)_{n \geq N_1}$ are well-defined, with $\beta_n^+ \leq \alpha \leq \beta_n^-$, and converge to $\alpha$. The definition of those sequences moreover implies the monotony of both of them, since $\varepsilon_n \leq x/2$ for $n\geq N_1$. 

The definition of the sequences $(\beta_n^+)_{n \geq N_1}$ and $(\beta_n^-)_{n \geq N_1}$ does not exclude the possibility for those sequences to simultaneously value $\alpha$ from a certain index. Hence we distinguish two cases. 

\smallskip
\noindent \underline{\textit{Case 2.1.}} There exists a neighborhood of zero such that on this neighborhood the density function verifies:  $f(t) = B t^{\alpha -1}$ for some fixed positive real number $B$. This property is equivalent to the existence of some threshold $N_2 \geq N_1$ such that $\frac{tf'(t)}{f(t)} = \alpha -1$ for any $t \in \left]0, \frac{x}{x - \varepsilon_{N_2}} \varepsilon_{N_2} \right]$. In this case the sequences $(\beta_n^+)_{n \geq N_1}$ and $(\beta_n^-)_{n \geq N_1}$ become constant, equal to $\alpha$, up to this rank. 

Then we have for any $n \geq N_2$ and $t \in (0,\varepsilon_n)$
\begin{align*}
    & \mathbb{P} [S_n \leq x -t, X_i \leq \varepsilon_n, \forall i \in \{1, \ldots , n \} ] \\
    & \quad =  \int_{(\mathbb{R}_+)^n} \mathds{1}_{\{ x_1 + \ldots + x_n \leq x - t , x_i \leq \varepsilon_n, \forall i \in \{1,\ldots,n \} \} }  \prod_{i=1}^n f(x_i) dx_1 \ldots dx_n \\
    & \quad =  \int_{(\mathbb{R}_+)^n} \mathds{1}_{\{ x_1 + \ldots + x_n \leq x - t , x_i \leq \varepsilon_n, \forall i \in \{1,\ldots,n \} \} }  \prod_{i=1}^n B.x_i^{\alpha -1} dx_1 \ldots dx_n.
\end{align*}
The second equality is due to the formula of the density function on the subset $\left]0, \frac{x}{x - \varepsilon_{N_2}} \varepsilon_{N_2} \right]$ and the monotony of the sequence $(\varepsilon_n)_{n \geq 3}$.
Applying the change of variables $x_i = \dfrac{x-t}{x} y_i, i = 1 , \ldots , n$, we obtain 
\begin{align*}
    & \mathbb{P} [S_n \leq x - t, X_i \leq \varepsilon_n,  \forall i \in \{1, \ldots , n \} ] \\
    & \quad = \left( \dfrac{x-t}{x} \right)^n \int_{(\mathbb{R}_+)^n} \mathds{1}_{\{ y_1 + \ldots + y_n \leq x , y_i \leq \frac{x}{x-t}\varepsilon_n, \forall i \in \{1,\ldots,n \} \} }  \prod_{i=1}^n B . \left(\dfrac{x-t}{x} y_i \right)^{\alpha -1} dy_1 \ldots dy_n \\
    & \quad = \left( \dfrac{x-t}{x} \right)^{n +n(\alpha -1)} \int_{(\mathbb{R}_+)^n} \mathds{1}_{\{ y_1 + \ldots + y_n \leq x , y_i \leq \frac{x}{x-t}\varepsilon_n, \forall i \in \{1,\ldots,n \} \} }  \prod_{i=1}^n B.y_i^{\alpha -1} dy_1 \ldots dy_n \\
    & \quad = \left( \dfrac{x-t}{x} \right)^{n\alpha } \int_{(\mathbb{R}_+)^n} \mathds{1}_{\{ y_1 + \ldots + y_n \leq x , y_i \leq \frac{x}{x-t}\varepsilon_n, \forall i \in \{1,\ldots,n \} \} }  \prod_{i=1}^n f \left( y_i \right) dy_1 \ldots dy_n \\
    & \quad \leq \left( \dfrac{x-t}{x} \right)^{n\alpha } \int_{(\mathbb{R}_+)^n} \mathds{1}_{\{ y_1 + \ldots + y_n \leq x \} }  \prod_{i=1}^n f \left( y_i \right) dy_1 \ldots dy_n = \left( \dfrac{x-t}{x} \right)^{n\alpha } \mathbb{P} [S_{n} \leq x ].
\end{align*}
The same way, 
\begin{align*}
    & \mathbb{P} [S_n \leq x - t, X_i \leq \varepsilon_n,  \forall i \in \{1, \ldots , n \} ] \\
    &\quad = \left( \dfrac{x-t}{x} \right)^{n\alpha } \int_{(\mathbb{R}_+)^n} \mathds{1}_{\{ y_1 + \ldots + y_n \leq x , y_i \leq \frac{x}{x-t}\varepsilon_n, \forall i \in \{1,\ldots,n \} \} }  \prod_{i=1}^n f \left( y_i \right) dy_1 \ldots dy_n \\
    & \quad \geq \left( \dfrac{x-t}{x} \right)^{n\alpha } \int_{(\mathbb{R}_+)^n} \mathds{1}_{\{ y_1 + \ldots + y_n \leq x, y_i \leq \varepsilon_n, \forall i \in \{1,\ldots,n \} \} }  \prod_{i=1}^n f \left( y_i \right) dy_1 \ldots dy_n \\
    & \quad = \left( \dfrac{x-t}{x} \right)^{n\alpha } \mathbb{P} [S_n \leq x , X_i \leq \varepsilon_n, \forall i \in \{1, \ldots , n \} ].
\end{align*}
We deduce that Inequalities \eqref{eq:lestim_equivalence1} and \eqref{eq:uestim_equivalence1} hold:
\begin{align*}
    & \dfrac{\mathbb{P} [S_{n+1} \leq x, X_i \leq \varepsilon_{n}, \forall i \in \{1, \ldots , n +1 \} ]}{\mathbb{P} [S_n \leq x]} = \int_0^{\varepsilon_n} f(t) \dfrac{\mathbb{P} [S_n \leq x -t, X_i \leq \varepsilon_n, \forall i \in \{1, \ldots , n \} ]}{\mathbb{P} [S_n \leq x]} dt \\
    & \qquad \leq \int_0^{\varepsilon_n} f(t) \left( \dfrac{x-t}{x} \right)^{n\alpha } dt,
\end{align*}
and 
\begin{align*}
    & \dfrac{\mathbb{P} [S_{n+1} \leq x, X_i \leq \varepsilon_{n}, \forall i \in \{1, \ldots , n +1 \} ]}{\mathbb{P} [S_n \leq x, X_i \leq \varepsilon_n, \forall i \in \{1, \ldots , n \} ]} = \int_0^{\varepsilon_n} f(t) \dfrac{\mathbb{P} [S_n \leq x -t, X_i \leq \varepsilon_n, \forall i \in \{1, \ldots , n \} ]}{\mathbb{P} [S_n \leq x, X_i \leq \varepsilon_n, \forall i \in \{1, \ldots , n \} ]} dt  \\
    & \qquad \geq \int_0^{\varepsilon_n} f(t) \left( \dfrac{x-t}{x} \right)^{n\alpha } dt.
\end{align*}

\smallskip
\noindent \underline{\textit{Case 2.2.}} If $f$ is not equal to the power function $Bt^{\alpha-1}$, then, for any $n \geq N_1$ we have $\beta_n^+ < \alpha < \beta_n^-$. Define $N_2 = \min_{n \geq N_1} \{ \beta_k^+ \geq \alpha/2, \forall k \geq n  \}$. The integer $N_2$ is defined to ensure the positiveness of the sequences of exponent $(\beta_n^+)_{n \geq N_2}$ and $(\beta_n^-)_{n \geq N_2}$. Then Proposition \ref{prop:slowly_varying} and Corollary \ref{coro:: regular_var} justify that the functions $t \mapsto t^{1 -\beta_n^+}f(t) $ and $t \mapsto t^{1 -\beta_n^-}f(t) $ are respectively non-decreasing and non-increasing on $\left] 0,\dfrac{x}{x - \varepsilon_n} \varepsilon_n \right]$, for any $n \geq N_2$. From these analytical properties, we deduce that
\begin{align} \label{ineq:: regular_var}
    \forall t \in \left] 0,\dfrac{x}{x - \varepsilon_n} \varepsilon_n \right], \forall c \in ]0,1], \quad \quad  c^{\beta_n^- - 1} f(t) \leq f(ct) \leq c^{ \beta_n^+ - 1} f(t) .
\end{align}
Applying these inequalities and using the same arguments as in the case where $\beta_n^+ = \alpha = \beta_n^-$, we obtain for $n \geq N_2$ and $t \in (0,\varepsilon_n)$
$$
    \mathbb{P} [S_n \leq x - t, X_i \leq \varepsilon_n,  \forall i \in \{1, \ldots , n \} ] \geq \left( \dfrac{x-t}{x} \right)^{n  \beta_n^-} \mathbb{P} [S_n \leq x , X_i \leq \varepsilon_n, \forall i \in \{1, \ldots , n \} ]
$$
and
$$
    \mathbb{P} [S_n \leq x - t, X_i \leq \varepsilon_n,  \forall i \in \{1, \ldots , n \} ] \leq \left( \dfrac{x-t}{x} \right)^{n  \beta_n^+} \mathbb{P} [S_n \leq x  ].
$$
We then deduce the inequalities \eqref{eq:lestim_equivalence1} and \eqref{eq:uestim_equivalence1}. 

\medskip
\noindent \underline{\textbf{Step 3:}} We now prove \eqref{eq:equivalence2}. 
The result is obvious in the case 2.1. Thus we suppose here that the sequences $\beta_n^+$ and $\beta_n^-$ are never equal to $\alpha$. We consider the sequence of functions $(g_n)_{n \geq N_2}$, defined as
\begin{align*}
    g_n : \left[ \beta_{N_2}^+ , \beta_{N_2}^- \right] & \to \mathbb{R} \\
    \beta & \mapsto \int_0^{\varepsilon_{n}} f(t) \left( \dfrac{x-t}{x} \right)^{n \beta} dt.
\end{align*}
For any $n \geq N_2$, the function $g_n$ is a differentiable function with continuous derivative on the subset $\left] \beta_{N_2}^+ , \beta_{N_2}^- \right[$ with
\begin{align*}
    g_n'(\beta) = \int_0^{\varepsilon_{n}} f(t) n \log \left( \frac{x - t}{x} \right) \left( \dfrac{x-t}{x} \right)^{n \beta} dt \leq 0.
\end{align*}
Moreover for any $\beta \in \left] \beta_{N_2}^+ , \beta_{N_2}^- \right[$, we have
\begin{align*}
    \left| g_n'(\beta) \right| & = \int_0^{\varepsilon_{n}} f(t) n \log \left( \frac{x}{x - t} \right) \left( \dfrac{x-t}{x} \right)^{n \beta} dt  = \int_0^{\varepsilon_{n}} f(t) n \log \left( 1 + \frac{ t}{x-t} \right) \left( \dfrac{x-t}{x} \right)^{n \beta} dt \\
    & \leq \int_0^{\varepsilon_{n}} f(t) n  \frac{ t}{x-t}  \left( \dfrac{x-t}{x} \right)^{n \beta} dt  \leq \frac{ 1}{x-\varepsilon_{n}} \int_0^{\varepsilon_{n}} t f(t) n    \left( \dfrac{x-t}{x} \right)^{n \beta} dt \\
    & \leq \frac{ 2}{x} \int_0^{\varepsilon_{n}} t f(t) n    \left( \dfrac{x-t}{x} \right)^{n \beta} dt.
\end{align*}
The last inequality can be deduced since $n \geq N_2 \geq N_1$, where $N_1$ is defined to ensure that $\varepsilon_{n} \leq x/2$, for any $n \geq N_1$. 
Moreover, we recalled before Equivalence (\ref{eq:mono_dens_equiv}) that
\begin{equation*} 
    \mathbb{P} [X_1 \leq t] \underset{t \xrightarrow{} 0^+}{\sim} \frac{1}{\alpha} t f(t).
\end{equation*}
We consequently deduce the existence of a constant $M >1$ such that for any $t \in ]0, \varepsilon_{N_2}]$, $$tf(t) \leq M \alpha \mathbb{P} [X_1 \leq t]. $$ 
Therefore for any $\beta \in \left] \beta_{N_2}^+ , \beta_{N_2}^- \right[$
\begin{align*}
    \left| g_n'(\beta) \right| & \leq \frac{ 2}{x} \int_0^{\varepsilon_{n}} t f(t) n    \left( \dfrac{x-t}{x} \right)^{n \beta} dt  \leq \frac{ 2 M \alpha }{x} \int_0^{\varepsilon_{n}} \mathbb{P} [X_1 \leq t] n    \left( \dfrac{x-t}{x} \right)^{n \beta} dt \\
    & = -  2 M \alpha \frac{n}{n\beta +1} \left[ \mathbb{P} [X_1 \leq t] \left( \dfrac{x-t}{x} \right)^{n \beta +1} \right]_0^{\varepsilon_{n}} + 2 M \alpha \frac{n}{n\beta +1} \int_0^{\varepsilon_{n}} f(t)    \left( \dfrac{x-t}{x} \right)^{n \beta +1} dt \\
    & \leq 2 M \alpha \frac{n}{n\beta +1} \int_0^{\varepsilon_{n}} f(t)    \left( \dfrac{x-t}{x} \right)^{n \beta +1} dt  \leq 2 M \alpha \frac{1}{\beta } \int_0^{\varepsilon_{n}} f(t)    \left( \dfrac{x-t}{x} \right)^{n \beta } dt \\& 
    = \frac{2 M \alpha}{\beta } g_n(\beta) \leq \frac{2 M \alpha}{\beta_{N_2}^+ } g_n(\beta)
\end{align*}
Using the mean value theorem, we deduce that for $n > N_2$:
\begin{align*}
    \left| g_n(\beta_n^-)  - g_n(\alpha) \right| & \leq \left| \beta_n^-  - \alpha \right| \sup_{\beta \in \left[ \beta_n^-  , \alpha \right]} \left| g_n'(\beta) \right| \leq \left| \beta_n^-  - \alpha \right| \frac{2 M \alpha}{\beta_{N_2}^+ } \sup_{\beta \in \left[ \beta_n^-  , \alpha \right]} \left| g_n(\beta) \right|\\
    \left| g_n(\alpha)  - g_n(\beta_n^+) \right| & \leq \left|   \alpha - \beta_n^+ \right| \sup_{\beta \in \left[   \alpha ,\beta_n^+ \right]} \left| g_n'(\beta) \right|\leq \left|   \alpha - \beta_n^+ \right|  \frac{2 M \alpha}{\beta_{N_2}^+ } \sup_{\beta \in \left[   \alpha ,\beta_n^+ \right]} \left| g_n(\beta) \right|.
\end{align*}
For any $n \geq N_2$, the function $g_n$ is non-increasing, which implies that 
\begin{align*}
    \left| \frac{g_n(\beta_n^-)}{g_n(\alpha)}  - 1 \right| \leq \frac{2 M \alpha}{\beta_{N_2}^+ } \left| \beta_n^-  - \alpha \right| \\
    \left| \frac{g_n(\alpha)}{g_n(\beta_n^+)}  - 1  \right| \leq \frac{2 M \alpha}{\beta_{N_2}^+ } \left|   \alpha - \beta_n^+ \right|.
\end{align*}
Since $(\beta_n^-)_{n \geq 2}$ and $(\beta_n^+)_{n \geq 2}$ are two sequences that converge to $\alpha$, we obtain \eqref{eq:equivalence2}. 

\medskip 
\noindent \underline{ \textbf{Step 4:}}
Applying Lemma \ref{lem: equivalence_inequ}, the two inequalities \eqref{eq:lestim_equivalence1} and \eqref{eq:uestim_equivalence1} obtained in {\bf Step 2}, the previous equivalence \eqref{eq:equivalence2}, combined to the results of \textbf{Step 1} justify the wanted asymptotic equivalence:
\begin{align*}
    \dfrac{\mathbb{P} [S_{n+1} \leq x ]}{\mathbb{P} [S_n \leq x ]}   \underset{n \to + \infty}{\sim} \int_0^{\varepsilon_n} f(t) \left( \dfrac{x-t}{x} \right)^{n\alpha } dt.
\end{align*}
Consider a function $l$ that diverges to the infinity at infinity and which verifies $0\leq \frac{l(n)}{n} \leq \varepsilon_n$, at least for $n$ large. It is the case if $l(t) \leq 2x \log(t)$ or if $l(t) \underset{t\to \infty}{=} o(\log(t))$. 
We have
\begin{align*}
    \int_0^{\varepsilon_n} f(t) \left( \dfrac{x-t}{x} \right)^{n \beta} dt =  \int_0^{\frac{l(n)}{n}} f(t) \left( \dfrac{x-t}{x} \right)^{n \alpha} dt + \int_{\frac{l(n)}{n}}^{\varepsilon_n} f(t) \left( \dfrac{x-t}{x} \right)^{n \alpha} dt.
\end{align*}
We prove that 
\begin{align*}
    \int_{\frac{l(n)}{n}}^{\varepsilon_n} f(t) \left( \dfrac{x-t}{x} \right)^{n \alpha} dt \underset{n \to + \infty}{ = } o \left( \int_0^{\varepsilon_n} f(t) \left( \dfrac{x-t}{x} \right)^{n \alpha} dt \right).
\end{align*}
To justify such asymptotic comparison, we rather prove that
\begin{align*}
    \int_{\frac{l(n)}{n}}^{\varepsilon_n} f(t) \left( \dfrac{x-t}{x} \right)^{n \alpha} dt \underset{n \to + \infty}{ = } o \left( \mathbb{P} \left[ X_1 \leq \frac{1}{n} \right] \right),
\end{align*}
since Theorems \ref{theo:: class c_5} and \ref{theo::lower_bound_continuous} justify that 
\begin{align*}
    \mathbb{P} \left[ X_1 \leq \frac{1}{n} \right] \underset{n \to + \infty}{ = } O \left( \dfrac{\mathbb{P} [ S_{n+1} \leq x] }{\mathbb{P} [ S_{n} \leq x]} \right)  \underset{n \to + \infty}{ = } O \left( \int_0^{\varepsilon_n} f(t) \left( \dfrac{x-t}{x} \right)^{n \alpha} dt \right).
\end{align*}
We shall consider the proof separately in the case where the density function admits a finite or infinite limit at zero, following the same distinction as in the proof of Theorem \ref{theo:: class c_5} (Step 3). 

First assume that \underline{$\lim_{t \to 0^+} f(t) = + \infty$}. It has been justified in the proof of Theorem \ref{theo:: class c_5} (Step 3) that in such configuration, there exists $\eta > 0$ such that $f$ is non-increasing on the subset $]0,\eta[$. Hence for $n$ sufficiently large such that $l(n)/n \leq \varepsilon_n \leq \eta$, $ ||f||_{\infty , [l(n)/n;\varepsilon_n]} \leq f(l(n)/n)$. We consequently deduce 
\begin{align*}
    & \int_{\frac{l(n)}{n}}^{\varepsilon_n} f(t) \left( \dfrac{x-t}{x} \right)^{n \alpha} dt  \leq ||f||_{\infty , [l(n)/n;\varepsilon_n]} \int_{\frac{l(n)}{n}}^{\varepsilon_n} \left( \dfrac{x-t}{x} \right)^{n \alpha} dt  \leq 
    f \left( \frac{l(n)}{n} \right) 
    \int_{\frac{l(n)}{n}}^{\varepsilon_n} \left( \dfrac{x-t}{x} \right)^{n \alpha} dt \\
    & \qquad \leq 
    f \left( \frac{l(n)}{n} \right) 
    \frac{1}{n \alpha + 1} \left( 1 - \frac{l(n)}{n} \right)^{n \alpha + 1}  \leq 
    f \left( \frac{l(n)}{n} \right) 
    \frac{1}{n \alpha + 1} \exp \left( - \frac{\alpha l(n)}{x} \right).
\end{align*}
Again since $f$ is non-increasing on $]0,\eta[$ and since $l(n) \geq 1$ (at least to $n$ large enough), $f(l(n)/n)\leq f(1/n)$ for large values of $n$,  and thus
\begin{align*}
    \int_{\frac{l(n)}{n}}^{\varepsilon_n} f(t) \left( \dfrac{x-t}{x} \right)^{n \alpha} dt & \leq f \left( \frac{1}{n} \right)  \frac{1}{n \alpha + 1} 
    \exp \left( - \frac{\alpha l(n)}{x} \right).
\end{align*}
Recall that the asymptotic equivalence (\ref{eq:mono_dens_equiv}) justifies that
$$
    f \left( \frac{1}{n} \right)  \frac{1}{n \alpha + 1} \underset{n \to + \infty}{=} O \left( \mathbb{P} \left[ X_1 \leq \frac{1}{n} \right] \right).
$$
Since $\lim_{n\to +\infty} l(n) = +\infty$, we deduce that 
$$
    \int_{\frac{l(n)}{n}}^{\varepsilon_n} f(t) \left( \dfrac{x-t}{x} \right)^{n \alpha} dt  \underset{n \to + \infty}{ = } o \left( \mathbb{P} \left[ X_1 \leq \frac{1}{n} \right] \right)
$$
and therefore 
$$
    \int_0^{\frac{l(n)}{n}} f(t) \left( \dfrac{x-t}{x} \right)^{n \alpha} dt   \underset{n \to + \infty}{\sim} \int_0^{\varepsilon_n} f(t) \left( \dfrac{x-t}{x} \right)^{n \alpha} dt.  
$$

Now assume that \underline{$f$ admits a finite positive limit at zero}, that is $\alpha = 1$. Then $f$ is bounded on a neighborhood on zero and Equivalence (\ref{eq:mono_dens_equiv}) implies that 
$$
      \frac{1}{n \alpha + 1} \underset{n \to + \infty}{=} O \left( \mathbb{P} \left[ X_1 \leq \frac{1}{n} \right] \right).
$$
We argue as in the previous case to deduce the wanted result.

Finally suppose that \underline{the limit of $f$ at zero is zero}. To prove the wanted asymptotic comparison in such configuration, we propose to use similar arguments as in the proof of Theorem \ref{theo:: class c_5}. 

Consider $\delta$ a positive real number such that $\delta > \alpha -1$. Since $f'$ is supposed to be ultimately monotone in a neighborhood of $0$ and due to Property \eqref{rem:limit_beha_slowly_var} and Proposition \ref{prop:slowly_varying}, we deduce that $t \mapsto  f(t) t^{- \delta}$ diverges to the infinity at zero and is a non-increasing function in a certain neighborhood of the origin. Hence for $n$ sufficiently large:
\begin{align*}
    & \int_{l(n)/n}^{\varepsilon_n}  f(t) \left( \frac{x-t}{x} \right)^{ n \alpha  } dt  =  \int_{1/n}^{\varepsilon_n}  f(t) t^{- \delta} t^{ \delta}  \left( \frac{x-t}{x} \right)^{ n \alpha  } dt \\
    & \leq  f \left( \frac{l(n)}{n} \right) \left( \frac{n}{l(n)} \right)^{\delta} \int_{l(n)/n}^{\varepsilon_n} t^{ \delta}  \left( \frac{x-t}{x} \right)^{ n \alpha  } dt \\
    &\qquad  \leq  f \left( \frac{l(n)}{n} \right) \left( \frac{n}{l(n)} \right)^{\delta} \int_{l(n)/n}^{x} t^{ \delta}  \left( \frac{x-t}{x} \right)^{ n \alpha  } dt.
\end{align*}
W.l.o.g. we can suppose that $\delta$ is not an integer. It is now possible to apply Formula (\ref{induc:: integration_formula}) by replacing the lower bound of the interval of integration $1/n$ by $l(n)/n$, to obtain that:
\begin{align*}
    &\int_{l(n)/n}^{x} t^{ \delta}  \left(  \frac{x-t}{x} \right)^{  n \alpha  } dt\\ 
    =
    & \sum_{j=0}^{\lfloor \delta \rfloor} \left( \prod_{i=0}^j (\delta - i) \left(  n \alpha + i + 1 \right)^{-1} \right) \frac{x^{j+1}}{\delta - j} \left( \frac{l(n)}{n} \right)^{\delta - j} \left( 1 - \frac{l(n)}{nx}  \right)^{ n \alpha + j + 1} \\ 
    & +  \left( \prod_{i=0}^{\lfloor \delta \rfloor} (\delta - i) \left(  n \alpha + i + 1 \right)^{-1} \right) x^{\lfloor \delta \rfloor +1} \int_{l(n)/n}^{x} t^{ \delta - \lfloor \delta \rfloor - 1} \left(  \frac{x-t}{x} \right)^{  n \alpha  + \lfloor \delta \rfloor + 1  } dt .
\end{align*}    
The first term of the right-hand side can be bounded by:
\begin{align*}
& \exp \left( - \frac{\alpha l(n)}{x} \right) l(n)^{\delta +1} \sum_{j=0}^{\lfloor \delta \rfloor} \left( \prod_{i=0}^j (\delta - i) \left(  n \alpha + i + 1 \right)^{-1} \right) \frac{x^{j+1}}{\delta - j} \left( \frac{1}{n} \right)^{\delta - j} \\
& \qquad \leq \exp \left( - \frac{\alpha l(n)}{x} \right) \left( \dfrac{l(n)}{n} \right)^{\delta+1}\sum_{j=0}^{\lfloor \delta \rfloor} \left( \prod_{i=0}^j (\delta - i) \left(  \alpha + \frac{i + 1}{n} \right)^{-1} \right) \frac{x^{j+1}}{\delta - j} .
\end{align*}
The second term is bounded from above by:
\begin{align*}
   & \left( \prod_{i=0}^{\lfloor \delta \rfloor} (\delta - i) \left(  n \alpha + i + 1 \right)^{-1} \right) x^{\lfloor \delta \rfloor +1} \left( \frac{l(n)}{n} \right)^{\delta - \lfloor \delta \rfloor - 1} \int_{l(n)/n}^{x} \left(  \frac{x-t}{x} \right)^{  n \alpha  + \lfloor \delta \rfloor + 1  } dt \\
    & \quad = l(n)^{\delta - \lfloor \delta \rfloor - 1} \left( \prod_{i=0}^{\lfloor \delta \rfloor} (\delta - i) \left(  n \alpha + i + 1 \right)^{-1} \right) \left( \frac{1}{n} \right)^{\delta - \lfloor \delta \rfloor - 1} \frac{x^{\lfloor \delta \rfloor +2}}{n \alpha  + \lfloor \delta \rfloor + 2} \left( 1 - \frac{l(n)}{nx}  \right)^{ n \alpha  + \lfloor \delta \rfloor + 2} \\
    & \quad \leq \exp \left( - \frac{\alpha l(n)}{x} \right) l(n)^{\delta +1 } \left( \prod_{i=0}^{\lfloor \delta \rfloor} (\delta - i) \left(  \alpha + \frac{i + 1}{n} \right)^{-1} \right) \left( \frac{1}{n} \right)^{\delta} \frac{x^{\lfloor \delta \rfloor +2}}{n \alpha  + \lfloor \delta \rfloor + 2}. 
 \end{align*}  
Hence we deduce that 
$$
  \int_{l(n)/n}^{x} t^{ \delta}  \left(  \frac{x-t}{x} \right)^{  n \alpha  } dt \underset{n \to + \infty}{=}  O \left(  \exp \left( - \frac{\alpha l(n)}{x} \right) \left( \frac{l(n)}{n}  \right)^{\delta +1} \right),
$$
thus
\begin{align*}
    \int_{\frac{l(n)}{n}}^{\varepsilon_n} f(t) \left( \dfrac{x-t}{x} \right)^{n \alpha} dt & \underset{n \to + \infty}{=}  O \left(  \exp \left( - \frac{\alpha l(n)}{x} \right) f \left( \frac{l(n)}{n}  \right) \left( \frac{l(n)}{n}  \right) \right).
\end{align*}
Since $t \mapsto  f(t) t^{- \delta}$ is a non-increasing function in a certain neighborhood of the origin, for $n$ sufficiently large, we have
$$
    \left( \frac{l(n)}{n}  \right)^{- \delta}  f \left( \frac{l(n)}{n}  \right) \leq \left( \frac{1}{n}  \right)^{- \delta}  f \left( \frac{1}{n}  \right) ,
$$
therefore 
$$ \frac{l(n)}{n} f \left( \frac{l(n)}{n}  \right)  \leq  l(n)^{ \delta +1} \left( \frac{1}{n}  \right) f  \left( \frac{1}{n}  \right).
$$
This ends the proof of the asymptotic comparison since
\begin{align*}
    \int_{\frac{l(n)}{n}}^{\varepsilon_n} f(t) \left( \dfrac{x-t}{x} \right)^{n \alpha} dt & \underset{n \to + \infty}{=}  O \left(  \exp \left( - \frac{\alpha l(n)}{x} \right) l(n)^{\delta +1} f \left( \frac{1}{n}  \right) \left( \frac{1}{n}  \right) \right) \\
    & \underset{n \to + \infty}{=}  O \left(  \exp \left( - \frac{\alpha l(n)}{x} \right) l(n)^{\delta +1} \mathbb{P} \left[ X_1 \leq \frac{1}{n} \right] \right) \\
    & \underset{n \to + \infty}{=}  o \left( \mathbb{P} \left[ X_1 \leq \frac{1}{n} \right] \right).
\end{align*}

\end{proof}

\subsection{Proofs for Section \ref{sec:: LDT}}

Let's start with the proof of Theorem \ref{ratio_lim_ldp}.

\begin{proof}
To prove this theorem, we reuse the arguments of proof of Cramér's article \cite{Cra38}. We re-introduce the notations from its article. For $h \in ]-a_2, a_1 [$, let $F_n$ and $\overline{F}_{n,h}$ be respectively the distribution functions of the variables
\begin{align*}
    \frac{Y_1 + \ldots + Y_n }{\sigma \sqrt{n} } \quad \quad \text{and} \quad \quad \frac{\overline{Z}_1^h + \ldots + \overline{Z}_n^h - \overline{m}(h) n }{\overline{\sigma}(h) \sqrt{n} }.
\end{align*}
To begin with, $\mathbb{P} \left[ S_n \leq 0  \right] = \mathbb{P} \left[ \frac{Y_1 + \ldots + Y_n }{\sigma \sqrt{n} } \geq \frac{\sqrt{n} \mathbb{E} [X_1]}{ \sigma}  \right] = 1 - F_n \left( \frac{\sqrt{n} \mathbb{E} [X_1]}{ \sigma}  \right)$ and $\mathbb{P} \left[ S_n \leq x  \right] = \mathbb{P} \left[ \frac{Y_1 + \ldots + Y_n }{\sigma \sqrt{n} } \geq \frac{\sqrt{n} \mathbb{E} [X_1]}{ \sigma} - \frac{x}{\sigma \sqrt{n}} \right] = 1 - F_n \left( \frac{\sqrt{n} \mathbb{E} [X_1]}{ \sigma}  - \frac{x}{\sigma \sqrt{n}} \right)$. From \cite[Formula (12b)]{Cra22}, we deduce that
\begin{align*}
    1 - F_n (t) = e^{-(h\overline{m}(h) - \log (R(h)))n} \int_{\frac{\sigma t - \overline{m}(h) \sqrt{n} }{\overline{\sigma} (h) }}^{+ \infty} e^{-h \overline{\sigma} (h) \sqrt{n} y} d\overline{F}_{n,h} (y).
\end{align*}
When considering $t = \frac{\overline{m}(h) \sqrt{n} }{\sigma}$, with $h \in ]-a_2, a_1 [$, we obtain \cite[Formula (21)]{Cra22}, i.e.
\begin{align*}
    1 - F_n \left( \frac{\overline{m}(h) \sqrt{n} }{\sigma} \right) = e^{-(h\overline{m}(h) - \log (R(h)))n} \int_{0}^{+ \infty} e^{-h \overline{\sigma} (h) \sqrt{n} y} d\overline{F}_{n,h} (y),
\end{align*}
where $\overline{F}_{n,h} $ can be decomposed as $\overline{F}_{n,h} = \Phi + Q_{n,h}$, with $\Phi$ the distribution function of the standard normal distribution and 
\begin{align*}
    Q_{n,h} (y) = \frac{P_{2,h} (y) }{\sqrt{n}} e^{- \frac{y^2}{2}} + M_{n,h} (y).
\end{align*}
As explained in \cite[pages 10 and 16]{Cra22} and \cite[Chapter 7]{Cra70}, $P_{2,h}$ is a polynomial function of degree 2, with his coefficients expressible as continuous functions in $h$ of the moments of the variable $Z_{1,h}$ and  $\sup_{y \in \mathbb{R}} | M_{n,h} (y) | \leq \frac{g(h)}{n}$, where g is also expressible as continuous functions in $h$ of the moments of the variable $Z_{1,h}$. Let us emphasize that the second property of Assumption \ref{ass:cramer_theo} is used here. To obtain more details on this assertions, one can find them in \cite[Chapter 7]{Cra70}. The continuity in $h$ of the moment functions of the variables $(Z_{1,h})_{h \in ]-a_2,a_1[}$ is discussed in  \cite[page 9]{Cra22}.

In order to obtain asymptotic equivalents of the sequences $\left( 1 - F_n \left( \frac{\sqrt{n} \mathbb{E} [X_1]}{ \sigma}  \right) \right)_{n \in \mathbb{N}^*}$ and \\
$\left( 1 - F_n \left( \frac{\sqrt{n} \mathbb{E} [X_1]}{ \sigma} - \frac{x}{\sigma \sqrt{n}} \right) \right)_{n \in \mathbb{N}^*}$, we define $h_\infty$ and $h_n$, for $n > x/\mathbb{E} [X_1]$, the positive real numbers that uniquely solve the equations 
\begin{align*}
    1) & \quad \overline{m}(h_\infty) = \mathbb{E} [X_1] \\
    2) & \quad \frac{\sqrt{n} \mathbb{E} [X_1]}{ \sigma} - \frac{x}{\sigma \sqrt{n}} = \frac{\overline{m}(h_n) \sqrt{n}  }{\sigma} \quad \iff \quad  \overline{m}(h_n) = \mathbb{E} [X_1] - \frac{x}{n}. 
\end{align*}
Existence of these numbers is ensured by the last hypothesis of Assumption \ref{ass:cramer_theo} and since $\overline{m}$ is a continuous function. Moreover on $]-a_2,a_1[$, the derivative of $\overline{m}$ exists and is equal to $\overline{\sigma}^2 (h) > 0$ (from our condition on $X_1$, $\overline{Z}_1^h$ cannot be degenerate). Hence $\overline{m}$ is an increasing function. Since $\overline{m}(0) = 0$ and $\mathbb{E} [X_1] >0$, we deduce that $h_\infty$ and $h_n$ are unique and positive. The continuity of $\overline{m}$ also implies that $\displaystyle \lim_{n \to + \infty} h_n = h_\infty$. Finally since $\overline{m}$ is of class $C^1$ with positive derivative, we have: $(h_n - h_\infty)  \underset{n \to + \infty}{=} O \left( \dfrac{1}{n} \right)$.

Remark that the set $\displaystyle H = \{ h_\infty \} \cup \left( \bigcup_{n > x/\mathbb E(X_1)} \{ h_n \} \right)$ is a compact subset and any moment function of the variables $(Z_{1,h})_{h \in ]-a_2,a_1[}$ is continuous, it is also the case for their respective image of $H$. Consequently, there consequently exists a positive constant $M$, independent of $h$, such that for any $h \in H$, $\sup_{y \in \mathbb{R}} | M_{n,h}(y) | \leq \frac{C M}{ n}$. 

In order to simplify the notations of the forthcoming formulas, we define
\begin{align*}
    \alpha = & h_\infty \overline{m} ( h_\infty) - \log ( R (h_\infty ) ) \\
    \alpha_n = & h_n \overline{m} ( h_n) - \log ( R (h_n ) ). 
\end{align*}
We have
\begin{align*}
    1 - F_n \left( \frac{\sqrt{n} \mathbb{E} [X_1]}{ \sigma}  \right) = e^{- \alpha n} \int_{0}^{+ \infty} e^{-h_\infty \overline{\sigma} (h_\infty) \sqrt{n} y} d\overline{F}_{n,h_\infty} (y).
\end{align*}
Finally using the decomposition of $\overline{F}_{n,h_\infty}$ and applying an integration by parts\footnote{The use of Lebesgue–Stieltjes type integration by parts is justified, since one of the functions is continuous and the other is of bounded variation.} with the term $M_{n,h_\infty}$, we have
\begin{align*}
    \int_{0}^{+ \infty} e^{-h_\infty \overline{\sigma} (h_\infty) \sqrt{n} y} d\overline{F}_{n,h} (y) = &  \int_{0}^{+ \infty}  \frac{e^{-h_\infty \overline{\sigma} (h_\infty) \sqrt{n} y - \frac{y^2}{2}}}{\sqrt{2 \pi}} dy \\
    & + \frac{1}{\sqrt{n}}\int_{0}^{+ \infty} e^{-h_\infty \overline{\sigma} (h_\infty) \sqrt{n} y - \frac{y^2}{2}} ( P_{2,h_\infty}' (y) - y P_{2,h_\infty} (y)) dy \\
    & - M_{n,h_\infty} (0) + h_\infty \overline{\sigma} (h_\infty) \sqrt{n} \int_{0}^{+ \infty} e^{-h_\infty \overline{\sigma} (h_\infty) \sqrt{n} y} M_{n,h_\infty} (y)  dy.
\end{align*}
From the proof arguments given in \cite[page 16]{Cra22}, one naturally obtains
\begin{align*}
    \int_{0}^{+ \infty}  \frac{e^{-h_\infty \overline{\sigma} (h_\infty) \sqrt{n} y - \frac{y^2}{2}}}{\sqrt{2 \pi}} dy \underset{n \to + \infty}{ \sim} \frac{1}{h_\infty \overline{\sigma} (h_\infty) \sqrt{2 \pi n}}.
\end{align*}
Finally, since $\sup_{y \in \mathbb{R}} | M_{n,h_\infty} (y) | \leq \frac{g(h_\infty)}{n}$ and $y \mapsto e^{- \frac{y^2}{2}} ( P_{2,h_\infty}' (y) - y P_{2,h_\infty} (y))$ is a bounded function on $\mathbb{R}$,
\begin{align*}
    \frac{1}{\sqrt{n}}\int_{0}^{+ \infty} & e^{-h_\infty \overline{\sigma} (h_\infty) \sqrt{n} y - \frac{y^2}{2}} ( P_{2,h_\infty}' (y) -  y P_{2,h_\infty} (y)) dy 
    - M_{n,h_\infty} (0) \\ 
    & +h_\infty \overline{\sigma} (h_\infty) \sqrt{n} \int_{0}^{+ \infty} e^{-h_\infty \overline{\sigma} (h_\infty) \sqrt{n} y} M_{n,h_\infty} (y)  dy = \underset{n \to + \infty}{O} \left( \frac{1}{n} \right).
\end{align*}
This justifies that
\begin{align*}
    \mathbb{P} \left[ S_n \leq 0  \right] \underset{n \to + \infty}{\sim} \frac{e^{-\alpha n}}{h_\infty \overline{\sigma}(h_\infty ) \sqrt{2 \pi n}}.
\end{align*}
Similarly, we have
\begin{align*}
    \left( 1 - F_n \left( \frac{\sqrt{n} \mathbb{E} [X_1]}{ \sigma} - \frac{x}{\sigma \sqrt{n}} \right) \right)_{n \in \mathbb{N}^*} = e^{- \alpha_n n} \int_{0}^{+ \infty} e^{-h_n \overline{\sigma} (h_n) \sqrt{n} y} d\overline{F}_{n,h_n} (y),
\end{align*}
and
\begin{align*}
    \int_{0}^{+ \infty} e^{-h_n \overline{\sigma} (h_n) \sqrt{n} y} d\overline{F}_{n,h_n} (y) = &  \int_{0}^{+ \infty}  \frac{e^{-h_n \overline{\sigma} (h_n) \sqrt{n} y - \frac{y^2}{2}}}{\sqrt{2 \pi}} dy \\
    & + \frac{1}{\sqrt{n}}\int_{0}^{+ \infty} e^{-h_n \overline{\sigma} (h_n) \sqrt{n} y - \frac{y^2}{2}} ( P_{2,h_n}' (y) - y P_{2,h_n} (y)) dy \\
    & - M_{n,h_n} (0) + h_n \overline{\sigma} (h_n)  \sqrt{n} \int_{0}^{+ \infty} e^{-h_n \overline{\sigma} (h_n) \sqrt{n} y} M_{n,h_n} (y)  dy.
\end{align*}
The first term is treated the same way as before and we have
\begin{align*}
    \int_{0}^{+ \infty}  \frac{e^{-h_n \overline{\sigma} (h_n) \sqrt{n} y - \frac{y^2}{2}}}{\sqrt{2 \pi}} dy \underset{n \to + \infty}{ \sim} \frac{1}{h_n \overline{\sigma} (h_n) \sqrt{2 \pi n}} \underset{n \to + \infty}{ \sim} \frac{1}{h_\infty \overline{\sigma} (h_\infty) \sqrt{2 \pi n}},
\end{align*}
where the last equivalent is deduced from the continuity of $\overline{\sigma}$ and consequently, $h_n \overline{\sigma} (h_n) \to h_\infty \overline{\sigma} (h_\infty) > 0$. \\
To study the other terms, recall that the coefficients of the polynomial $P_{2,h}$ are continuous functions in $h$ and $H$ is a compact subset. Due to these properties, it is possible to uniformly bound the polynomial $y \to | P_{2,h_n}' (y) - y P_{2,h_n} (y) |$ by another polynomial of degree 3, independently of $h_n \in H$. The previously provided inequality  $\sup_{y \in \mathbb{R}} | M_{n,h_n}(y) | \leq \frac{C M}{ n}$ is also independent of $h_n \in H$. Moreover since the image of $H$ by the function $h \mapsto h \overline{\sigma} (h)$ is a compact subset contained in $]0, + \infty[$, it admits a positive minimum in $H$. Finally the same arguments justify that 
\begin{align*}
    \frac{1}{\sqrt{n}}\int_{0}^{+ \infty}  & e^{-h_n \overline{\sigma} (h_n) \sqrt{n} y - \frac{y^2}{2}} ( P_{2,h_\infty}' (y) -  y P_{2,h_n} (y)) dy 
    - M_{n,h_n} (0) \\ 
    & +  h_n \overline{\sigma} (h_n)  \sqrt{n} \int_{0}^{+ \infty} e^{-h_n \overline{\sigma} (h_n) \sqrt{n} y} M_{n,h_n} (y)  dy = \underset{n \to + \infty}{O} \left( \frac{1}{n} \right).
\end{align*}
Therefore 
\begin{align*}
    \mathbb{P} \left[ S_n \leq x  \right] \underset{n \to + \infty}{\sim} \frac{e^{-\alpha_n n}}{h_\infty \overline{\sigma}(h_\infty ) \sqrt{2 \pi n}} = \frac{e^{-\alpha n} e^{-(\alpha_n - \alpha)n }}{h_\infty \overline{\sigma}(h_\infty ) \sqrt{2 \pi n}}.
\end{align*}
To conclude, it is possible to prove that $\lim_{n\to +\infty} e^{-(\alpha_n - \alpha)n } = e^{h_\infty x}$. Indeed we have
\begin{align*}
    \alpha_n - \alpha & = h_n \overline{m} (h_n) - \log ( R (h_n ) ) - ( h_\infty \overline{m} (h_\infty) - \log ( R (h_\infty ) ) \\
    & = h_n \overline{m} (h_n) - h_\infty \overline{m} (h_\infty) + \log ( R (h_\infty ) ) - \log ( R (h_n ) ) \\
    & = h_n \overline{m} (h_n) - h_\infty \overline{m} (h_\infty) + (h_\infty - h_n) \overline{m} (h_n) + (h_\infty - h_n)^2 \overline{\sigma} (h_n)^2 + o((h_n - h_\infty)^2) \big) \\
    & =  h_\infty ( \overline{m} (h_n) - \overline{m} (h_\infty) ) - (h_n - h_\infty)^2 \overline{\sigma} (h_n)^2 + o((h_n - h_\infty)^2) \\
    & = - h_\infty \frac{x}{n} - (h_n - h_\infty)^2 \overline{\sigma} (h_n) + o((h_n - h_\infty)^2).
\end{align*}
Recall that $(h_n - h_\infty)  \underset{n \to + \infty}{=} O \left( \frac{1}{n} \right)$. We obtain that the limit of $ -(\alpha_n - \alpha)n $ is equal to $ h_\infty x$, when $n$ goes to $+\infty$. This achieves the proof of this theorem. 
\end{proof}

\paragraph{Proof of Lemma \ref{lem:Condition_C_true}}
\begin{proof}
Assume that $X_1$ is bounded from below by a negative constant $b$. Then for any $h<0$,
    $$R_{X_1}(h) = \int_{\mathbb R} e^{hy} dF(y) \leq \mathbb P(X_1 \geq 0) + \int_b^0 e^{hy} dF(y) \leq \mathbb P(X_1 \geq 0) + e^{bh} \mathbb P(X_1 \leq 0) <+\infty.$$
Thus $a_1 = +\infty$. Now we choose $b< 0$ such that $\int_{[b,3b/4)} dF(y) = \mathbb P(b \leq X < 3b/4) > 0$. For all $h<0$, 
$$0\leq \int_0^\infty ye^{hy} dF(y) \leq \int_0^\infty y dF(y) < +\infty.$$
Hence when $h$ tends to $-\infty$:
$$\int_0^\infty ye^{hy} dF(y)= o \left( \int_b^0 ye^{hy}dF(y) \right),\qquad \int_0^\infty e^{hy} dF(y)= o \left(  \int_b^0 e^{hy}dF(y) \right).$$
For $b/2\leq y \leq 0$, since $h<0$, $1\leq e^{hy} \leq e^{hb/2}$. Thus
$$ \int_{[b/2,0)} e^{hy}dF(y) \leq e^{hb/2} \mathbb P(b/2 \leq X_1 < 0),$$
and 
$$\int_{[b,b/2)} e^{hy}dF(y)  \geq \int_{[b,3b/4)}e^{hy}dF(y) \geq e^{3bh/4} \mathbb P(b \leq X_1 < 3b/4).$$
Hence we deduce that:
$$
\overline{m}(h)  \leq \dfrac{ \int_{[b,b/2)} ye^{hy} dF(y) + \int_{[0,+\infty)} ye^{hy} dF(y)}{R(h)} \leq \dfrac{ \frac{b}{2}\int_{[b,b/2)} e^{hy}dF(y) + \int_{[0,+\infty)} ye^{hy}dF(y) }{ \int_{[b,b/2)} e^{hy} dF(y)  + \int_{[b/2,+\infty)} e^{hy} dF(y)}.
$$
Therefore 
$$\lim_{h \downarrow -a_1} \overline{m}_{X_1}(h) \leq \dfrac{b}{2} < 0.$$

If for any $b<0$, $\mathbb{P} [X_1 \leq b] > 0$ and $a_1 = + \infty$, then we repeat the above arguments and we obtain that $\lim_{h \downarrow -a_1} \overline{m}_{X_1}(h)$ is smaller than any $b/2$. Hence this limit should be $-\infty$. 
\end{proof}

\paragraph{Proof of Lemma \ref{lem:: large_dev_tech}}
\begin{proof}
Let $C>1$.
Using the application of Chebycheﬀ’s inequality from the proof in Cramér's Theorem (see \cite[Theorem 2.2.3]{Dem09}), we have for any $\lambda \leq 0$
\begin{align*}
    \mathbb{P} \left[ S_n \leq x  \right]  = \mathbb{P} \left[ M_n \leq \frac{x}{n}  \right] & \leq e^{-n \Lambda^*\left( \frac{x}{n} \right)}.
\end{align*}
The purpose is now to determinate the value of $\Lambda^*\left( \frac{x}{n}\right)$ or at least an asymptotic equivalent. We recall that $\Lambda^*\left( \frac{x}{n} \right)  = \underset{\lambda \in \mathbb{R}}{\sup} \left\{ \lambda \frac{x}{n} - \log \mathbb{E} \left[ e^{\lambda X_1 } \right] \right\} =  \underset{\lambda \leq 0}{\sup} \left\{ \lambda \frac{x}{n} - \log \mathbb{E} \left[ e^{\lambda X_1 } \right] \right\}$, where the last equality holds due to \cite[Lemma 2.2.5]{Dem09}, since for $n$ sufficiently large: $\frac{x}{n} \leq \mathbb E (X_1)$. Define $ g_n : \lambda \in ]- \infty, 0] \mapsto \lambda \frac{x}{n} - \log \mathbb{E} \left[ e^{\lambda X_1 } \right]$. We have
\begin{align*}
    g_n (\lambda) = \lambda \frac{x}{n} - \log \left( \sum_{a \in X_1(\Omega)} e^{\lambda a } \mathbb{P} [X_1 = a] \right) = \lambda \frac{x}{n} - \log \left( \mathbb{P}[X_1=0] + \sum_{a \in X_1(\Omega) \backslash \{0\}} e^{\lambda a } \mathbb{P} [X_1 = a] \right)
\end{align*}
And since $(X_n)_{n \in \mathbb{N}^*}$ belongs to class $\mathcal{C}_2$, $x_{\min} = \inf \{ a \in  X_1(\Omega) \backslash \{0\}\} >0$. We consequently deduce that for any $\lambda \leq 0$
\begin{align*}
    g_n (\lambda) & \geq \lambda \frac{x}{n} - \log \mathbb{P}[X_1=0] - \log \left( 1 + e^{\lambda x_{\min}} \frac{\mathbb{P}[X_1>0]}{\mathbb{P}[X_1=0]} \right).
\end{align*}
Therefore 
$$ \Lambda^*\left( \frac{x}{n}\right)  \geq  \underset{\lambda \leq 0}{\sup} \left\{ \lambda \frac{x}{n} - \log \mathbb{P}[X_1=0] - \log \left( 1 + e^{\lambda x_{\min}} \frac{\mathbb{P}[X_1>0]}{\mathbb{P}[X_1=0]} \right) \right\}.$$
Instead of determinate the maximum value of the function $g_n$ which would require to know the number of the atoms and their respective values, we propose to study the function $$h_n : \lambda \in ]- \infty, 0] \mapsto \lambda \frac{x}{n} - \log \mathbb{P}[X_1=0] - \log \left( 1 + e^{\lambda x_{\min}} \frac{\mathbb{P}[X_1>0]}{\mathbb{P}[X_1=0]} \right).$$
$h_n$ is a differentiable function on $\mathbb{R}_-$, and for any $n$ such that $x_{\min}\mathbb P(X_1=0) > \dfrac{x}{n}$, $h_n$ realises its unique maximum value at $y_n<0$ given by:
$$y_n = \frac{1}{x_{\min}} \left( \log \left( \frac{x}{n} \right) - \log \left( \frac{\mathbb{P}[X_1>0]}{\mathbb{P}[X_1=0]}\left( x_{\min} - \frac{x}{n} \right)  \right) \right) \underset{n \xrightarrow{} + \infty}{\sim}  \frac{1}{x_{\min}} \log \left( \frac{x}{n} \right).$$ 
We finally obtain
\begin{align*}
    \Lambda^*\left( \frac{x}{n}\right) & \geq  \underset{\lambda \leq 0}{\sup} \left\{ \lambda \frac{x}{n} - \log \mathbb{P}[X_1=0] - \log \left( 1 + e^{\lambda x_{\min}} \frac{\mathbb{P}[X_1>0]}{\mathbb{P}[X_1=0]} \right) \right\} = h_n \left( y_n \right) \\
    & = \frac{1}{x_{\min}} \log \left( \frac{x}{n} \right) \frac{x}{n} -\frac{1}{x_{\min}}  \frac{x}{n} \log \left( \frac{\mathbb{P}[X_1>0]}{\mathbb{P}[X_1=0]}\left( x_{\min} - \frac{x}{n} \right)  \right) \\  & \quad  - \log \mathbb{P}[X_1=0]
    - \log \left( 1 +  \frac{x}{nx_{\min} - x} \right) \\
    & = - \frac{x}{x_{\min}} \frac{ \log n}{ n} - \log \mathbb{P}[X_1=0] + O\left(  \frac{1}{n} \right). 
\end{align*}
The combination of all these inequalities leads to the establishment of the final result. Since $C>1$, we have for $n$ sufficiently large 
\begin{align*}
    \frac{1}{n} \log \mathbb{P} \left[ S_n \leq x  \right] &  \leq \log \mathbb{P}[X_1=0] +  \frac{x}{x_{\min}} \frac{ \log n}{ n} + O\left(  \frac{1}{n} \right) 
    \\
    & \leq \log \mathbb{P}[X_1=0] +  \frac{Cx}{x_{\min}} \frac{ \log n}{ n}.
\end{align*}

\end{proof}

\paragraph{Proof of Lemma \ref{lem:: ratio_limit_tech}}
\begin{proof}
In such configuration, we recall that the asymptotic equivalence from Proposition \ref{theo:: class_c2} justifies that for any $0< c_x < \mathbb{P} [X_1 = 0] M_x < C_x$, there exists $N \in \mathbb{N}$ such that for all $n \geq N$,
\begin{align*}
    \mathbb{P} [X_1 = 0 ] + \frac{c_x + \mathbb{P} [X_1 = 0] M_x}{2n} \leq \frac{\mathbb{P} [ S_{n+1} \leq x ]}{\mathbb{P} [ S_n \leq x ]} \leq \mathbb{P} [X_1 = 0 ] + \frac{C_x + \mathbb{P} [X_1 = 0] M_x}{2n}.
\end{align*}
Consequently, one can deduce the existence of a sequence $(b_n)_{n \in \mathbb{N}}$ such that for all positive integer $n$
$$
    \mathbb{P} [ S_n \leq x ] = \prod_{k=1}^n ( \mathbb{P} [X_1 = 0 ] + b_k ),
$$
such that $0 < b_n \leq \mathbb{P} [X_1 \in ]0,x] ]$ and $ \lim_{n\to +\infty} n b_n = \mathbb{P} [X_1 = 0] M_x$. Consequently:
\begin{align*}
    \log \mathbb{P} [ S_n \leq x ] & = n \log \mathbb{P} [X_1 = 0 ] + \sum_{k=1}^n \log \left( 1 + \frac{b_k}{\mathbb{P} [X_1 = 0 ]} \right) \\
    & \leq n \log \mathbb{P} [X_1 = 0 ] + \sum_{k=1}^n  \frac{b_k}{\mathbb{P} [X_1 = 0 ]} \\
    & = n \log \mathbb{P} [X_1 = 0 ] + \sum_{k=1}^{N-1}  \frac{b_k}{\mathbb{P} [X_1 = 0 ]} + \sum_{k=N}^{n}  \frac{C_x + \mathbb{P} [X_1 = 0] M_x}{2k \mathbb{P} [X_1 = 0 ]}
\end{align*}
where the definition of $N$ induces that $b_k \leq C_x / k$ for any $k \geq N$. 
Since $\frac{C_x + \mathbb{P} [X_1 = 0] M_x}{2} < C_x$, we deduce that for $n$ sufficiently large:
\begin{align*}
    \frac{1}{n} \log \mathbb{P} [ S_n \leq x ] \leq \log \mathbb{P} [X_1 = 0 ] + \frac{ C_x}{\mathbb{P} [X_1 = 0 ]} \frac{\log n}{n}.
\end{align*}
Similarly, we have for any $x \geq 0$, $\log (1+x) \geq x - \frac{1}{2}x^2$. Thus
\begin{align*}
    \log \mathbb{P} [ S_n \leq x ] & = n \log \mathbb{P} [X_1 = 0 ] + \sum_{k=1}^n \log \left( 1 + \frac{b_k}{\mathbb{P} [X_1 = 0 ]} \right) \\
    & \geq n \log \mathbb{P} [X_1 = 0 ] + \sum_{k=1}^n  \frac{b_k}{\mathbb{P} [X_1 = 0 ]} - \frac{1}{2} \left( \frac{b_k}{\mathbb{P} [X_1 = 0 ]} \right)^2 \\
    & = n \log \mathbb{P} [X_1 = 0 ] + \sum_{k=1}^{N-1}  \frac{b_k}{\mathbb{P} [X_1 = 0 ]} - \frac{1}{2} \left( \frac{b_k}{\mathbb{P} [X_1 = 0 ]} \right)^2  + \sum_{k=N}^{n} \frac{b_k}{\mathbb{P} [X_1 = 0 ]} - \frac{1}{2} \left( \frac{b_k}{\mathbb{P} [X_1 = 0 ]} \right)^2 .
\end{align*}
Since $$\frac{1}{n} \sum_{k=1}^{N-1}  \frac{b_k}{\mathbb{P} [X_1 = 0 ]} - \frac{1}{2} \left( \frac{b_k}{\mathbb{P} [X_1 = 0 ]} \right)^2 \underset{n \xrightarrow{} + \infty}{=} o\left( \frac{\log n}{n} \right)$$ 
and 
$$\lim_{n \to + \infty} \sum_{k=N}^{n}  \frac{1}{2} \left( \frac{b_k}{\mathbb{P} [X_1 = 0 ]} \right)^2 \leq \frac{C_x^2 \pi^2}{12 \mathbb{P} [X_1 = 0 ]^2},$$ 
we deduce that for $n$ sufficiently large:
\begin{align*}
    \frac{1}{n} \log \mathbb{P} [ S_n \leq x ] \geq \log \mathbb{P} [X_1 = 0 ] + \frac{ c_x}{ \mathbb{P} [X_1 = 0 ]} \frac{\log n}{n},
\end{align*}
which achieves the proof of the lemma.
\end{proof}

\end{document}